\documentclass[a4paper,11pt]{amsart}

\pdfoutput=1

\usepackage[text={400pt,660pt},centering]{geometry}

\usepackage{amsthm, amssymb, amsmath, amsfonts, mathrsfs}
\usepackage{mathtools}
\usepackage{scalerel} 
\usepackage{stackrel}
\usepackage[colorlinks=true, pdfstartview=FitV, linkcolor=blue, citecolor=blue, urlcolor=blue,pagebackref=false]{hyperref}
\usepackage{esint} 
\usepackage{MnSymbol} 
\usepackage{booktabs} 

\usepackage{microtype}
\usepackage{cleveref}

\newtheorem{proposition}{Proposition}
\newtheorem{theorem}[proposition]{Theorem}
\newtheorem{lemma}[proposition]{Lemma}

\theoremstyle{remark}
\newtheorem{remark}[proposition]{Remark}

\theoremstyle{definition}

\numberwithin{equation}{section}
\numberwithin{proposition}{section}
\numberwithin{table}{section}

\renewcommand{\le}{\leqslant}
\renewcommand{\ge}{\geqslant}
\renewcommand{\leq}{\leqslant}
\renewcommand{\geq}{\geqslant}

\renewcommand{\subset}{\subseteq}
\newcommand{\mcl}{\mathcal}

\newcommand{\msc}{\mathscr}
\newcommand{\G}{\mathcal{G}}
\newcommand{\E}{\mathbb{E}}

\renewcommand{\a}{\mathbf{a}}
\newcommand{\aone}{\mathbf{a}^{\{1\}}}
\newcommand{\ab}{{\overbracket[1pt][-1pt]{\a}}}
\newcommand{\ahom}{{\overbracket[1pt][-1pt]{\a}}}

\newcommand{\Ll}{\left}
\newcommand{\Rr}{\right}
\newcommand{\lhs}{left-hand side}
\newcommand{\rhs}{right-hand side}
\newcommand{\1}{\mathbf{1}}
\newcommand{\R}{\mathbb{R}}

\newcommand{\N}{\mathbb{N}}

\renewcommand{\P}{\mathbb{P}}

\renewcommand{\tilde}{\widetilde}
\newcommand{\al}{\alpha}
\newcommand{\de}{\delta}

\renewcommand{\epsilon}{\varepsilon}
\renewcommand{\d}{{\mathrm{d}}}

\newcommand{\cu}{{\scaleobj{1.2}{\square}}}
\newcommand{\vb}{\ \big \vert \ }

\renewcommand{\fint}{\strokedint}
\newcommand{\Rd}{{\mathbb{R}^d}}

\newcommand{\mmd}{\mathcal{M}_\delta}
\DeclareMathOperator{\dist}{dist}
\DeclareMathOperator{\supp}{supp}

\newcommand{\cD}{\msc{D}}   
\newcommand{\cC}{\msc{C}}   
\newcommand{\cL}{\msc{L}}   
\newcommand{\cH}{\msc{H}}   

\newcommand{\Ind}[1]{\mathbf{1}_{\left\{#1\right\}}}
\newcommand{\id}{\mathsf{Id}}
\newcommand{\norm}[1]{\left\Vert{#1}\right\Vert}

\newcommand{\dbint}[1]{[\![#1]\!]}

\newcommand{\mres}{\mathbin{\vrule height 1.4ex depth 0pt width
0.13ex\vrule height 0.13ex depth 0pt width 1ex}}

\title[Smoothness of the diffusion coefficients]{Smoothness of the diffusion coefficients for particle systems in continuous space}
\author[A.\ Giunti, C.\ Gu, 
J.-C.\ Mourrat, M.\ Nitzschner]{Arianna Giunti, Chenlin Gu, \\
Jean-Christophe Mourrat, Maximilian Nitzschner}

\address[Arianna Giunti]{Department of Mathematics, Imperial College London, London, United Kingdom}

\address[Chenlin Gu]{DMA, Ecole normale sup\'erieure, PSL University, Paris, France; Mathematics Department, NYU Shanghai \& NYU-ECNU Institute of Mathematical Sciences,  China}

\address[Jean-Christophe Mourrat]{Ecole Normale Sup\'erieure de Lyon and CNRS, Lyon, France; Courant Institute of Mathematical Sciences, New York University, New York NY, USA}

\address[Maximilian Nitzschner]{Courant Institute of Mathematical Sciences, New York University, New York NY, USA}

\begin{document}

\begin{abstract}

For a class of particle systems in continuous space with local interactions, we show that the asymptotic diffusion matrix is an infinitely differentiable function of the density of particles. Our method allows us to identify relatively explicit descriptions of the derivatives of the diffusion matrix in terms of correctors. 

\bigskip

\noindent \textsc{MSC 2010:} 82C22, 35B27, 60K35.

\medskip

\noindent \textsc{Keywords:} interacting particle system, hydrodynamic limit, bulk diffusion matrix.

\end{abstract}
\maketitle


%
%
%
%
%
%
%
%
\section{Introduction}

We study a class of interacting particle systems with local interactions in continuous space. These are systems of interacting Brownian particles, where each particle evolves in $\R^d$ according to a diffusion matrix that depends on the locations of the particles nearby. The models we study are reversible with respect to the Poisson measures with constant density, uniformly elliptic, and of non-gradient type. In the limit of large scales, the evolution of the empirical distribution of particles is expected to be captured by a nonlinear diffusion equation. Such a result is usually called a \emph{hydrodynamic limit}, and the matrix appearing in this diffusion equation is often called the \emph{bulk diffusion matrix}. The qualifier ``bulk'' underlines that this matrix is meant to describe the collective evolution of the cloud of particles. We refer to the monographs \cite{kipnis1998scaling, komorowski2012fluctuations, spohn2012large} for thorough expositions on the topic. 

The purpose of this work is to show that the bulk diffusion matrix is an infinitely differentiable function of the density of particles. That the bulk diffusion matrix is sufficiently regular as a function of the density of particles is a necessary ingredient in the proof of the hydrodynamic limit of the model, see for instance \cite{fuy}. We hope that, when combined with \cite{bulk}, the present work will allow for the establishment of a quantitative version of the statement of hydrodynamic limit.

Similar results on the smoothness of the effective diffusion matrix have already been derived for a number of other models of particle systems \cite{beltran, bernardin,lov-reg,lov-reg2,naga1,naga2,naga3, sued}. Those works all rely on the approach introduced in \cite{lov-reg} to show the regularity of the self-diffusion matrix of a tagged particle in the symmetric simple exclusion process on $\mathbb{Z}^d$. This approach relies on certain duality properties of the process under consideration. 

The approach we employ here seems different and more direct. In particular, we end up with relatively explicit expressions for the derivatives of the bulk diffusion matrix in terms of correctors, which are natural objects that already appear in the description of the bulk diffusion matrix itself. 

Our method takes inspiration from works on the homogenization of elliptic equations with random coefficients \cite{ AL2, AL1, dg1, dl-diff}; see also \cite{almog1, almog2, almog3, BM, kozlov}. One can for instance consider a setting in which the random coefficients of the equation are a local function of a Poisson point process with constant density, and ask whether the homogenized matrix depends smoothly on the density of the point process. This question has been answered positively in \cite[Theorem~5.A.1]{Mitia-thesis}, following the suggestion in \cite[Remark~2.7]{dg1} to rely on precise quantitative estimates on the corrector and the Green function (the results of \cite{dg1} require the perturbative point process to have a uniformly bounded number of points in a given bounded region of space, a property that does not hold for Poisson point processes). These precise estimates are currently not available in the context of interacting particle systems, and we show here that they are not necessary for the proof of smoothness of the homogenized coefficients. 
Outside of the present paper, we are unaware of results concerning interacting particle systems for which the number of particles in a bounded region of space is not uniformly bounded.

In a previous version of this paper, we stated: ``We believe that the method used here could be adapted to the case of elliptic equations and yield a simpler proof of \cite[Theorem~5.A.1]{Mitia-thesis} that does not rely on quantitative homogenization theory.'' That this is indeed the case has been rigorously established recently in~\cite{brusselssprouts}.

It is asserted in~\cite{brusselssprouts} that the problem considered here is a particular case of the corresponding problem for elliptic equations with random coefficients. This is however not the case, and indeed, the two problems have a different mathematical structure. In the context of particle systems, one \emph{cannot} keep the set of all the other particles frozen in place, as the relevant operator must encode the movement of the entire cloud of particles at once. In other words, in the case of elliptic equations with random coefficients, for each fixed realization of the randomness, we can write down the relevant $d$-dimensional equations, which will feature random coefficients. Instead, for interacting particle systems, the relevant equations must be deterministic and $Nd$-dimensional, if the number of particles is fixed at $N$. In particular, the bulk diffusion matrix is \emph{different} from the asymptotic diffusivity matrix of the diffusion in the random environment obtained by freezing the cloud of all the other particles.

One basic difficulty when studying homogenized coefficients is that they are defined as infinite-volume quantities. As in all other works on the topic, we will therefore start by studying a localized version of the homogenized coefficients. In the context of elliptic equations with random coefficients, this is usually achieved by the introduction of a zero-order term in the equation solved by the corrector. Through this procedure, one obtains a localized corrector that is also stationary under the action of translations. 

In the context of particle systems, the correctors are deterministic functions of the cloud of particles. Stationarity thus boils down to invariance under translations, and by the ergodic theorem, a function that is invariant under translations must be constant. A different approach is therefore necessary. We will rely on the finite-volume quantities introduced in \cite{bulk}, which are inspired by the analogous quantities introduced in \cite{armstrong2016quantitative} for PDEs (see also the monograph \cite{AKMbook}, and \cite{informal} for a gentle introduction). As will be explained below, there are two such quantities; in the PDE setting, one relates to the imposition of a constant tangential gradient on the boundary, while the other relates to the imposition of a constant normal flux on the boundary. In the context of particle systems, we will rely on this second ``flux'' quantity for our proofs, and thus first obtain the smoothness of the inverse of the homogenized matrix. Perhaps surprisingly, significant additional technical difficulties seem to arise if one tries to devise an argument that would rely instead on the ``gradient'' quantity; we refer to Subsection~\ref{subsec:elementary} for more on this point.

Once our quantity is localized in finite volume, it becomes smooth and we can compute its derivatives at arbitrarily high order using classical formulas such as \eqref{e.classical.formula} (this formula is called a chaos expansion in \cite{bookPoisson}, and is called a cluster expansion in~\cite{dg1}). 
The difficulty that remains is to find suitable estimates in order to pass to the limit of infinite volume. The estimates we need have a similar form as those appearing in \cite{AL2, AL1, dg1} in the PDE setting; the authors of \cite{dg1} also single out \cite[Proposition~3.4]{anantharaman-thesis} as an important source of inspiration. However, several adaptations are needed, due to the nature of our localization procedure as well as to the differences inherent to particle systems mentioned above; see for instance the discussion below the statement of Proposition~\ref{cor:var.formulation}.

Another setting in which an expansion for the homogenized coefficients is studied are colloidal particle suspensions \cite{DG-E, GerardVaret2019, GerardVaretHoefer_Einstein,  Haines2012APO, Niethammer2020ALV}. Under a suitable limit of many small particles, the homogenized equation for the fluid is a Stokes system having an effective viscosity. The latter admits an expansion in terms of the asymptotic volume fraction occupied by the particles.

We finally mention that questions similar to the ones contained in this paper were also investigated in the context of the $\nabla \phi$ model (a Gibbs measure modeling a fluctuating interface) \cite{gradphi2}, and non-linear elliptic equations with random coefficients \cite{ferg2,ferg1, fisneu}. In these settings, the goal is to show that the homogenized coefficients depend smoothly on the slope of the limit homogenized solution. This is not a situation in which the varying parameter can be nicely encoded by random fields with short-range correlations. As a consequence, a different, more quantitative approach is then mandatory.

%
%
%
%
%
%
%
%
\section{Precise statement of the main results}
\label{s.statements}

We start by introducing some notation. We view a cloud of particles in $\R^d$ as an element of $\mmd(\Rd)$, the space of $\sigma$-finite measures that are sums of Dirac masses on $\Rd$. The dynamics of the particles is encoded by a mapping $\a_{\circ} : \mmd(\Rd) \to \R^{d\times d}_{\mathrm{sym}}$ taking values in the space of symmetric $d$-by-$d$ matrices. We assume that this mapping satisfies the following properties.

\begin{enumerate}

\item[$\bullet$] \emph{Uniform ellipticity}: there exists $\Lambda < \infty$ such that for every $\mu \in \mmd(\Rd)$,
\begin{align}\label{a.elliptic}
\mathrm{Id} \leq \a_\circ(\mu)  \leq \Lambda \, \mathrm{Id}.
\end{align}

\item[$\bullet$] \emph{Finite range of dependence}: for every $\mu \in \mmd(\Rd)$, we have that
\begin{align}\label{a.local} 
\a_\circ(\mu) = \a_\circ(\mu \mres B_{1/2}),
\end{align}
where $B_{1/2}$ denotes the Euclidean ball of unit diameter centered at the origin, and~$\mres$ is the restriction operator defined in \eqref{e.def.mres}.

\end{enumerate}
In \eqref{a.elliptic} and throughout the paper, whenever $a$ and $b$ are symmetric matrices, we write $a \le b$ to mean that $b-a$ is a positive semidefinite matrix. 

Roughly speaking, we want a particle sitting at the origin and surrounded by a cloud of particles $\mu$ to undergo an instantaneous diffusion driven by the matrix~$\a_\circ(\mu)$. We extend the mapping $\a_\circ$ by stationarity by setting, for every $x \in \Rd$ and $\mu \in \mmd(\Rd)$, 
\begin{equation*}  
\a(\mu,x) := \a_\circ(\tau_{-x} \mu),
\end{equation*}
where $\tau_{-x} \mu$ is the measure $\mu$ translated by the vector $-x$; in other words, for every Borel set $U$, we have $\tau_{-x} \mu(U) = \mu(x + U)$. 
For every $\rho_0 \ge 0$, we denote by $\P_{\rho_0}$ the law of a Poisson point process over $\Rd$ with constant intensity $\rho_0$. We denote by $\E_{\rho_0}$ the associated expectation, and use $\mu$ for the canonical random variable on this probability space. The interacting particle system we aim to study is associated with the formal Dirichlet form
\begin{equation*}  
f \mapsto \E_{\rho_0} \Ll[ \int_\Rd \nabla f \cdot \a \nabla f \, \d \mu \Rr] .
\end{equation*}
We refer to \eqref{e.def.deriv} below for the definition of the gradient of a sufficiently smooth function defined on $\mmd(\Rd)$, and \cite{gu2020decay} for a rigorous construction of the stochastic process.

We expect the evolution of this particle system to be described by a ``homogenized'' or ``hydrodynamic'' equation over large scales. Indeed, this has been shown for discrete models similar to the continuous one studied here, see in particular \cite{fuy}. In order to justify this rigorously, it is very useful to know about the regularity of the homogenized matrix, usually called the \emph{bulk diffusion matrix}, that enters into the equation. The aim of the present work is to show that this matrix is indeed an infinitely differentiable function of the particle density.

For our purposes, it will be convenient to identify the bulk diffusion matrix as a limit of finite-volume approximations. In finite volume, there are in fact two natural approximations to the bulk diffusion matrix, which were introduced in \cite{bulk} and are inspired by \cite{AKMbook, armstrong2016quantitative, informal}. They are based on the following subadditive quantities: for every bounded domain $U$, $p, q \in \Rd$, and $\rho_0 > 0$, we define
\begin{equation}\label{eq:defNu}
\begin{split}
\nu(U,p, \rho_0) &:= \inf_{v \in \cH^1_0(U)}\E_{\rho_0} \Ll[ \frac{1}{\rho_0 \vert U \vert} \int_{U} \frac{1}{2} (p+\nabla v) \cdot \a (p+\nabla v) \, \d \mu \Rr], \\
\nu_*(U,q, \rho_0) &:=  \sup_{u \in \cH^1(U)} \E_{\rho_0} \Ll[\frac{1}{\rho_0 \vert U \vert}\int_{U} \Ll( -\frac 1 2 \nabla u \cdot \a \nabla u + q \cdot \nabla u \Rr) \, \d \mu \Rr],
\end{split}
\end{equation}
where $|U|$ denotes the Lebesgue measure of $U$. Recall that $\mu$ is a sum of Dirac masses; for any function $F$, the integral $\int_U F \, \d \mu = \int_U F(z) \, \d \mu(z)$ therefore stands for the summation of $F(z)$ over every point $z$ in the intersection of $U$ and the support of $\mu$. 
The precise definitions of the function spaces $\cH^1(U)$ and $\cH^1_0(U)$ are given in Section~\ref{s.prelim} below. Informally, the functions in $\cH^1(U)$ are those whose squared gradient have finite integral over $U$; the functions in $\cH^1_0(U)$ must in addition respond continuously to the exit from $U$ or the entrance into~$U$ of a particle.
One can check (see \cite[Proposition~4.1]{bulk} or Subsection~\ref{subsec:elementary} below) that there exist symmetric $d$-by-$d$ matrices $\ab(U, \rho_0), \ab_*(U, \rho_0)$ that satisfy the bound \eqref{a.elliptic} and such that, for every $p,q \in \Rd$,
\begin{equation}\label{eq:NuMatrix}
\nu(U, p, \rho_0) = \frac{1}{2} p \cdot \ab(U, \rho_0) p \quad \text{and} \quad  \nu_*(U, q, \rho_0) = \frac{1}{2} q \cdot \ab_*^{-1}(U, \rho_0) q.
\end{equation}
For every $m \in \N$, we denote by $\cu_m := (- 3^m/2, 3^m/2)^d$ the cube of side-length $3^m$ centered at the origin. We also have that the sequence $(\ab(\cu_m, \rho_0))_{m \in \N}$ is decreasing, and the sequence $(\ab_*(\cu_m, \rho_0))_{m \in \N}$ is increasing. We define the bulk diffusion matrix as  the limit of the latter sequence:
\begin{align}\label{eq:defab}
\ab(\rho_0) := \lim_{m \to \infty} \ab_*(\cu_m, \rho_0).
\end{align}
It was shown in \cite{bulk} that the sequence $(\ab(\cu_m,\rho_0))_{m \in \N}$ converges to the same limit, and moreover, that there exists an exponent $\al > 0$ and a constant $C < \infty$ such that for every $m \ge 1$,
\begin{equation}\label{eq:rate}
\Ll| \ab(\cu_m, \rho_0) - \ab(\rho_0) \Rr| + \Ll| \ab_*(\cu_m, \rho_0) - \ab(\rho_0)\Rr|  \leq C 3^{-\alpha m}.
\end{equation}
The results of the present paper do not rely on this quantitative information. Indeed, to show our main results, we only appeal to \eqref{eq:defab} as the definition of the limit diffusion matrix. This definition coincides with the more classical one based on full-space stationary correctors, as explained in \cite[Appendix~B]{bulk}. 

Throughout the paper, we fix $q \in \Rd$, and denote by $\psi_m \in \cH^1(\cu_m)$ the optimizer in the definition of $\nu_*(\cu_m, q, \rho_0)$, see \eqref{eq:defNu}. The optimizer for $\nu_*(\cu_m, q, \rho_0)$ is unique provided that we impose the condition in \eqref{e.fix.constants} (the formulas derived throughout the paper only involve gradients of $\psi_m$, and are therefore insensitive to the precise way we ``fix the constants'').

For reasons that will be clarified below, we prefer to work with $\psi_m$, which optimizes some $\nu_*$ quantity, rather than with the corresponding optimizer for $\nu$. One consequence of this choice is that we have easier access to information about the smoothness of the mapping $\rho \mapsto \ab^{-1}(\rho)$ than of the mapping $\rho \mapsto \ab(\rho)$. Of course, since these matrices are uniformly elliptic in the sense of \eqref{a.elliptic}, discussing the smoothness of one or the other is equivalent (and from a physical perspective, it is no less natural to focus on ``fixing the average flux at $q$'' than to focus on ``fixing the average gradient at $p$'').

For clarity of exposition, we will first present a proof that the mapping $\rho \mapsto \ab^{-1}(\rho)$ is $C^{1,1}$. The precise statement is as follows.

\begin{theorem}[{$C^{1,1}$} regularity]  
\label{t.C11}
The following limit is well-defined and finite
\begin{equation}
\label{e.def.first}
c_1(\rho_0) := \lim_{m \to \infty} \int_{\R^d} \E_{\rho_0} \Ll[ \frac{1}{\rho_0 |\cu_m|} \int_{\cu_m} \nabla \psi_m \cdot ( \a - \aone)\nabla \psi_m^{\{1\}} \, \d \mu \Rr] \, \d x_1,
\end{equation}
where we write $\aone(\mu, z,x_1) := \a \Ll( \mu + \de_{x_1}, z \Rr)$ and $\nabla \psi_m^{\{1\}}(\mu, z,x_1) := \nabla \psi_m \Ll( \mu + \de_{x_1}, z \Rr)$. Moreover, as $\rho\in \R$ tends to zero, we have
\begin{equation}\label{C11.expansion}
q \cdot \ab^{-1}(\rho_0 + \rho) q = q \cdot \ab^{-1}(\rho_0) q + \rho c_1(\rho_0) + O(\rho^2).
\end{equation}
The term $O(\rho^2)$ hides a multiplicative constant that depends only on $d$, $\Lambda$ and $|q|$ (but not on $\rho_0$).
\end{theorem}
\begin{remark}
Theorem \ref{t.C11} yields that $q \cdot \ab^{-1}(\cdot) q$ is $C^{1,1}$. Indeed, an immediate consequence of expansion \eqref{C11.expansion} is that $c_1(\cdot)$ is the derivative of $q \cdot \ab^{-1}(\cdot) q$. Moreover, using \eqref{C11.expansion} around $\rho_0$ and $\rho_0 +\rho$, we see that $c_1(\rho_0 + \rho) = c_1(\rho_0) + O(\rho)$, i.e. that $c_1$ is Lipschitz continuous.
\end{remark}

A more explicit writing of the right side of \eqref{e.def.first} is:
\begin{equation*}  
\int_{\R^d} \E_{\rho_0} \Ll[ \frac{1}{\rho_0 |\cu_m|} \int_{\cu_m} \nabla \psi_m(\mu,z) \cdot (\a(\mu,z) - \a \Ll( \mu + \de_{x_1},z \Rr)  ) \nabla \psi_m(\mu + \de_{x_1},z) \, \d \mu(z) \Rr] \, \d x_1.
\end{equation*}
In general, we use superscripts to indicate changes in the ``measure'' argument of the function under consideration: for instance, the quantity $\aone$ is obtained from $\a$ by replacing the argument $\mu$ with $\mu + \de_{x_1}$. 

In order to describe higher-order derivatives, we need to generalize this notation to arbitrary subsets of indices. For every finite subset $E \subset \N_+$ and $f$ an arbitrary function of the measure $\mu$, we define
\begin{equation}
\label{e.def.fE}
f^E : (\mu,(x_i)_{i \in E}) \mapsto f(\mu + \sum_{i \in E} \delta_{x_i}).
\end{equation}
For every $i \in \N_+$, we also write 
\begin{equation}
\label{e.def.diff}
D_i f := f^{\{i\}} - f.
\end{equation}
Notice that for every $i \neq j \in \N_+$, we have
\begin{equation*}  
D_i D_j f = (f^{\{j\}} - f)^{\{i\}} - (f^{\{i\}} - f) = f^{\{i,j\}} - f^{\{i\}} - f^{\{j\}} + f.
\end{equation*}
In particular, the operators $D_i$ and $D_j$ commute. We can therefore define, for every $E = \{i_1,\ldots, i_p\} \subset \N_+$, the quantity
\begin{equation}  
\label{e.def.DE}
D_E f := D_{i_1} \cdots D_{i_p} f.
\end{equation}
Finally, we need at times to apply these operators to more complex expressions such as $f + g$, where $f$ and $g$ are two functions of the measure $\mu$, with the understanding that the operator applies only to $f$ and not to $g$. We use the superscript $\#$ to indicate the functions on which these operators are meant to be applied, keeping the others ``frozen''. That is, we write for instance
\begin{align*}
(f^\# + g)^E = f^E + g, \qquad (f^\# g)^E = f^E g,
\end{align*}
and similarly with more complex expressions. 
We also have
\begin{align*}
D_E(f^\# + g) = \Ll\{
\begin{array}{ll}
f+g, & \text{ if } E = \emptyset, \\
D_E f, & \text{ if } E \neq \emptyset, 
\end{array}
\Rr. \qquad D_E(f^\#  g) = (D_E f) g.
\end{align*} 
We use the notation $\dbint{1,k} := \{1,2,\ldots, k\}$. Here is our main result.
\begin{theorem}[Smoothness] 
\label{t.smooth}
For each $k \in \N_+$, there exists a constant $C_k(\Lambda,d) < \infty$ (not depending on $\rho_0$) such that the  limit 
\begin{multline}  
\label{e.def.ck}
c_k(\rho_0) \\ := \lim_{m \to \infty} \int_{(\R^d)^k}  \E_{\rho_0}\bigg[ \frac{1}{\rho_0 \vert \cu_m \vert} \int_{\cu_m} \nabla \psi_m
\cdot D_{\dbint{1,k}}\Ll((\a - \a^\#  )\nabla \psi^\#_m\Rr) \, \d \mu \bigg] \, \d x_1 \, \cdots \, \d x_k,
\end{multline}
is well-defined and satisfies $|c_k(\rho_0)| \le C_k |q|^2$. Moreover, the mapping $\rho_0 \mapsto q \cdot \ab^{-1}(\rho_0) q$ is infinitely differentiable, and its $k$-th derivative is $c_k(\rho_0)$.
\end{theorem}

Note that due to the local nature of the term $(\a - \a^\#)$, see~\eqref{a.local}, the integrals in~\eqref{e.def.first} and~\eqref{e.def.ck} are in fact finite-volume quantities, as the outermost integrals may be replaced by, for instance, $\int_{\cu_{m+1}}$ and $\int_{(\cu_{m+1})^k}$, respectively.

If one modifies the unit range of dependence assumption in \eqref{a.local}, assuming instead that there exists $R \in (0,\infty)$ such that for every $\mu \in \mmd(\Rd)$,
\begin{equation}
\label{e.assumption.R}
    \a_\circ(\mu) = \a_\circ(\mu \mres B_{R/2}),
\end{equation}
then the fact that the mapping $\ab$ depends smoothly on the density $\rho_0$ can be obtained by a change of scale. The particle density for the rescaled process is $R^d \rho_0$. Since the constants in Theorem~\ref{t.smooth} do not depend on $\rho_0$, we obtain that, under the assumption of \eqref{e.assumption.R} and with the same definition \eqref{e.def.ck} of $c_k(\rho_0)$, we have
\begin{equation*}
    |c_k(\rho_0)| \le C_k |q|^2 \,R^{kd}.
\end{equation*}
Estimates of this type may be useful for controlling situations in which the range of dependence is unbounded.

We now comment on the reason why we choose to work with quantities derived from $\nu_*$ rather than $\nu$. Recall that the function $\psi_m$ is the optimizer in the definition \eqref{eq:defNu} of $\nu_*(\cu_m,q,\rho_0)$. This object may seem to depend upon the choice of the particle density $\rho_0$. However, it is in fact not the case. Indeed, the optimization problem for $\nu_*$ can be split into a sum of unrelated optimization problems, one for each fixed number of particles in $\cu_m$. The optimizer for $\nu_*$ is thus a superposition of these optimizers, irrespectively of the underlying density of the measure. 
We refer to Section~\ref{s.prelim} below for a more detailed discussion of this property. The fact that we can view the same object $\psi_m$ as the optimizer of $\nu_*(\cu_m,q,\rho_0)$ for arbitrary values of $\rho_0$ would not be valid were we to work with the optimizers of~$\nu(\cu_m,p,\rho_0)$. 


The remainder of the paper is organized as follows. We discuss function spaces more precisely in Section~\ref{s.prelim}, and prove a technically useful lemma stating that the quantity~$\nu_*$ does not change if the particles become distinguishable. We then show Theorem~\ref{t.C11} in Section~\ref{s.C11}. The more general Theorem~\ref{t.smooth} is then proved in Section~\ref{s.higher-order}. Finally, in Section~\ref{s.local.unif}, we show that the mappings $\rho_0 \mapsto \ab(\cu_m,\rho_0)$ and $\rho_0 \mapsto \ab_*(\cu_m,\rho_0)$ converge to $\rho_0 \mapsto  \ab(\rho_0)$ locally uniformly, and that this is also the case for the convergence in~\eqref{e.def.ck} towards the higher-order derivatives of $\ab(\rho_0)$. 

%
%
%
%
%
%
%
%
\section{Setting and functional framework}
\label{s.prelim}

In this section, we rigorously introduce the notation and functional framework that we use in this paper. In particular, we define the function spaces $\cH^1(\cu_m)$ and $\cH^1_0(\cu_m)$ that appear in the optimization problems $\nu$ and $\nu_*$ in \eqref{eq:defNu}. This will also allow us to justify why, as mentioned in the previous section, we will prove the main results of this paper by mainly working with the quantity $\nu_*$ instead of $\nu$.  


\subsection{Configuration space}
We denote by $\R^d$ the standard Euclidean space, by $Q_s := \Ll(-s/2, s/2\Rr)^d$ the open hypercube of side length $s > 0$, and we write $\cu_m := Q_{3^m}$ for $m \in \N$. We also use $\cu$ as a shorthand notation for the unit cube $\cu_0$. %

We recall that $\mmd(\Rd)$ is the space of $\sigma$-finite measures that are sums of Dirac masses on $\Rd$, which we think of as the configuration space of particles, and that $\P_{\rho_0}$ corresponds to the probability measure for the Poisson point process having constant density $\rho_0 > 0$. We write $\E_{\rho_0}$ for the expectation with respect to $\P_{\rho_0}$. For a Borel set $U\subset \R^d$,  we denote by $\mcl F_U$ the $\sigma$-algebra generated by the mappings $\mu \mapsto \mu(V)$, for all Borel sets $V \subset U$, completed with all the $\P_{\rho_0}$-null sets. We use the notation $\mcl F$ for  $\mcl F_{\Rd}$. With this construction, assumption \eqref{a.local} yields that the random matrix $\a_\circ : \mmd(\Rd) \to \R^{d\times d}_{\mathrm{sym}}$ is an $\mcl F_{B_{1/2}}$-measurable mapping.

\subsection{Function spaces}
We now introduce several function spaces on $\mmd(\Rd)$ that will be used in this paper. In particular, we will give the rigorous definition of $\cH^1(U)$ and~$\cH^1_0(U)$.

We start with basic considerations concerning $\mcl F$-measurable functions on $\mmd(\Rd)$. Given a Borel set $U \subset \Rd$, it is often useful to decompose an $\mcl F$-measurable function into a series of Borel-measurable functions on Euclidean spaces, conditioned on the number of particles in $U$ and the configuration outside $U$. More precisely, we denote by $\mcl B_U$ the set of Borel subsets of $U$. For every $\mu \in \mmd(\Rd)$, we denote by $\mu \mres U \in \mmd(\Rd)$ the measure such that, for every Borel set $V \subset \Rd$,
\begin{equation}  
\label{e.def.mres}
(\mu \mres U)(V) = \mu(U \cap V). 
\end{equation}
Then for $f : \mmd(\Rd) \to \R$ which is $\mcl F$-measurable, we define
\begin{equation}  \label{eq:Projection}
f_n(\cdot, \mu \mres U^c) :
\Ll\{
\begin{array}{rcl}  
U^n  & \to & \R \\
(x_1, \ldots, x_n) & \mapsto & f \Ll( \sum_{i = 1}^n \de_{x_i} + \mu \mres U^c \Rr).
\end{array}
\Rr.
\end{equation}
The function $f_n$ is $\mcl B_U^{\otimes n} \otimes \mcl F_{U^c}$-measurable. Reciprocally, given a series of permutation-invariant functions with such measurability properties, we can reconstruct an $\mcl F$-measurable function $f$ by specifying that, on the event $\mu \mres U = \sum_{i = 1}^n \de_{x_i}$, we have $f(\mu) := f_n(x_1,\ldots,x_n,\mu \mres U^c)$. We call the mapping $f \mapsto (f_n)_{n \in \N}$ the ``canonical projection", and refer to \cite[Lemmas~2.2 and A.1]{bulk} for more details.

We now explain the notion of derivatives for functions defined on $\mmd(\Rd)$. For every sufficiently smooth function $f : \mmd(\Rd) \to \R$, $\mu \in \mmd(\Rd)$, and $x \in \supp \mu$, the gradient $\nabla f(\mu,x)$ is such that, for every ${k \in \{1,\ldots,d\}}$, 
\begin{equation}  
\label{e.def.deriv}
\mathrm e_k \cdot \nabla f(\mu, x) = \lim_{h \to 0} \frac{f(\mu - \delta_x + \delta_{x + h \mathrm e_k}) - f(\mu)}{h},
\end{equation}
where $(\mathrm e_1,\ldots, \mathrm e_d)$ is the canonical basis of $\Rd$. While we wish to emphasize that the function $\nabla f(\mu,\cdot)$ is naturally defined only on $\supp \mu$, we extend it for convenience~as
\begin{equation}  
\label{e.nabla.convention}
\text{for every } x \notin \supp \mu, \quad \nabla f(\mu,x) := 0. 
\end{equation}

To clarify the notion of smooth functions appearing in the previous paragraph, we can appeal to the canonical projections discussed above. For every bounded open set $U \subset \Rd$, we define the sets of smooth functions $\cC^\infty(U)$ and $\cC^\infty_c(U)$ in the following way. We have that $f \in \cC^{\infty}(U)$ if and only if $f$ is an $\mcl F$-measurable function, and for every $\mu \in \mmd(\Rd)$ and $n \in \N$, the function $f_n(\cdot, \mu \mres U^c)$ appearing in \eqref{eq:Projection} is infinitely differentiable on $U^n$. 
The space $\cC^{\infty}_c(U)$ is the subspace of $\cC^{\infty}(U)$ of functions that are $\mcl F_K$-measurable for some compact set $K \subseteq U$.

We define $\cL^2$ to be the space of $\mcl F$-measurable functions $f$ such that $\E_{\rho_0}\Ll[f^2\Rr]$ is finite.  As usual, elements in this function space that coincide $\P_{\rho_0}$-almost surely are identified. We now define $\cH^1(U)$ as the infinite-dimensional analogue of the classical Sobolev space $H^1$: for every $f \in \cC^\infty(U)$, we introduce the norm
\begin{align}\label{eq:defH}
\norm{f}_{\cH^1(U)} = \Ll(\E_{\rho_0}[f^2(\mu)] + \E_{\rho_0}\Ll[\int_U \vert \nabla f(\mu, x) \vert^2 \, \d \mu(x) \Rr]\Rr)^{\frac{1}{2}},
\end{align}
and set 
\begin{equation}
\begin{aligned}\label{def:spaces}
\cH^1(U) &:= \overline{\{ f \in \cC^\infty(U) \, \colon \, \norm{f}_{\cH^1(U)} < +\infty \}}^{\norm{ \, \cdot \,}_{\cH^1(U)}},\\
\cH^1_0(U) &:= \overline{\{ f \in \cC^\infty_c(U) \, \colon \, \norm{f}_{\cH^1(U)} < +\infty \}}^{\norm{ \, \cdot \,}_{\cH^1(U)}},\\
\end{aligned}
\end{equation}
namely the completion, under $\norm{\cdot}_{\cH^1(U)}$, of the sets of functions in $\cC^\infty(U)$ or $\cC^\infty_c(U)$ that have finite norm $\norm{\cdot}_{\cH^1(U)}$. As for classical Sobolev spaces, for every $f \in \cH^1(U)$, we can interpret $\nabla f(\mu,x)$ when $x \in U \cap \supp \mu$ in a weak sense. This may be understood via the canonical projection in \eqref{eq:Projection}.

The two spaces $\cH^1(U)$ and $\cH^1_0(U)$ share many similarities as well as some fundamental differences. The latter ones, in turn, derive from the differences between $\cC^\infty(U)$ and $\cC^{\infty}_c(U)$: On the one hand, functions in $\cC^\infty(U)$ do depend on $\mu \mres U^c$ and the number of particles $\mu(U)$ in a relatively arbitrary (measurable) way. On the other hand, functions in the subset $\cC^{\infty}_c(U)$ are $\mathcal{F}_U$-measurable as they do not depend on particles that cross the boundary $\partial U$. 

\smallskip

When managing elements in $\cH^1(U)$ or $\cH^1_0(U)$, it is at times useful to think about them in terms of their canonical projection defined in \eqref{eq:Projection}: Let $f \in \cH^1(U)$ and let $(f_n)_{n\in \N}$  be the associated canonical projection. Then for $\P_{\rho_0}$-almost every $\mu\mres U^c$ and every $n\in \N$, we have that
\begin{itemize}

\item The function $f_n( \, \cdot \, , \mu\mres U^c)$ belongs to the (standard) Sobolev space $H^1(U^n)$;

\item The function $f_n( \, \cdot \, , \mu\mres U^c)$ is invariant under permutations: if $S_n$ denotes the set of permutations of $\dbint{1,n}$ and we write $(x_1, \cdots ,  x_n) \in U^n$, then for every $\sigma \in S_n$ it holds
\begin{align}\label{perm.invariance}
f_n( x_1, \cdots, x_n, \mu\mres U^c) = f_n( x_{\sigma(1)}, \cdots x_{\sigma(n)} , \mu\mres U^c ) \ \ \ \ \text{almost everywhere in $U^n$.}
\end{align}
\end{itemize}

If $f \in \cH^1_0(U)$, then the canonical partition needs to satisfy the following additional ``compatibility condition'': for every $n\in \N_+$ and on the set $\{(x_1, \cdots, x_n) \in U^n \, \colon \, x_1 \in \partial U \}$, it holds
\begin{align}\label{comp.condition}
f_{n}(x_1, \cdots, x_n , \mu\mres U^c) = f_{n-1}(x_2, \cdots, x_n,  \mu\mres U^c ),
\end{align}
where the identity is to be understood in the sense of traces. Note that, by the invariance under permutations, the above property also holds if $x_1$ is replaced by any other coordinate $x_i$, $i=2, \cdots, n$. Moreover, for every $f \in \cH^1_0(U)$ and $n \in \N$, we have that $f_n(\cdot,\mu \mres U^c)$ in fact does not depend on $\mu \mres U^c$. 

We summarize the previous remarks in Table \ref{t.differences}.
\begin{table}[ht]
\begin{center}
\begin{tabular}{ c|cccc } 
\toprule
\shortstack{Function \\ space} & \shortstack{$H^1$-regularity for \\  particles in $U$} & \shortstack{Compatibility\\ condition when \\ particles cross $\partial U$} & \shortstack{$\mcl F_U$- \\measurable} & \shortstack{For every\\ open set  $V \subset U$} \\
\midrule
$\cH^1(U)$ & Yes & No&  No& $\cH^1(U) \subset \cH^1(V)$\\ 
$\cH^1_0(U)$ & Yes &  Yes &  Yes& $\cH^1_0(V) \subset \cH^1_0(U)$\\ 
\bottomrule
\end{tabular}
\medskip
\caption{Differences between $\cH^1(U)$ and $\cH^1_0(U)$.}
\label{t.differences}
\end{center}
\end{table}

\subsection{Elementary properties of optimizers}
\label{subsec:elementary}
As seen in the previous subsection, the spaces $\cH^1$ and $\cH^1_0$ differ in important ways, and this will translate into differences for the optimizers of $\nu$ and $\nu_*$. In fact, except in part of Section~\ref{s.local.unif}, we will only rely on quantities derived from $\nu_*$. In this subsection, we present some key properties of optimizers of this quantity, and highlight those that would not be shared by the optimizers of~$\nu$. 

For $U$ a Lipschitz domain and $q \in \Rd$, we denote by $\psi_{U,q} \in \cH^1(U)$ the maximizer in the definition of $\nu_*(U,q,\rho_0)$. By \cite[Proposition~4.1]{bulk} (see also Lemma~\ref{lem:J} below), this optimizer exists, and is unique provided we also impose that
\begin{equation}
\label{e.fix.constants}
{\E \Ll[ \psi_{U,q} \vb \mu(U), \mu \mres U^c \Rr] = 0}.
\end{equation} 
This optimizer is $\mcl{F}_{B_{1/2}(U)}$-measurable, with $B_{1/2}(U) = \{x \in \Rd: \dist(x, U) < \frac{1}{2}\}$. Since $q \mapsto \psi_{U,q}$ is a linear mapping, there exists a matrix $\ab_*(U,\rho_0)$ such that 
\begin{equation*}  
\nu_*(U,q,\rho_0) = \E_{\rho_0} \Ll[\frac{1}{\rho_0 \vert U \vert}\int_{U} \Ll( -\frac 1 2 \nabla \psi_{U,q} \cdot \a \nabla \psi_{U,q} + q \cdot \nabla \psi_{U,q} \Rr) \, \d \mu \Rr] = \frac 1 2 q \cdot \ab_*^{-1}(U,\rho_0) q .
\end{equation*}
The uniform ellipticity assumption \eqref{a.elliptic} readily implies that $\mathrm{Id} \le \ab_*(U,\rho_0) \le \Lambda \mathrm{Id}$.
By the first variation, we have for every $u \in \cH^1(U)$ that
\begin{equation}  
\label{e.first.var}
\E_{\rho_0} \Ll[\int_{U} \Ll( - \nabla \psi_{U,q} \cdot \a \nabla u+ q \cdot \nabla u \Rr) \, \d \mu \Rr] = 0.
\end{equation}
Using this with $u = \psi_{U,q}$, we get that
\begin{equation}  
\label{e.energy.identities}
\begin{split}
q \cdot \ab_*^{-1}(U,\rho_0) q &= \E_{\rho_0} \Ll[\frac{1}{\rho_0 \vert U \vert}\int_{U} \nabla \psi_{U,q} \cdot \a \nabla \psi_{U,q} \, \d \mu \Rr] \\
& = \E_{\rho_0} \Ll[\frac{1}{\rho_0 \vert U \vert}\int_{U} q \cdot \nabla \psi_{U,q} \, \d \mu \Rr].    
\end{split}
\end{equation}
In particular, using the uniform ellipticity assumption once more, we obtain the basic Dirichlet energy estimate
\begin{equation}
\label{energy.basic}
\E_{\rho_0} \Ll[\frac{1}{\rho_0 \vert U \vert}\int_{U} |\nabla \psi_{U,q}|^2\, \d \mu \Rr] \le |q|^2.
\end{equation}
Similar properties are also valid for optimizers of $\nu$, and we refer to \cite[Proposition~4.1]{bulk} for details. Optimizers of $\nu_*$ differ however in one crucial aspect: denoting by $(\psi_{U,q,n})_{n \in \N}$ the canonical projection of $\psi_{U,q}$, see \eqref{eq:Projection}, we can identify each $\psi_{U,q,n}(\cdot,\mu \mres U^c)$ as the solution to an elliptic equation. In particular, the function~$\psi_{U,q}$, which was defined as the optimizer in the definition of  $\nu_*(U,q,\rho_0)$, \emph{in fact does not depend on $\rho_0$}. This property would not be valid for optimizers of $\nu$. 

In order to clarify this, we introduce the following notation: for each $n \in \N$, $\mu \in \mmd(\Rd)$, and $u \in H^1(U^n)$, we write
\begin{multline}\label{eq:DualPDE}
\mcl J_n(u, U, q, \mu \mres U^c) 
\\
:=  \frac1{\rho_0|U|}\fint_{U^n} \sum_{i=1}^n \Ll( -\frac 1 2 \nabla_{x_i} u \cdot \a\Ll(\sum_{i=1}^n \delta_{x_i} + \mu \mres U^c, x_i\Rr) \nabla_{x_i} u + q \cdot \nabla_{x_i} u \Rr) \, \d x_1 \cdots \, \d x_n. 
\end{multline}
This quantity corresponds to the functional that is optimized in the definition of~$\nu_*$, see \eqref{eq:defNu}, but where we have conditioned on $\mu(U) = n$ and $\mu \mres U^c$; and where moreover, we substituted an arbitrary $u \in H^1(U^n)$ in place of the canonical projection $u_n$ of some function $u \in \cH^1(U)$. We thus have 
\begin{align}  
\notag
\nu_*(U,q,\rho_0) 
& = \sup_{u \in \cH^1(U)}\E_{\rho_0} \Ll[  \sum_{n \in \N} \P_{\rho_0} \Ll[ \mu(U) = n \Rr] \, \mcl J_n(u_n(\cdot,\mu \mres U^c),U,q,\mu \mres U^c) \Rr] \\
\label{e.supJ}
& \le \E_{\rho_0} \Ll[ \sum_{n \in \N} \P_{\rho_0} \Ll[ \mu(U) = n \Rr] \, \sup_{u \in H^1(U^n)} \mcl J_n(u,U,q,\mu \mres U^c) \Rr].
\end{align}
The next lemma implies that the inequality above is in fact an equality. Recall that we denote by $(\psi_{U,q,n})_{n \in \N}$ the canonical projection of $\psi_{U,q}$, the optimizer of $\nu_*(U,q,\rho_0)$.

\begin{lemma}\label{lem:J}
For every $\mu \in \mmd(\Rd)$ and $n \in \N$, let $u_{U,q,n}(\cdot,\mu \mres U^c) \in H^1(U^n)$ be the unique maximizer of the functional $\mcl J_n(\cdot, U,q,\mu \mres U^c)$ subject to the constraint $\fint_{U^n} u_{U,q,n}(\cdot, \mu \mres U^c) = 0$. For $\P_{\rho_0}$-almost every $\mu \in \mmd(\Rd)$ and every $n \in \N$, we have
\begin{equation}  
\label{e.lem:J}
u_{U,q,n}(\cdot, \mu \mres U^c) = \psi_{U,q,n}(\cdot, \mu \mres U^c).
\end{equation}
\end{lemma}
\begin{remark}  
The quantities $u_{U,q,n}(\cdot, \mu \mres U^c)$ and $\psi_{U,q,n}(\cdot, \mu \mres U^c)$ in fact only depend on the restriction of $\mu \mres U^c$ to the set of points that are at distance at most $1/2$ from~$U$, by the finite-range dependence assumption \eqref{a.local}. The statement that \eqref{e.lem:J} holds for $\P_{\rho_0}$-almost every $\mu \in \mmd(\Rd)$ therefore does not depend on $\rho_0 > 0$. We are forced to state \eqref{e.lem:J} only for $\P_{\rho_0}$-almost every $\mu$ since a priori we only know that $\psi_{U,q,n}(\cdot,\mu \mres U^c)$ is well-defined for $\P_{\rho_0}$-almost every $\mu$; but the lemma itself provides us with a straightforward way to extend the definition to every $\mu \in \mmd(\Rd)$. In the proof below, we observe that there exists a function $u_{U,q} \in \cH^1(U)$ whose canonical projection is $(u_{U,q,n})_{n \in \N}$, and then show that $u_{U,q} = \psi_{U,q}$. 
\end{remark}
\begin{proof}[Proof of Lemma~\ref{lem:J}]
We first observe that, for each $\mu \in \mmd(\Rd)$, the function $u_{U,q,n}(\cdot, \mu \mres U^c)$ is invariant under permutation of its coordinates. This is immediate from the facts that $\mcl J_n(\cdot,U,q,\mu \mres U^c)$ admits a unique mean-zero maximizer, and that this functional as well as the mean-zero constraint are invariant under permutations. 

We now define the function $u_{U,q}: \mcl M_\delta(\R^d) \to \R$ in such a way that, on the event that $\mu \mres U = \sum_{i = 1}^n \de_{x_i}$, we have
\begin{align*}
u_{U,q}(\mu) := u_{U,q,n}(x_1, \cdots, x_n , \mu\mres U^c).
\end{align*}
This definition makes sense since we have verified that $u_{U,q,n}(\cdot,\mu \mres U^c)$ is invariant under permutation of its coordinates. It is also clear that the canonical projection of the function $u_{U,q}$ is the family of functions $(u_{U,q,n})_{n \in \N}$ (so the notation is sound). 

We now argue that $u_{U,q} \in \cH^1(\cu_m)$, and by the uniqueness of the optimizer for~$\nu_*$ and \eqref{e.supJ}, this will imply that $u_{U,q} = \psi_{U,q}$, as desired. Let now $\mu\mres U^c$ be fixed. By construction, each function $u_{U,q,n}(\cdot , \mu \mres U^c)$ satisfies, for every $v \in H^1(U^n)$, the variational identity
\begin{align*}
\int_{U^n} \sum_{i=1}^n \Ll(  \nabla_{x_i} u_{U,q,n}(\cdot , \mu \mres U^c) \cdot \a\Ll(\sum_{i=1}^n \delta_{x_i} + \mu \mres U^c, x_i\Rr) \nabla_{x_i} v - q \cdot \nabla_{x_i} v \Rr) \, \d x_1 \cdots \, \d x_n =  0.
\end{align*} 
Choosing $v= u_{U,q,n}(\cdot , \mu \mres U^c)$ and using \eqref{a.elliptic} and Young's inequality, we infer that
\begin{align}\label{slice.energy}
\frac 1 n \sum_{i=1}^n \fint_{U^n} |\nabla_{x_i} u_{U,q,n}(\cdot , \mu \mres U^c)|^2 \le |q|^2.
\end{align}
Moreover, since $u_{U,q,n}(\cdot , \mu \mres U^c)$ has zero-average on $U^n$, we may apply Poincar\'e's inequality in the product domain $U^n$ (see for instance \cite{firstcourse} or \cite[Proposition~3.1]{bulk}) and obtain that there exists a constant $C(U) < +\infty$ such that
\begin{align}\label{slice.L2}
\fint_{U^n} |u_{U,q,n}(\cdot , \mu \mres U^c)|^2 \le C  \sum_{i=1}^n \fint_{U^n} |\nabla_{x_i} u_{U,q,n}(\cdot , \mu \mres U^c)|^2 \stackrel{\eqref{slice.energy}}{\le} C n |q|^2.
\end{align}
Estimates \eqref{slice.energy} and \eqref{slice.L2} and the definition of $\E_{\rho_0}\Ll[\, \cdot \, \Rr]$ immediately imply that $u_{u,q} \in \cH^1(U)$. This concludes the proof of Lemma \ref{lem:J}. 
\end{proof}

As announced, Lemma~\ref{lem:J} demonstrates that the optimizer for $\nu_*(U,q,\rho_0)$ in fact does not depend on $\rho_0$: regardless of the density, it is always the same $\psi_{U,q}$ whose canonical projections are described by this lemma. The only difference is that optimizers for $\nu_*(U,q,\cdot)$ at different densities receive point processes with different densities as their argument. 

\smallskip

We also stress that another immediate consequence of Lemma \ref{lem:J}, see also \eqref{slice.energy}, is that each maximizer $\psi_{U,q}$ satisfies the following improved energy inequality:
\begin{align}\label{energy.basic.slice}
\E_{\rho_0} \Ll[\frac{1}{\rho_0\vert U \vert}\int_{U} |\nabla \psi_{U,q}|^2 \, \d\mu \, \biggl| \, \mu(U) , \mu\mres U^c \Rr] \leq |q|^2 \frac{\mu(U)}{\rho_0 |U|},
\end{align}
for every $\mu\mres U^c$ and number of particles $\mu(U) \in \N$ fixed. Note that this inequality implies \eqref{energy.basic}.

In most of the paper, we keep the parameter $q$ fixed, and work with the domain $U = \cu_m$. We recall the notation $\psi_m := \psi_{\cu_m,q}$. 

\subsection{Coupling of point processes}\label{subsec:Couple}
When studying the regularity of the bulk diffusion matrix, it is useful to introduce a coupling between different densities. Recall that we keep $\rho_0 \in (0, \infty)$ fixed, and let $\mu \sim \mathrm{Poisson}(\rho_0)$ be the ``reference'' Poisson point process, with constant density $\rho_0$. For $\rho\ge 0$, we define another independent Poisson point process $\mu_\rho \sim \mathrm{Poisson}(\rho)$, which we think of as a small perturbation. Then we denote by $\P = \P_{\rho_0} \otimes \P_{\rho}$ the joint probability measure, with associated expectation $\E$, and we observe that ${\mu + \mu_\rho \sim \mathrm{Poisson}(\rho_0 + \rho)}$, by the superposition property for independent Poisson point processes.

Notice that the definition of the space $\cH^1$ actually depends on the density of particles, although this was kept implicit in the notation. When we want to resolve the ambiguity, we write $\cH^1(U, \mu)$ 
for the space as defined in \eqref{eq:defH}, and we write $\cH^1(U, \mu + \mu_\rho)$ for the same space but with density $(\rho_0 + \rho)$.  In line with the notation introduced in \eqref{e.def.fE}, we use a superscript $\rho$ to indicate when the measure argument of a function is taken to be $\mu + \mu_\rho$. For instance, when we write $\a^{\rho}$ in some expression, we always understand that it is evaluated as $\a(\mu + \mu_\rho, \cdot)$; the notation~$\a$ is understood to be evaluated at $\mu$ instead. The same convention applies as well to $\psi_m$ and $\psi_m^\rho$: the former represents $\psi_m(\mu)$ and the latter $\psi_m(\mu + \mu_\rho)$. As discussed in the previous subsection, the function $\psi_m^{\rho}$ can be interpreted as the optimizer of $\nu_*(\cu_m,q,\rho_0 + \rho)$. This notation allows us to write, for instance,
\begin{align*}
\nu_*(\cu_m, q, \rho_0 + \rho) = \E \Ll[\frac{1}{(\rho_0 + \rho) \vert \cu_m \vert}\int_{\cu_m} \Ll( -\frac 1 2 \nabla \psi^\rho_m \cdot \a^\rho \nabla \psi^\rho_m + q \cdot \nabla \psi^\rho_m \Rr) \, \d (\mu + \mu_\rho) \Rr].
\end{align*}

We can also define a quantity $\nu_*$ perturbed by adding a finite number of particles uniformly. We denote by $E \subset \N_+$ the index set, and write $\mu_E := \sum_{i \in E} \delta_{x_i}$. Throughout the paper we use the following compact notation for the integration with respect to the particles in $E$:
\begin{align}\label{eq:SymbolInte}
\int_{U^E} ( \cdots) := \int_{U^{\vert E \vert}} ( \cdots) \, \prod_{i \in E} \d x_i, \qquad \fint_{U^E} ( \cdots) :=  \frac 1 {|U|^{|E|}} \int_{U^{\vert E \vert}} ( \cdots) \, \prod_{i \in E} \d x_i,
\end{align}
with the understanding that, if $E = \emptyset$, then $\fint_{U^\emptyset}(\cdots) = (\cdots)$.

We define the function space $\cH^1(U, \mu + \mu_E)$ as the completion in $\cH^1(U, \mu + \mu_E)$ of the space of functions in $\cC^{\infty}(U)$ such that the norm
\begin{align*}
\norm{f}^2_{\cH^1(U, \mu + \mu_E)} := \int_{U^E}\Ll(\E_{\rho_0}\Ll[f^2(\mu + \mu_E)\Rr] + \E_{\rho_0}\Ll[\int_U \vert \nabla f(\mu + \mu_E, x) \vert^2 \, \d(\mu + \mu_E)(x) \Rr]\Rr),
\end{align*}
is finite. Similarly to the notation $\a^\rho$ discussed above, we use the shorthand notation $\a^E$ to denote the function $\a(\mu + \mu_E, \cdot)$. The dual problem $\nu^E_*(\cu_m, q, \rho_0)$ is defined as 
\begin{multline}\label{eq:DualE}
\nu^E_*(\cu_m, q, \rho_0) \\
:= \sup_{u \in \cH^1(\cu_m, \mu + \mu_E)} \fint_{(\cu_m)^E}\E_{\rho_0} \Ll[\frac{1}{\rho_0 \vert \cu_m \vert}\int_{\cu_m} \Ll( -\frac 1 2 \nabla u \cdot \a^E \nabla u + q \cdot \nabla u \Rr) \, \d (\mu+\mu_E) \Rr],
\end{multline}
and we denote its optimizer by $\psi^E_m$. Similarly to what was discussed for $\psi_m^\rho$ in the previous subsection, we have that $\psi^E_m$ coincides with the function $\psi_m(\mu + \mu_E)$, and we can always think of the superscript $E$ as indicating the operation of adding $\mu_E$ to the argument of the function, see \eqref{e.def.fE}.


We have built the configuration space in order to capture the notion of indistinguishable particles: if we exchange the positions of two particles, the measure does not change. However, when perturbing the measure $\mu$ with the addition of $\mu_\rho$ (or $\mu_E$), the setting naturally introduces some amount of distinguishability between particles, as some come from the measure $\mu$ and some from the measure~$\mu_\rho$ (or $\mu_E$). Lemma~\ref{lem:J} has clarified in particular that ``nothing is gained'' in the optimization problem if we allow the particles to be distinguishable. We now ``project'' this statement into a form in which, roughly speaking, we can only distinguish from which measure (such as $\mu$, $\mu_\rho$ or $\mu_E$) a particle ``comes''.

\begin{proposition}\label{cor:var.formulation}
For all finite sets  $E, F \subset \N_+$, we have that
\begin{align}\label{eq:Harmonic1}
&\int_{(\cu_{m+1})^{E \cup F}}\E\Ll[\int_{\cu_m} (\nabla \psi_m^{\rho, F}  \cdot \a^E \nabla \psi^E_m -  \nabla\psi_m^{\rho,F} \cdot q) \, \d \mu  \Rr] = 0, \\
&\int_{(\cu_{m+1})^{E \cup F}}\E\Ll[\int_{\cu_m} (\nabla \psi_m^{F}  \cdot \a^{\rho,E} \nabla \psi^{\rho,E}_m -  \nabla\psi_m^{F} \cdot q)\, \d \mu \Rr]=0.
\label{eq:Harmonic2}
\end{align}
\end{proposition}

Before turning to the proof, we point out some possibly surprising features of this result. First, as pointed out above, these relations differ from \eqref{e.first.var} in that the test functions can distinguish between different types of particles: for instance, the function $\psi_m^F$ depends only on $\mu + \mu_F$, and cannot be thought of as a function of $\mu  +\mu_\rho + \mu_E$, as one might hope at first. Second, the integration of the additional particles indexed by $E \cup F$ is carried over the larger domain $\cu_{m+1}$, instead of the domain $\cu_m$ that one might expect. And finally, we integrate over $\mu$ only, while one might at first expect \eqref{eq:Harmonic1} and \eqref{eq:Harmonic2} to be integrated against $\mu + \mu_E$ and $\mu + \mu_\rho + \mu_E$ respectively. The proof below will need to address each of these aspects. The particular form of \eqref{eq:Harmonic1} and \eqref{eq:Harmonic2} we have chosen here will turn out to be the most convenient later on: for instance, we will often need to study linear combinations of $\psi_m^E$'s for different sets $E$, and it is most convenient that the measure against which we integrate does not depend on $E$. Similarly, when we study the effect of a change in the density, some additional particles that fall in a layer around $\cu_m$ will need to be taken into account, and it is more convenient that \eqref{eq:Harmonic1} and \eqref{eq:Harmonic2} take such perturbations into account.

\begin{proof}[Proof of Proposition~\ref{cor:var.formulation}] We first show \eqref{eq:Harmonic1}. The proof can be divided into 4 steps.

\textit{Step 1: Decomposition.}
For $E, F \subset \N$ fixed, we split $E \cup F = E \sqcup (F\setminus E)$ and write $\mu_F = \mu_{F \cap E} + \mu_{F\setminus E}$. By Fubini's theorem we reorganize
\begin{equation}
\begin{aligned}\label{cor:var.formulation.1}
\int_{(\cu_{m+1})^{E \cup F}}&\E\biggl[ \int_{\cu_m} (\nabla \psi_m^{\rho, F}  \cdot \a^{E} \nabla \psi^{E}_m -  \nabla\psi_m^{\rho,F} \cdot q)\, \d \mu \biggr]\\
&= \int_{(\cu_{m+1})^{F\setminus E}} \E \Ll[ \sum_{n\in \N} \P_{\rho_0}[\mu(\cu_m)=n] \, A_n(\mu_\rho, \mu\mres(\cu_m)^c, \mu_{F \setminus E}) \Rr],
\end{aligned}
\end{equation}
where for every $n\in \N$ we defined
\begin{align*}
&A_n(\mu_\rho, \mu\mres(\cu_m)^c, \mu_{F\setminus E})\\
&: = \int_{(\cu_{m+1})^{E}} \E\Ll[\int_{\cu_m} (\nabla \psi_m^{\rho,F} \cdot \a^E \nabla \psi^E_m - \nabla \psi_m^{\rho,F} \cdot q) \, \d \mu \, \biggl |  \, \mu_\rho, \,  \mu\mres{(\cu_m)^c}, \, \mu(\cu_m)= n  \Rr].
\end{align*}
We note that only $\psi_m^{\rho,F}$ depends on the realization of $\mu_\rho$. Hence, the previous term can be rewritten as
\begin{align*}
&A_n(\mu_\rho, \mu\mres(\cu_m)^c, \mu_{F\setminus E})\\
&= \int_{(\cu_{m+1})^E}\E_{\rho_0}\Ll[\int_{\cu_m} (\nabla \psi_m^{\rho,F} \cdot \a^E \nabla \psi^E_m - \nabla \psi_m^{\rho,F} \cdot q) \, \d \mu \, \biggl | \,  \mu\mres{(\cu_m)^c}, \, \mu(\cu_m)= n \Rr],
\end{align*}
in which the measures $\mu_\rho$ and $\mu_{F\setminus E}$ in $\psi_m^{\rho,F}$ are fixed. 

We now apply a further decomposition of $A_n$. Let $G \subset E$ be the set of particles in $\cu_m$, then the integration becomes
\begin{align*}
\int_{(\cu_{m+1})^E} = \sum_{G \subset E} \int_{(\cu_{m+1}\setminus \cu_m)^{E\setminus G}} \int_{(\cu_m)^{G}},    
\end{align*}
and we can write
\begin{multline}\label{cor:var.formulation.2}
A_n(\mu_\rho, \mu\mres(\cu_m)^c, \mu_{F\setminus E}) \\
= \sum_{G \subset E} \int_{(\cu_{m+1}\setminus \cu_m)^{E\setminus G}} B_n(\mu_\rho, \mu\mres(\cu_m)^c, \mu_{F\setminus E}, \mu_{E \setminus G}),
\end{multline}
where the quantity $B_n(\mu_\rho, \mu\mres(\cu_m)^c, \mu_{F\setminus E}, \mu_{E \setminus G})$ is defined as 
\begin{multline*}
B_n(\mu_\rho, \mu\mres(\cu_m)^c, \mu_{F\setminus E}, \mu_{E \setminus G}) \\
 := \int_{(\cu_m)^{G}} \E_{\rho_0}\Ll[\int_{\cu_m} (\nabla \psi_m^{\rho,F} \cdot \a^E \nabla \psi^E_m - \nabla \psi_m^{\rho,F} \cdot q) \, \d \mu  \, \biggl | \, \mu\mres{(\cu_m)^c}, \, \mu(\cu_m)= n \Rr].    
\end{multline*}

\textit{Step 2: Finding the associated variational problem.}
We now claim that, for each $G \subset E$ and every $\mu_{E \setminus G}$,
\begin{equation}
\begin{aligned}\label{vanishing.slices}
B_n(\mu_\rho, \mu\mres(\cu_m)^c, \mu_{F\setminus E}, \mu_{E \setminus G}) = 0.
\end{aligned}
\end{equation}
To prove this, we begin by specifying where  the functions $\psi_m^E$ and $\a^E$ are evaluated. Splitting $\mu_E = \mu_{E\cap G} + \mu_{E\setminus G}$ and recalling the definition \eqref{eq:Projection} of the canonical projection for $\psi_m$, we note that the term $\psi_m^{E}$ in the expectation corresponds to $\psi_{m, n+|G|}( \, \cdot \, , \mu\mres (\cu_m)^c + \mu_{E\setminus G})$. By Lemma \ref{lem:J}, this function is a maximizer for the functional $\mcl J_{n+ |G|}(\cdot, \cu_m, q, \mu\mres (\cu_m)^c + \mu_{E\setminus G} )$. Moreover, we notice that the left-hand side of \eqref{vanishing.slices} is quite similar to the variational formulation for the optimization problem for $\mcl J_{n+ |G|}(\cdot, \cu_m, q, \mu\mres (\cu_m)^c + \mu_{E\setminus G} )$ that is tested against the function $\psi_m^{\rho,F}$, so we define
\begin{multline*}
\tilde{B}_n(\mu_\rho, \mu\mres(\cu_m)^c, \mu_{F\setminus E}, \mu_{E \setminus G}) \\
:=
\int_{(\cu_m)^{G}} \E_{\rho_0}\Ll[\int_{\cu_m} (\nabla \psi_m^{\rho,F} \cdot \a^E \nabla \psi^E_m - \nabla \psi_m^{\rho,F} \cdot q) \, \d (\mu + \mu_G) \, \biggl | \, \mu\mres{(\cu_m)^c}, \, \mu(\cu_m)= n \Rr]. 
\end{multline*}
In the following, we will 
\begin{itemize}
    \item verify that $\psi_m^{\rho,F}$ is an admissible test function for the optimization problem for $\mcl J_{n+ |G|}(\cdot, \cu_m, q, \mu\mres (\cu_m)^c + \mu_{E\setminus G} )$, showing that $\tilde{B}_n = 0$;
    \item deduce from $\tilde{B}_n = 0$ the claim~\eqref{vanishing.slices}. 
\end{itemize}

\textit{Step 3: The test function is admissible.}
 Conditioned on $\mu(\cu_m) = n$, we write $\mu \mres \cu_m$ and $\mu_{F}$ as 
\begin{align*}
\mu \mres \cu_m = \sum_{i=1}^n \delta_{y_i} , \qquad 
\mu_F = \sum_{i=1}^{\vert F\cap G \vert}\delta_{x_{\alpha_i}} + \sum_{j=1}^{\vert F \setminus G\vert} \delta_{x_{\beta_j}}.
\end{align*}
Then conditioned on $\mu(\cu_m) = n$, for $\P$-almost every realization of $\mu \mres (\cu_m)^c, \mu_\rho$,  and (Lebesgue-) almost every realization of $\mu_{F\setminus E}$, the function
\begin{align*}
(y_1, \cdots, y_n, x_{\alpha_1}, \cdots, x_{\alpha_{\vert F \cap G\vert}}) \mapsto  \psi_{m}\Ll( \sum_{i=1}^n \delta_{y_i} +  \sum_{i=1}^{\vert F\cap G \vert}\delta_{x_{\alpha_i}} + \mu_{F \setminus G} + \mu \mres (\cu_m)^c + \mu_\rho \Rr),  
\end{align*}
 belongs to $H^1((\cu_m)^{n+\vert F \cap G\vert})$ thanks to \eqref{energy.basic.slice}. Thus it also belongs to  $H^1((\cu_m)^{n+\vert G\vert})$ with respect to the integration $(\mu + \mu_G)$, and it is an admissible function for the optimization problem for $\mcl J_{n+ |G|}(\cdot, \cu_m, q, \mu\mres (\cu_m)^c + \mu_{E\setminus G} )$. This implies  that for $\P$-almost every realization of $\mu \mres (\cu_m)^c, \mu_\rho$, and Lebesgue-almost every realization of $\mu_{F\setminus E}$, we have
 \begin{align*}
\tilde{B}_n(\mu_\rho, \mu\mres(\cu_m)^c, \mu_{F\setminus E}, \mu_{E \setminus G}) = 0.     
 \end{align*}

\textit{Step 4: Passage from $\tilde{B}_n = 0$ to $B_n = 0$.} 
 We stress that from the gradient of $\psi^{\rho,F}_m$ in $\tilde{B}_n$, out of the $(n + \vert G \vert)$ particles in $(\mu + \mu_G)$ only those in the support of $(\mu + \mu_{F \cap G})$ contribute. Thus we can rewrite  $\tilde{B}_n$ as 
\begin{multline*}
\tilde{B}_n(\mu_\rho, \mu\mres(\cu_m)^c, \mu_{F\setminus E}, \mu_{E \setminus G}) \\
=
\int_{(\cu_m)^{G}} \E_{\rho_0}\Ll[\int_{\cu_m} (\nabla \psi_m^{\rho,F} \cdot \a^E \nabla \psi^E_m - \nabla \psi_m^{\rho,F} \cdot q) \, \d (\mu + \mu_{F \cap G}) \, \biggl | \, \mu\mres{(\cu_m)^c}, \, \mu(\cu_m)= n \Rr]. 
\end{multline*}
Notice now that the integrals above give the same contribution for every particle in ${(\mu + \mu_{F \cap G})}$, because $\psi_m^{\rho,F}, \psi^E_m$ and $\a^E$ are all invariant under permutations for these particles. As a consequence, we have 
\begin{align*}
&B_n(\mu_\rho, \mu\mres(\cu_m)^c, \mu_{F\setminus E}, \mu_{E \setminus G}) \\
&= \Ll(\frac{n}{n + \vert F \cap G\vert}\Rr)  \tilde{B}_n(\mu_\rho, \mu\mres(\cu_m)^c, \mu_{F\setminus E}, \mu_{E \setminus G}) \\
&= 0.
\end{align*}
We thus established \eqref{vanishing.slices}. Then we put it back to \eqref{cor:var.formulation.1} and \eqref{cor:var.formulation.2}, which implies that the left-hand side of  \eqref{cor:var.formulation.1} is zero and concludes the proof of \eqref{eq:Harmonic1}.

\medskip

We now turn to \eqref{eq:Harmonic2}. The proof is similar and one can repeat the 4 steps above. The only difference is that we also need to do the expansion according to the number of particles $\mu_{\rho}(\cu_m)$ and we skip the details. 
\end{proof}

\begin{remark}  
The proof in fact yields the following stronger result: for all finite sets  $E, F \subset \N_+$, and $G \subset E$, we have that
\begin{align*}
&\int_{(\cu_{m+1})^{E \cup F}}\E\Ll[\int_{\cu_m} (\nabla \psi_m^{\rho, F}  \cdot \a^E \nabla \psi^E_m -  \nabla\psi_m^{\rho,F} \cdot q) \, \d (\mu + \mu_G)  \Rr] = 0, \\
&\int_{(\cu_{m+1})^{E \cup F}}\E\Ll[\int_{\cu_m} (\nabla \psi_m^{F}  \cdot \a^{\rho,E} \nabla \psi^{\rho,E}_m -  \nabla\psi_m^{F} \cdot q)\, \d (\mu + \mu_G) \Rr]=0.
\end{align*} 
\end{remark}

\begin{remark}\label{rmk:Combination}
We point out that, choosing $\rho=0$ in \eqref{eq:Harmonic1}, we recover the same identity with $\psi_m^{\rho,F}$ replaced by $\psi_m^F$. From this, we may also change the density of the distribution of $\mu$ from $\rho_0$ to $\rho_0 +\rho$, and obtain the analogue of \eqref{eq:Harmonic2} with $\psi_m^{F}$ replaced by $\psi_m^{\rho,F}$.

Finally, we note that, by linearity, \eqref{eq:Harmonic1} and \eqref{eq:Harmonic2} are also true if we use test functions of the form $D_F \psi_m^{G\setminus F}$ or $D_F \psi_m^{\rho, G\setminus F}$, $F, G \subset \N$ and we replace the outer integrals by $\int_{(\cu_{m+1})^{E \cup F \cup G}}$.
\end{remark}

\section{First-order differentiability}\label{s.C11}

In this section we prove Theorem~\ref{t.C11}. We explain at first its main ingredient, and it also gives us the opportunity to exemplify the sort of arguments that will be generalized later to obtain Theorem~\ref{t.smooth}. 

We recall that we have fixed a vector $q \in \Rd$, and that $\psi_m$ denotes the optimizer in the  definition \eqref{eq:defNu} of the quantity $\nu_*(\cu_m, q, \rho_0)$. We use the notation $\lesssim$ for $\leq C \times$ with the constant $C$ depending only $d$, $\Lambda$ and the length of the vector $q \in \R^d$.

The quantity that we will study is the difference between the diffusion coefficients at different densities
\begin{align}\label{eq:defPer}
\Delta^\rho(\rho_0)  := q \cdot \ab^{-1}(\rho_0 + \rho) q - q \cdot \ab^{-1}(\rho_0) q,
\end{align}
as well as, for each $m\in \N$, its finite-volume analogue
\begin{align}\label{eq:defPerApp}
 \Delta^\rho_m(\rho_0) := q \cdot  \ab_*^{-1}(\cu_m, \rho_0 + \rho) q - q \cdot  \ab_*^{-1}(\cu_m, \rho_0) q.
\end{align}

In order to prove Theorem \ref{t.C11}, we first establish its finite-volume version for~$\Delta^\rho_m(\rho_0)$, with estimates that hold uniformly over $m$, and then pass to the limit. We recall that the notation $\psi_m^{\{1\}}, \aone$ is defined in \eqref{e.def.fE}.

\begin{proposition}\label{p.C11.m}
For any $\rho, \rho_0 > 0$ fixed, it holds
\begin{align}\label{limit.increments}
\Delta^\rho(\rho_0) = \lim_{m \to \infty} \Delta^\rho_m(\rho_0).
\end{align}
Moreover, uniformly over $m\in \N$, we have
\begin{align}\label{expansion.11.m}
| \Delta^\rho_m(\rho_0) -  c_{1,m}(\rho_0) \rho|  \lesssim  \rho^2,
\end{align}
with
\begin{align}\label{c_1.m}
&c_{1,m}(\rho_0) := \int_{\R^d} \E \Ll[ \frac{1}{\rho_0 |\cu_m|} \int_{\cu_m} \nabla \psi_m \cdot (\a - \aone)\nabla \psi_m^{\{1\}} \, \d \mu \Rr] \, \d x_1.
\end{align}
\end{proposition}

The proof of Proposition \ref{p.C11.m} relies on two ingredients. The first is the following representation formula for the difference term $\Delta_m^\rho(\rho_0)$.
\begin{lemma}
\label{l.increment}
For every $m \in \N$ and $\rho > 0$, we have
\begin{align}\label{incr.rep.1}
\Delta_m^\rho(\rho_0) = \E\Ll[ \frac{1}{\rho_0\vert \cu_m \vert} \int_{\cu_m} \nabla \psi_m \cdot(\a - \a^{\rho})\nabla \psi_m^\rho \, \d \mu \Rr].
\end{align}
\end{lemma}
One may find that \eqref{incr.rep.1} and \eqref{c_1.m} look quite similar, which explains that $c_{1,m}(\rho_0)$ is indeed its first order approximation. To verify this approximation, we also need some estimate, which is the second ingredient for the proof of Proposition \ref{p.C11.m}. The next lemma allows us to compare the behavior of the optimizers $\psi_m, \psi_m^\rho$ when the measures $\mu$ or $\mu+ \mu_\rho$ are perturbed by one or two additional particles. Given $E \subset \N_+$, we recall the definitions \eqref{e.def.diff} and \eqref{e.def.DE} for the finite difference $D_E$, and the notation for the integrals $\int_{U^E}$ in \eqref{eq:SymbolInte} .
\begin{lemma}
\label{l.corrector.1} 
For every $m\in \N$ and $E, F \subset \N_+$ with $|E| \leq 2$ and $|F| \leq 1$, we have 
\begin{align}\label{corrector.1.a}
\int_{(\R^d)^{E}}\E \Ll[ \frac{1}{\rho_0\vert \cu_m \vert}\int_{\cu_m} |\nabla D_{E}\psi_m |^2 \, \d \mu  \Rr]  \lesssim 1,\\
\E \Ll[ \frac{1}{\rho_0\vert \cu_m \vert}\int_{\cu_{m}} \Ll|\int_{(\R^d)^F} \nabla D_{F}\psi_m \, \Rr|^2  \, \d\mu \Rr] \lesssim 1. \label{corrector.1.b}
\end{align}
\end{lemma}
\begin{remark}  
Applying Lemma~\ref{l.corrector.1} for an underlying particle density of $\rho_0 + \rho$ instead of $\rho$, we see that the same estimates as in \eqref{corrector.1.a} and \eqref{corrector.1.b} hold if we replace $\psi_m$ by $\psi_m^\rho$ and $\mu$ by $\mu + \mu_\rho$. 
\end{remark}

Proposition~\ref{p.C11.m} and Lemma~\ref{l.corrector.1} will be generalized in Section \ref{s.higher-order} to prove Theorem~\ref{t.smooth}, where a higher-order approximation is needed. Estimate \eqref{corrector.1.a}, indeed, corresponds to Proposition \ref{prop:KeyEstimate} with $|F| \leq 2, G= \emptyset$, while \eqref{corrector.1.b} corresponds to $F = G$ with $|F| =1$.

We organize the remainder of this section as follows. We finish the introduction with a lemma gathering some basic properties of Poisson point processes. In Subsection~\ref{sub.C11.a}, we prove Lemmas \ref{l.increment} and \ref{l.corrector.1}. Then we devote Subsection \ref{sub.C11.b} to the proof of the key result Proposition \ref{p.C11.m}. Subsection \ref{sub.C11.c} builds upon Proposition \ref{p.C11.m} to conclude for the validity of Theorem \ref{t.C11}.

\smallskip

Here are some basic estimates for Poisson point processes that we extensively use in the arguments of this section.
\begin{lemma}\label{l.aux}
Let $\rho \in (0, +\infty)$. For every measurable $F: \mathcal{M}_\delta(\R^d) \to \R$ such that $\E_\rho \Ll[\, |F|\, \Rr] < +\infty$, $z \in \R^d$, and finite set $E \subset \N_+$, we have
\begin{equation}
\begin{aligned}\label{l.aux.indicator} 
\E_\rho \Ll[ F(\mu_\rho) \, \1_{\{\mu_\rho(\cu + z) = 1 \}} \Rr] &= \rho \int_{\cu +z} \E_\rho \Ll[  F^{\{1\}}(\mu_\rho) \, \1_{\{\mu_\rho(\cu +z) = 0\}}  \Rr] \, \d x_1,\\
|\E_\rho \Ll[ F(\mu_\rho) \ \1_{\{\mu_\rho(\cu +z ) \geq |E| \}} \Rr] | &\leq \rho^{|E|} \int_{(\cu +z)^{E}} \E_\rho \bigl[ \, |F^{E}(\mu_\rho)| \, \bigr].
\end{aligned}
\end{equation}

\smallskip

For every measurable function $H : \mathcal{M}_{\delta}(\R^d) \times \R^d \to \R$ satisfying the integrability condition
$\E_\rho\Ll[ \int_U |H(\mu_\rho, x)|\, \d\mu_\rho(x) \Rr]< +\infty$,
we have Mecke's identity (c.f also \cite[Chapter 1]{last2016stochastic})
\begin{align}\label{mecke}
\E_\rho \Ll [  \frac{1}{\rho \, |U|} \int_{U} H( \mu_\rho , x) \, \d \mu_\rho(x) \Rr] = \fint_{U}  \E_\rho \Ll [  H(\mu_\rho + \delta_x, x) \Rr] \, \d x.
\end{align}
\end{lemma}

\begin{proof}[Proof of Lemma \ref{l.aux}]
Without loss of generality, we may fix $z =0$ in \eqref{l.aux.indicator}. The first identity there follows immediately if we spell out the definition of the expectation $\E_\rho$ and use the independence of increments of the Poisson point process:
\begin{align*}
\E_\rho \Ll[ F(\mu_\rho) \, \1_{\{\mu_\rho(\cu) = 1\}} \Rr]& =   \E_\rho \bigl[ e^{-\rho}\rho \int_{\cu} F( \delta_{x_1} +  \mu_\rho \mres (\cu)^c)  \, \d x_1 \bigr] \\
& = \rho \int_{\cu} \E_\rho \bigl[ \1_{\{\mu_\rho (\cu) = 0\}}F( \delta_{x_1} +  {\mu_\rho})  \bigr]  \, \d x_1.
\end{align*}
For the second estimate in \eqref{l.aux.indicator}, we write $n := |E|$ and observe that
\begin{align*}
\E_\rho \Ll[ F(\mu_\rho) \, \1_{\mu_\rho(\cu) \geq |E|} \Rr]  = \E_\rho \Ll[ e^{-\rho}  \sum_{k=n}^\infty \frac{\rho^k}{k!} \int_{(\cu)^k} F( \sum_{i=1}^k \delta_{x_i} + \mu_\rho\mres(\cu)^c) \, \d x_1 \cdots \, \d x_k \Rr].
\end{align*}
This allows us to bound
 \begin{align*}
|& \E_\rho \Ll[ F(\mu_\rho) \, \1_{\{\mu_\rho(\cu) \geq |E|\}} \Rr]|\leq \rho^n  \, \times\\
& \int_{\cu^n}  \E_\rho \Ll[   \sum_{k=n}^\infty e^{-\rho}\frac{\rho^{k-n}}{(k-n)!} \int_{(\cu)^{k-n}} |F( \sum_{j=1}^n \delta_{y_j} +   \sum_{i=1}^{k-n} \delta_{x_i}  + \mu_\rho\mres \cu^c)| \, \d x_1 \cdots \, \d x_{k-n} \Rr] \d y_1 \cdots \, \d y_{n}.
\end{align*}
which is the second estimate in \eqref{l.aux.indicator}.
Finally, \eqref{mecke} may be obtained from the definition of $\E_\rho$ and the invariance of $H$ under permutations of the atoms in $\mu$.
\end{proof}

\subsection{Representation formula and corrector estimates }\label{sub.C11.a} In this subsection we prove Lemmas \ref{l.increment} and \ref{l.corrector.1}.

\begin{proof}[Proof of Lemma \ref{l.increment}]
We use the definition of $\Delta_m^\rho(\rho_0)$ and \eqref{e.energy.identities} to write
\begin{align*}
\Delta^\rho_m(\rho_0) = \E \Ll[ \frac{1}{(\rho_0 + \rho) \vert \cu_m \vert} \int_{\cu_m}   q \cdot \nabla \psi^\rho_m  \, \d (\mu+\mu_\rho)\Rr]  - \E \Ll[ \frac{1}{\rho_0  \vert \cu_m \vert} \int_{\cu_m}   q \cdot \nabla \psi_m \, \d \mu\Rr].
\end{align*}
Identity \eqref{mecke} of Lemma \ref{l.aux} applied to $\nabla \psi_m^\rho$, first with density $\rho$ (with respect to to $\mu_\rho$) and then $\rho_0$ (with respect to $\mu$), yields that
\begin{align*}
 \E \Ll[ \frac{1}{(\rho_0 + \rho) \vert \cu_m \vert} \int_{\cu_m}  q \cdot \nabla \psi^\rho_m  \, \d \mu_\rho\Rr] =  \E \Ll[ \frac{\rho}{\rho_0(\rho_0 + \rho) \vert \cu_m \vert} \int_{\cu_m}  q \cdot \nabla \psi^\rho_m  \, \d \mu\Rr],
\end{align*}
and, hence, also
\begin{align}\label{decomp.delta}
\Delta^\rho_m(\rho_0) = \E \Ll[ \frac{1}{\rho_0 \vert \cu_m \vert} \int_{\cu_m}   q \cdot (\nabla \psi^\rho_m- \nabla \psi_m) \, \d \mu\Rr].
\end{align}
To establish representation \eqref{incr.rep.1} it now only remains {to apply \eqref{eq:Harmonic2} in Proposition~\ref{cor:var.formulation} and \eqref{eq:Harmonic1} with the choice  $E = F = \emptyset$}. 
\end{proof}

\begin{proof}[Proof of Lemma \ref{l.corrector.1}]
We start by noting that if $E = F= \emptyset$, then the inequalities of Lemma \ref{l.corrector.1} correspond to the basic energy estimate in \eqref{energy.basic}. Hence, we only need to focus on the cases $E, F \neq \emptyset$. With no loss of generality, we prove \eqref{corrector.1.a} with  $E= \{1\}$ or $E=\{1,2 \}$ and  \eqref{corrector.1.b} with $F= \{1\}$.
We also stress that, since by construction the maximizer $\psi_m$ is $\mathcal{F}_{Q_{3^m+1}}$-measurable, for every non-empty subset $G \subset \N$ we have that $D_G \psi_m$ vanishes whenever one of the particles $\{ x_j \}_{j\in G}$ does not belong to $Q_{3^m+1}$ (c.f. Definitions \eqref{e.def.DE} and \eqref{e.def.diff}). This implies that in  \eqref{corrector.1.a}-\eqref{corrector.1.b} we may replace the integrals over $\R^d$ by integrals over any set $U \supseteq Q_{3^m+1}$. In line with the notation of Section \ref{s.prelim}, throughout the proof we fix $U = \cu_{m+1}$.

\smallskip

We start with \eqref{corrector.1.a} when $E= \{1\}$. In view of the previous remarks and spelling out the integrand, this may be rewritten as 
\begin{equation}\label{corrector.1.a.2}
\begin{aligned}
&\int_{\cu_{m+1}}\E\Ll[\frac{1}{\rho_0\vert \cu_m \vert} \int_{\cu_m} |\nabla (\psi_m^{\{ 1\}} - \psi_m)|^2 \d \mu  \Rr] \d x_1\lesssim 1,\\
\end{aligned}
\end{equation}
We consider identity \eqref{eq:Harmonic1} of Proposition \ref{cor:var.formulation}  with $\rho = 0$, $E = \emptyset $ and $E = \{1\}$ and with test function $D_{\{1\}}\psi_m$ (c.f.\ Remark \ref{rmk:Combination}). 
\begin{align*}
\int_{\cu_{m+1}}\E\Ll[\int_{\cu_m} (\nabla D_{\{1\}}\psi_m  \cdot \a \nabla \psi_m -  \nabla D_{\{1\}}\psi_m \cdot q) \, \d \mu  \Rr] = 0,\\
\int_{\cu_{m+1}}\E\Ll[\int_{\cu_m} (\nabla D_{\{1\}}\psi_m  \cdot \a^{\{1\}} \nabla \psi_m^{\{1\}} -  \nabla D_{\{1\}}\psi_m \cdot q) \, \d \mu  \Rr] = 0.
\end{align*}
Subtracting the resulting identities yields that
\begin{equation}\label{corrector.est.1}
\begin{aligned}
\int_{\cu_{m+1}}\E \bigl[\frac{1}{\rho_0  \vert \cu_m \vert} \int_{\cu_m} \nabla D_{\{1\}}\psi_m \cdot \a^{\{1\}}& \nabla D_{\{1\}}\psi_m \, \d\mu \bigr] \, \d x_1 \\
 = -\int_{\cu_{m+1}}\E\bigl[\frac{1}{\rho_0  \vert \cu_m \vert}& \int_{\cu_m}\nabla D_{\{1\}}\psi_m \cdot (D_{\{1\}}\a) \, \nabla \psi_m \, \d \mu \bigr] \, \d x_1.
\end{aligned}
\end{equation}
We now appeal to the uniform ellipticity assumption \eqref{a.elliptic} and the Cauchy--Schwarz inequality to infer from \eqref{corrector.est.1} that
\begin{equation*}
\begin{aligned}
 \int_{\cu_{m+1}}\E\bigl[\frac{1}{\rho_0  \vert \cu_m \vert} \int_{\cu_m} &|\nabla D_{\{1\}}\psi_m|^2 \, \d \mu \bigr] \, \d x_1\\
& {\lesssim} \int_{\cu_{m+1}}\E\bigl[\frac{1}{\rho_0  \vert \cu_m \vert} \int_{\cu_m} (D_{\{1\}}\a)^2 |\nabla \psi_m|^2 \d \mu \bigr] \, \d x_1.
\end{aligned}
\end{equation*}
We obtain the first inequality in \eqref{corrector.1.a.2} after noting that \eqref{a.local} and \eqref{a.elliptic} for $\a$ imply that also
\begin{align*}
\int_{\cu_{m+1}}\E&\bigl[\frac{1}{\rho_0  \vert \cu_m \vert} \int_{\cu_m} (D_{\{1\}}\a)^2 |\nabla \psi_m|^2 \d \mu \bigr] \, \d x_1\\
&\lesssim \E\bigl[\frac{1}{\rho_0  \vert \cu_m \vert} \int_{\cu_m} \int_{\cu +z} |\nabla \psi_m(\mu, z)|^2 \, \d x_1 \, \d \mu(z) \bigr] \stackrel{\eqref{energy.basic}}{\lesssim}1.
\end{align*}
This establishes \eqref{corrector.1.a.2}. 

\smallskip

The proof of \eqref{corrector.1.a} when $E= \{ 1, 2\}$ follows a similar argument. Observe that 
\begin{align*}
D_{\{1,2\}}\psi_m = \psi_m^{\{1,2\}} - \psi_m^{\{1\}} - \psi_m^{\{2\}} + \psi_m,    
\end{align*}
we may add and subtract suitable combinations of identity~\eqref{eq:Harmonic1} in Proposition \ref{cor:var.formulation} with $E\in \{ \emptyset, \{1\}, \{2\}, \{1,2\} \}$ and test function $\nabla D_{\{1,2\}}\psi_m$ (c.f. Remark \ref{rmk:Combination}) to infer that 
\begin{equation}\label{corrector.est.2}
\begin{aligned}
&\int_{(\cu_{m+1})^2}\hspace{-0.1cm}\E\biggl[\frac{1}{\rho_0  \vert \cu_m \vert} \int_{\cu_m} \nabla D_{\{1,2\}}\psi_m \cdot \a^{\{1, 2\}} \nabla D_{\{1,2\}} \psi_m \, \d \mu  \biggr] \, \d x_1 \, \d x_2 \\
& = \sum_{i=1}^2 \int_{(\cu_{m+1})^2}\E\biggl[\frac{1}{\rho_0  \vert \cu_m \vert} \int_{\cu_m}\nabla D_{\{1,2\}}\psi_m \cdot (\a^{\{i\}}- \a^{\{1, 2\}}) \nabla D_{\{i\}} \psi_m  \d \mu \biggr] \d x_1 \, \d x_2\\
&\quad\quad\quad \quad-\int_{(\cu_{m+1})^2} \E \biggl[\frac{1}{\rho_0  \vert \cu_m \vert} \int_{\cu_m}\nabla D_{\{1,2\}}\psi_m \cdot (D_{\{1,2\}}\a) \, \nabla \psi_m  \d\mu \biggr] \d x_1 \, \d x_2.
\end{aligned} 
\end{equation}
From this, we obtain \eqref{corrector.1.a} when $E= \{ 1, 2\}$ as was done in the case $E=\{1\}$. This time, besides the Cauchy--Schwarz inequality and \eqref{a.elliptic}-\eqref{a.local}, we also rely on inequality~\eqref{corrector.1.a.2} for $|E|=1$ that was proved above.

\smallskip

To conclude the proof of this lemma, it remains to establish inequality \eqref{corrector.1.b}. As argued at the beginning of the proof of this lemma, this can be reduced to
\begin{equation}\label{corrector.1.b.2}
\begin{aligned}
\E \Ll[\frac{1}{\rho_0\vert \cu_m \vert} \int_{\cu_{m}} |\int_{(\cu_{m+1})^{\{1\}}} \nabla D_{\{1\}}\psi_m \, |^2  \d\mu \Rr] \lesssim 1.
\end{aligned}
\end{equation}
We appeal again to Proposition \ref{cor:var.formulation}: we subtract \eqref{eq:Harmonic1} with $E=\emptyset$ and test function $D_{\{2\}}\psi_m$ (c.f. Remark \ref{rmk:Combination}) from the same identity with $E= \{1\}$ and test function $D_{\{2\}}\psi_m$. This yields
\begin{equation*}
\begin{aligned}
\int_{(\cu_{m+1})^{\{1,2\}}}\E \bigl[&\frac{1}{\rho_0  \vert \cu_m \vert} \int_{\cu_m}  \nabla D_{\{2\}}\psi_m \cdot \a \,  \nabla D_{\{1\}}\psi_m \, \d \mu \bigr] \\
& = -\int_{(\cu_{m+1})^{\{1,2\}}} \E\bigl[\frac{1}{\rho_0  \vert \cu_m \vert} \int_{\cu_m} \nabla D_{\{2\}}\psi_m \cdot (D_{\{1\}}\a) \, \nabla \psi_m^{\{ 1\}} \d \mu \bigr].
\end{aligned}
\end{equation*}
Appealing to Fubini's theorem and observing that, by a simple relabelling of the integration variable, it holds that $\int_{(\cu_{m+1})^{\{1\}}} D_{\{1\}}\psi_m = \int_{(\cu_{m+1})^{\{2\}}} D_{\{2\}}\psi_m$, we infer that
\begin{equation*}
\begin{aligned}
\E \bigl[\frac{1}{\rho_0  \vert \cu_m \vert}& \int_{\cu_m} \nabla\bigl( \int_{(\cu_{m+1})^{\{1\}}}D_{\{1\}}\psi_m\bigr) \cdot \a \, \nabla\bigl( \int_{(\cu_{m+1})^{\{1\}}}D_{\{1\}}\psi_m\bigr) \, \d \mu \bigr]\\
& = -\E\bigl[\frac{1}{\rho_0  \vert \cu_m \vert} \int_{\cu_m} \nabla\bigl( \int_{(\cu_{m+1})^{\{1\}}}D_{\{1\}}\psi_m\bigr) \cdot \bigl(\int_{/\cu_{m+1})^{\{1\}}} (D_{\{1\}}\a) \, \nabla \psi_m^{\{ 1\}}\bigr) \, \d \mu \bigr].
\end{aligned}
\end{equation*}
By \eqref{a.elliptic} and the Cauchy--Schwarz inequality, this also implies that 
\begin{align*}
\E\bigl[\frac{1}{\rho_0  \vert \cu_m \vert}& \int_{\cu_m}  |\int_{(\cu_{m+1})^{\{1\}}}\nabla D_{\{1\}}\psi_m  \,  |^2  \, \d \mu \bigr]\\
&\leq \E \bigl[\frac{1}{\rho_0  \vert \cu_m \vert} \int_{\cu_m} |\int_{(\cu_{m+1})^{\{1\}}} (D_{\{1\}}\a) \nabla \psi_m^{\{ 1\}} \, |^2 \d \mu \bigr].
\end{align*}
We thus conclude the proof of \eqref{corrector.1.b.2} provided that the term on the right-hand side above is $\lesssim 1$: by the triangle inequality we have that
\begin{align*}
\E \bigl[\frac{1}{\rho_0  \vert \cu_m \vert}& \int_{\cu_m} |\int_{(\cu_{m+1})^{\{1\}}} (D_{\{1\}}\a) \nabla \psi_m^{\{ 1\}} \,  |^2 \d \mu \bigr]\\
&\leq  \E \bigl[\frac{1}{\rho_0  \vert \cu_m \vert} \int_{\cu_m} |\int_{(\cu_{m+1})^{\{1\}}} (D_{\{1\}}\a) \nabla D_{\{1\}} \psi_m \, |^2 \d \mu \bigr] \\
& \hspace{3cm}+ \E \bigl[\frac{1}{\rho_0  \vert \cu_m \vert} \int_{\cu_m} |\int_{(\cu_{m+1})^{\{1\}}} (\a- \a^{\{1\}}) \, |^2 |\nabla \psi_m|^2 \d \mu \bigr].
\end{align*}
The second term on the right-hand side is immediately bounded by $\lesssim 1$ due to assumptions \eqref{a.elliptic}-\eqref{a.local} on $\a$ and \eqref{energy.basic}. The first term admits the same upper bound thanks to the Cauchy--Schwarz inequality, \eqref{a.elliptic}-\eqref{a.local}, and  \eqref{corrector.1.a.2}. The proof of Lemma \ref{l.corrector.1} is complete.
\end{proof}

\subsection{Proof of Proposition \ref{p.C11.m}}\label{sub.C11.b}
In this section we use Lemmas \ref{l.increment} and \ref{l.corrector.1} to show Proposition \ref{p.C11.m}.

\begin{proof}[Proof of Proposition \ref{p.C11.m}]
Limit \eqref{limit.increments} follows immediately from definitions \eqref{eq:defPer}, \eqref{eq:defPerApp} and \eqref{eq:defab}. We thus turn to \eqref{expansion.11.m} and prove this inequality in three different steps. 

\smallskip

\textit{Step 1.} We claim that 
\begin{equation}\label{11.m.step1}
\begin{aligned}
\bigl|& \Delta_m^\rho(\rho_0) - \rho c_{1,m}(\rho_0)\bigr| \lesssim   \rho^2\\
& + \rho \Ll| \int_{\cu_{m+1}}\E \Ll[ \frac{1}{\rho_0  \vert \cu_m \vert} \int_{\cu_m} \nabla \psi_m \cdot \bigl(D_{\{1\}}\a \bigr) (\nabla \psi_m^{\rho, \{1\}}- \nabla \psi_m^{\{ 1\}}) \d \mu \Rr] \, \d x_1 \Rr|.
\end{aligned}
\end{equation}

We begin by using the representation formula for $\Delta_m^\rho(\rho_0)$ of Lemma \ref{l.increment}, the definition of the expectation $\E$ and assumption \eqref{a.local} for $\a$ to rewrite
\begin{equation}
\begin{aligned}\label{11.m.step1.a}
 \Delta_m^\rho(\rho_0) &=  \E \Ll[ \frac{1}{\rho_0  \vert \cu_m \vert} \int_{\cu_m}\nabla \psi_m \cdot \E_\rho \Ll[ \1_{\mu_\rho(\cu + z) \geq 1}(\a - \a^\rho) \nabla \psi_m^\rho \Rr] \d \mu(z) \Rr]\\
 & = \E \Ll[ \frac{1}{\rho_0  \vert \cu_m \vert} \int_{\cu_m}\nabla \psi_m \cdot \E_\rho \Ll[ \1_{\mu_\rho(\cu + z) =1}(\a - \a^\rho) \nabla \psi_m^\rho \Rr] \d \mu(z) \Rr]\\
 & \quad \quad + \E \Ll[ \frac{1}{\rho_0  \vert \cu_m \vert} \int_{\cu_m}\nabla \psi_m \cdot \E_\rho \Ll[ \1_{\mu_\rho(\cu + z) \geq 2}(\a - \a^\rho) \nabla \psi_m^\rho \Rr] \d \mu(z) \Rr].
 \end{aligned}
 \end{equation}
We claim that the second term on the right-hand side above is bounded by a constant multiple of $\rho^2$: using the Cauchy--Schwarz inequality, the bound \eqref{energy.basic}, and the second inequality in~\eqref{l.aux.indicator} of Lemma~\ref{l.aux} with $E= \{1, 2\}$, we infer that
\begin{equation}\label{more.particles.estimate}
\begin{aligned}
\biggl|\E \biggl[ \frac{1}{\rho_0  \vert \cu_m \vert} \int_{\cu_m}\nabla \psi_m& \cdot \E_\rho \Ll[ \1_{\mu_\rho(\cu + z) \geq 2}(\a - \a^\rho) \nabla \psi_m^\rho \Rr] \d \mu \biggr]\biggr|\\
&\lesssim \rho^2 \E \Ll[ \frac{1}{\rho_0 \vert \cu_m \vert} \int_{\cu_m} \bigl(\int_{(\cu +z)^{2}} |\nabla \psi_m^{\rho, \{1, 2\}}|^2 \, \d x_1 \, \d x_2\bigr) \d \mu(z) \Rr],
\end{aligned}
\end{equation}
 and the right-hand side is $\lesssim \rho^2$, as one can see by writing
 \begin{align*}
 \psi_m^{\rho,\{1,2\}} = D_{\{1,2\}}\psi_m^{\rho} - D_{\{1\}}\psi_m^{\rho} - D_{\{2\}}\psi_m^{\rho} + \psi_m^{\rho},    
 \end{align*}
 and then applying the triangle inequality and Lemma \ref{l.corrector.1}. Inserting this into \eqref{11.m.step1.a}, we have that
 \begin{equation}
\begin{aligned}\label{11.m.step1.b}
\Ll| \Delta_m^\rho(\rho_0) - \E \Ll[ \frac{1}{\rho_0  \vert \cu_m \vert} \int_{\cu_m}\nabla \psi_m \cdot \E_\rho \Ll[ \1_{\mu_\rho(\cu + z) =1}(\a - \a^\rho) \nabla \psi_m^\rho \Rr] \d \mu(z) \Rr] \Rr|  \lesssim \rho^2.
  \end{aligned}
 \end{equation}
 We now apply the first inequality of \eqref{l.aux.indicator} to the inner expectation in the term on the left-hand side above. This, together with the locality \eqref{a.local} of $\a$, yields that
 \begin{align*}
 &\E \Ll[ \frac{1}{\rho_0  \vert \cu_m \vert} \int_{\cu_m}\nabla \psi_m \cdot \E_\rho \Ll[ \1_{\mu_\rho(\cu + z) =1}(\a - \a^\rho) \nabla \psi_m^\rho \Rr] \d \mu(z) \Rr]\\
 & = \rho \E \Ll[ \frac{1}{\rho_0  \vert \cu_m \vert} \int_{\cu_m}\nabla \psi_m \cdot \bigl(\int_{\cu + z} \E_\rho \Ll[ \1_{\mu_\rho(\cu + z) =0}(\a - \a^{\{1\}}) \nabla \psi_m^{\rho,1} \Rr] \, \d x_1 \bigr)\d \mu(z) \Rr]\\
 & \stackrel{\eqref{c_1.m}}{=}  \rho\, c_{m,1} - \rho \, \E \Ll[ \frac{1}{\rho_0  \vert \cu_m \vert} \int_{\cu_m}\nabla \psi_m \cdot \bigl(\int_{\cu + z} (D_{\{1\}}\a) (\nabla \psi_m^{\rho,\{1\}} - \nabla \psi_m^{\{1\}}) \, \d x_1 \bigr)\d \mu(z) \Rr]\\
 &+  \rho \, \E \Ll[ \frac{1}{\rho_0  \vert \cu_m \vert} \int_{\cu_m}\nabla \psi_m \cdot \bigl(\int_{\cu + z} \E_\rho \Ll[ \1_{\mu_\rho(\cu + z) \geq 1}(D_{\{1\}}\a) \nabla \psi_m^{\rho,\{1\}} \Rr] \, \d x_1 \bigr)\d \mu(z) \Rr].
 \end{align*}
 To conclude from this and \eqref{11.m.step1.b} that inequality \eqref{11.m.step1} holds, it remains to prove that the last term above is $\lesssim \rho^2$. This may be done using again the second inequality in \eqref{l.aux.indicator} and Lemma \ref{l.corrector.1}, as done for the term in \eqref{more.particles.estimate}. 
 
  \medskip

\textit{Step 2.} We now argue that the term appearing on the right-hand side of~\eqref{11.m.step1} may be rewritten as follows: 
\begin{equation}
\begin{aligned}\label{11.m.step2}
&\int_{\cu_{m+1}} \E \biggl[ \frac{1}{\rho_0  \vert \cu_m \vert}  \nabla \psi_m \cdot (D_{\{1\}}\a) (\nabla \psi_m^{\rho, \{1\}} - \nabla \psi_m^{\{ 1\}})  \d \mu \biggr] \, d x_1 \\
 & \quad =  \int_{\cu_{m+1}}  \E \biggl[\frac{1}{\rho_0  \vert \cu_m \vert}\int_{\cu_m} \nabla D_{\{1\}}\psi_m  \cdot (\a^{\rho, \{1\}} - \a^{\{1\}})\nabla\psi_m^{\rho, \{1\}}\, \d \mu \biggr] \, \d x_1.
\end{aligned}
\end{equation}
We appeal to Proposition~\ref{cor:var.formulation}: we subtract \eqref{eq:Harmonic1} with $E= F= \{1 \}$ from \eqref{eq:Harmonic1} with $E = \emptyset$ and $F= \{1 \}$. This, together with the symmetry of $\a$, yields that
\begin{align*}
-\int_{\cu_{m+1}}\E \bigl[\frac{1}{\rho_0  \vert \cu_m \vert}& \int_{\cu_m}   \nabla \psi_m  \cdot (D_{\{1\}}\a) \nabla \psi_m^{\rho,\{1\}} \, \d \mu  \bigr] \d x_1 \\
&= \int_{\cu_{m+1}} \E \bigl[\frac{1}{\rho_0  \vert \cu_m \vert} \int_{\cu_m}  \nabla \psi_m^{\rho,\{1\}}  \cdot \a^{\{1\}} \nabla D_{\{1\}}\psi_m \, \d \mu   \bigr] \d x_1.\\
\end{align*} 
We now subtract this inequality from the same one with $\rho=0$ (see also the discussion in Remark \ref{rmk:Combination}) and conclude that
\begin{multline}\label{eq:11.mm.step2.A}
\int_{\cu_{m+1}} \E\bigl[\frac{1}{\rho_0  \vert \cu_m \vert} \int_{\cu_m}  \nabla \psi_m\cdot (D_{\{1\}}\a)(\nabla\psi_m^{\rho, \{1\}} - \nabla \psi_m^{\{ 1\}})  \, \d \mu  \bigr] \, \d x_1 \\
= \int_{\cu_{m+1}}  \E \bigl[\frac{1}{\rho_0  \vert \cu_m \vert} \int_{\cu_m} (\nabla \psi_m^{\{ 1\}} - \nabla\psi_m^{\rho, \{1\}})  \cdot \a^{\{1\}} \nabla D_{\{1\}}\psi_m \, \d \mu   \bigr]  \, \d x_1.
\end{multline}
We now treat the term on the right-hand side above in an analogous way. We consider \eqref{eq:Harmonic1} and \eqref{eq:Harmonic2} in Proposition \ref{cor:var.formulation} with $E=\{1\}$ and test function $D_{\{1\}}\psi_m$ (this is possible by Remark \ref{rmk:Combination}). This yields
\begin{multline}\label{eq:11.mm.step2.B}
 \int_{\cu_{m+1}}\E \bigl[\frac{1}{\rho_0  \vert \cu_m \vert} \int_{\cu_m}  \nabla D_{\{1\}}\psi_m \cdot \a^{\{1\}} (\nabla \psi_m^{\{ 1\}} -  \nabla\psi_m^{\rho, \{1\}}) \, \d \mu   \bigr] \d x_1 \\
 = \int_{\cu_{m+1}}\E \bigl[\frac{1}{\rho_0  \vert \cu_m \vert}\int_{\cu_m} \nabla D_{\{1\}} \psi_m  \cdot (\a^{\rho, \{1\}} - \a^{\{1\}})\nabla\psi_m^{\rho, \{1\}} \, \d \mu   \bigr]  \d x_1.
 \end{multline}
We compare the two displays \eqref{eq:11.mm.step2.A} and \eqref{eq:11.mm.step2.B}, which give \eqref{11.m.step2} and thus conclude the proof of Step~2.

\medskip

\textit{Step 3.} In this step, we give an estimate that
\begin{equation}
\begin{aligned}\label{11.m.step3}
\biggl| \int_{\cu_{m+1}} \E\biggl[ \frac{1}{\rho_0  \vert \cu_m \vert}  \int_{\cu + z}\nabla \psi_m \cdot (D_{\{1\}}\a)(\nabla \psi_m^{\rho, \{1\}} - \nabla \psi_m^{\{ 1\}})   \, &\d \mu \biggr] \, \d x_1 \biggr| \lesssim \rho.
\end{aligned}
\end{equation}
This, together with the result \eqref{11.m.step1} of Step 1, will establish Proposition \ref{p.C11.m}.

\smallskip

Appealing to Step 2, the proof of this step can be reduced to establishing that
\begin{equation}\label{11.m.step3.a}
\begin{aligned}
\biggl| \int_{\cu_{m+1}}  \E \bigl[\frac{1}{\rho_0  \vert \cu_m \vert}\int_{\cu_m} \nabla D_{\{1\}}\psi_m \cdot (\a^{\rho, \{1\}} - \a^{\{1\}})  \nabla\psi_m^{\rho, \{1\}}  \, \d \mu  \bigr] \, \d x_1\biggr| \lesssim \rho.
\end{aligned}
\end{equation}
We use the triangle inequality to split
\begin{equation}\label{11.m.step3.b}
\begin{aligned}
 &\biggl| \int_{\cu_{m+1}} \E \biggl[ \frac{1}{\rho_0  \vert \cu_m \vert} \int_{\cu_m} \nabla D_{\{1\}}\psi_m \cdot \bigl(\a^{\{1\}} - \a^{\rho, \{1\}}\bigr)\nabla \psi_m^{\rho, \{1\}} \d \mu  \biggr] \, \d x_1 \biggr|\\
 &\leq  \biggl| \int_{\cu_{m+1}}  \E \biggl[ \frac{1}{\rho_0  \vert \cu_m \vert} \int_{\cu_m} \nabla D_{\{1\}}\psi_m \cdot \bigl(\a^{\{1\}} - \a^{\rho, \{1\}}\bigr)\nabla D_{\{1\}}\psi_m^\rho  \d \mu  \biggr]\, \d x_1\biggr|\\
 & \quad\quad\quad \quad  +  \biggl|  \int_{\cu_{m+1}} \E \biggl[ \frac{1}{\rho_0  \vert \cu_m \vert} \int_{\cu_m} \nabla D_{\{1\}}\psi_m \cdot \bigl(\a^{\{1\}} - \a^{\rho, \{1\}}\bigr)\nabla \psi_m^{\rho}  \d \mu \biggr]\, \d x_1\biggr|,
\end{aligned}
\end{equation}
and treat separately the two integrals above.

\smallskip

We begin with the first one and argue similarly to \eqref{11.m.step1.a} of Step 1: we use \eqref{a.local},  Lemma \ref{l.aux.indicator} and the Cauchy--Schwarz inequality to control 
\begin{align*}
  \biggl| \int_{\cu_{m+1}} & \E \biggl[ \frac{1}{\rho_0  \vert \cu_m \vert} \int_{\cu_m} \nabla D_{\{1\}}\psi_m \cdot \bigl(\a^{\{1\}} - \a^{\rho, \{1\}}\bigr)\nabla D_{\{1\}}\psi_m^\rho  \d \mu  \biggr]\, \d x_1\biggr|\\
 &\leq  \rho \biggl(\int_{\cu_{m+1}}  \E \biggl[ \frac{1}{\rho_0  \vert \cu_m \vert} \int_{\cu_m} |\nabla D_{\{1\}}\psi_m|^2  \d \mu  \biggr] \, \d x_1\biggr)^{\frac 1 2}\\
 &\quad\quad\times \biggl( \int_{\cu_{m+1}} \E \biggl[ \frac{1}{\rho_0  \vert \cu_m \vert} \int_{\cu_m}\int_{\cu + z} |\nabla D_{\{1\}}\psi_m^{\rho,2} |^2   \d x_2   \d \mu (z) \biggr] \, \d x_1 \biggr)^{\frac 1 2}\\
 &\stackrel{\text{Lemma \ref{l.corrector.1}}}{\lesssim} \rho \,  \biggl( \int_{\cu_{m+1}} \E \biggl[ \frac{1}{\rho_0  \vert \cu_m \vert} \int_{\cu_m}\int_{\cu + z} |\nabla D_{\{1\}}\psi_m^{\rho,2} |^2   \d x_2   \d \mu(z) \biggr] \, \d x_1 \biggr)^{\frac 1 2}.
\end{align*} 
Writing
\begin{align*} 
D_{\{1\}}\psi_m^{\rho,2}& = \psi_m^{\rho, \{1, 2\}}-\psi_m^{\rho, \{2\}}  = D_{\{1,2\}} \psi_m^\rho  - D_{\{1\}}\psi_m^\rho,
\end{align*}
 and appealing again to the triangle inequality and to the estimates of Lemma \ref{l.corrector.1}, we infer that
\begin{align}\label{step3.term.a}
& \biggl| \int_{\cu_{m+1}}  \E \biggl[ \frac{1}{\rho_0  \vert \cu_m \vert} \int_{\cu_m} \nabla D_{\{1\}} \psi_m \cdot \bigl(\a^{\{1\}} - \a^{\rho, \{1\}}\bigr) \nabla D_{\{1\}} \psi_m^\rho \, \d \mu \biggr]\, \d x_1\biggr| \lesssim \rho.
\end{align}

\smallskip

The second integral in \eqref{11.m.step3.b} may be treated in a similar way, if we split
\begin{align*}
&  \biggl|  \int_{\cu_{m+1}} \E \biggl[ \frac{1}{\rho_0  \vert \cu_m \vert} \int_{\cu_m} \nabla D_{\{1\}} \psi_m \cdot \bigl(\a^{\{1\}} - \a^{\rho, \{1\}}\bigr)\nabla \psi_m^{\rho} \d \mu \biggr]\, \d x_1\biggr|\\
&\leq \biggl| \E \biggl[ \frac{1}{\rho_0  \vert \cu_m \vert}\int_{\cu_m} \int_{\cu_{m+1} \setminus (\cu + z)}  \nabla D_{\{1\}} \psi_m \cdot \E_{\rho}\bigl[\bigl(\a^{\{1\}} - \a^{\rho, \{1\}}\bigr)\nabla \psi_m^{\rho} \bigr] \, \d x_1 \,  \d \mu \biggr]\biggr|\\
&\hspace{1cm} + \biggl| \E \biggl[ \frac{1}{\rho_0  \vert \cu_m \vert}\int_{\cu_m} \int_{\cu +z}  \nabla D_{\{1\}} \psi_m \cdot \E_{\rho}\bigl[ \bigl(\a^{\{1\}} - \a^{\rho, \{1\}}\bigr)\nabla \psi_m^{\rho} \bigr] \, \d x_1 \,  \d \mu \biggr]\biggr|.
\end{align*}
The second term may be bounded by $\lesssim \rho$ using again \eqref{a.local} and an argument analogous to the one used for \eqref{step3.term.a}. On the other hand, since by \eqref{a.local}, for every $z \in \cu_m$ and $x \in \cu_m \setminus (\cu + z)$ we have that
$$
\a^{\{1\}}(\mu , z)-\a^{\rho, \{1\}}(\mu, z) = \1_{\{\mu_\rho(\cu + z) \geq 1\}} \bigl(\a(\mu, z) - \a^\rho(\mu, z)\bigr),
$$
the first term on the right-hand side above may be rewritten as
\begin{align*}
&\E \biggl[ \frac{1}{\rho_0  \vert \cu_m \vert}\int_{\cu_m} \int_{\cu_{m+1} \setminus (\cu + z)}  \nabla D_{\{1\}} \psi_m \cdot \E_{\rho}\bigl[\bigl(\a^{\{1\}} - \a^{\rho, \{1\}}\bigr)\nabla \psi_m^{\rho} \bigr] \, \d x_1 \,  \d \mu \biggr]\\
&= \E \biggl[ \frac{1}{\rho_0  \vert \cu_m \vert}\int_{\cu_m} \bigl(\int_{\cu_{m+1} \setminus (\cu + z)} \nabla D_{\{1\}} \psi_m \d x_1 \bigr) \cdot \E_{\rho}\bigl[\1_{\{\mu_\rho(\cu + z) \geq 1\}}\bigl(\a- \a^{\rho}\bigr)\nabla \psi_m^{\rho} \bigr]  \d \mu \biggr].
\end{align*}
We may bound this term by $\lesssim \rho$ by appealing once again to the Cauchy--Schwarz inequality and Lemmas \ref{l.corrector.1} and \ref{l.aux.indicator}. This yields that also the second integral in~\eqref{11.m.step3.b}  is bounded by $\lesssim \rho$.  This establishes \eqref{11.m.step3.a} and concludes the proof of Step~3. Proposition \ref{p.C11.m} is therefore proved.
\end{proof}

\subsection{Proof of Theorem \ref{t.C11}}\label{sub.C11.c}
Equipped with Proposition \ref{p.C11.m}, we are now ready to prove the main result of this section.
\begin{proof}[Proof of Theorem \ref{t.C11}]
A first consequence of Proposition~\ref{p.C11.m} is that $\{ c_{1,m}(\rho_0)\}_{m\in \N}$ in~\eqref{c_1.m} is uniformly bounded over $m\in \N$. Indeed, by definition \eqref{c_1.m}, assumptions \eqref{a.elliptic}-\eqref{a.local} on $\a$ and the Cauchy--Schwarz inequality, we have that
\begin{align*}
|c_{1,m}(\rho_0)|& = \biggl |\E \Ll[ \frac{1}{\rho_0 |\cu_m|} \int_{\cu_m} \nabla \psi_m \cdot  \int_{(\cu +z)} (\aone - \a) \nabla \psi_m^{\{1\}}  \, \d x_1 \, \d \mu(z) \Rr] \biggr |\\
& \leq \E \Ll[ \frac{1}{\rho_0 |\cu_m|} \int_{\cu_m} |\nabla \psi_m|^2 \d \mu \Rr]^{\frac 1 2} \E \Ll[ \frac{1}{\rho_0 |\cu_m|} \int_{\cu_m} \int_{\cu+z} |\nabla  \psi_m^{\{ 1\}}|^2 \, \d \mu \Rr]^{\frac 1 2}.
\end{align*}
The first factor on the right-hand side above is bounded thanks to \eqref{energy.basic}. The second one can be controlled by the triangle inequality, Lemma \ref{l.corrector.1} and again \eqref{energy.basic}.

Let $\rho_0 > 0$ be fixed. The uniform bound for $\{ c_{1,m}(\rho_0)\}_{m\in \N}$ implies that we may find a subsequence (possibly depending on $\rho_0$) and a number $c^*_1(\rho_0)$ such that 
$$
\lim_{j \to +\infty }c_{1,m_j}(\rho_0) := c^*_1(\rho_0).
$$
Passing to the limit along this subsequence in the inequality \eqref{expansion.11.m} of Proposition~\ref{p.C11.m} and using \eqref{limit.increments}, we infer that for every $\rho > 0$
\begin{align}\label{unique.limit.1}
|\Delta^\rho(\rho_0) - c^*_1(\rho_0) \rho| \lesssim \rho^2.
\end{align}
On the one hand, the arbitrariness of $\rho > 0$ in this inequality implies that the value $c_1^*(\rho_0)$ is the limit for the full sequence $\{ c_{m,1}(\rho_0)\}_{m\in\N}$, which we denote by $c_1(\rho_0)$. On the other hand, definition~\eqref{eq:defPer} allows to immediately infer that for every $\rho_0 > 0$ fixed and $\rho\ge 0$ tending to zero, we have
\begin{align}\label{rhs.expansion}
q \cdot \ab^{-1}(\rho_0 + \rho)q = q \cdot \ab^{-1}(\rho_0)q + c_1(\rho_0)\rho + O(\rho^2).
\end{align}
To conclude the proof of Theorem \ref{t.C11}, it thus remains to extend \eqref{rhs.expansion} to negative values of $\rho$ that tend to zero. We do this by applying \eqref{rhs.expansion} with the pair $(\rho_0,\rho_0 + \rho)$ substituted with $(\rho_0 - \rho,\rho_0)$ to get that, as $\rho \ge 0$ tends to zero,
\begin{align}\label{lhs.expansion}
q \cdot \ab^{-1}(\rho_0 - \rho)q = q \cdot \ab^{-1}(\rho_0)q - c_1(\rho_0 - \rho)\rho + O(\rho^2).
\end{align}
To conclude the desired expansion, it remains to show that we may replace $c_1(\rho_0- \rho)$ by $c_1(\rho_0)$ in this display. Defining $f(\cdot) := q \cdot \ab^{-1}(\cdot )q$ and appealing to identity \eqref{rhs.expansion}, we write
\begin{align*}
c_1(\rho_0 - \rho)&  = \frac{f(\rho_0) -f(\rho_0- \rho)}{\rho} + O(\rho)\\
& =  \frac{f(\rho_0) -f(\rho_0 +\rho)}{\rho}  +  \frac{f(\rho_0+\rho) -f(\rho_0-\rho)}{\rho} + O(\rho)\\
& = -c_1(\rho_0) + 2 c_1(\rho_0 - \rho) + O(\rho).
\end{align*}
In the last line, we apply \eqref{rhs.expansion} at $\rho_0$ for the first term, and \eqref{rhs.expansion} at $(\rho_0 - \rho)$ for the second term wit step size $2\rho$. The notation $O(\rho)$ is valid as the hidden constant is independent from the density. The equation above gives us 
\begin{align*}
c_1(\rho_0 - \rho) = c_1(\rho_0) + O(\rho).  
\end{align*}
Inserting this into \eqref{lhs.expansion} yields that \eqref{rhs.expansion} holds also for negative perturbations $\rho$. This completes the proof of Theorem~\ref{t.C11}. 
\end{proof}


%
%
%
%
%
%
%
%
\section{Higher-order differentiability}
\label{s.higher-order}

The goal of this section is to generalize the results of the previous section, and ultimately prove Theorem~\ref{t.smooth} stating that the mapping $\rho_0 \mapsto \ahom(\rho_0)$ is infinitely differentiable. 


As a preparation, we recall that the notation $f^E$ and $D_E$ is introduced in~\eqref{e.def.fE}, \eqref{e.def.DE} and state basic algebraic properties of these operators.
\begin{proposition}[Algebraic properties]
For every ${f,g: \mmd(\Rd) \to \R}$ and every finite set $E \subset \N_+$, the following identities hold.
\begin{itemize}
\item \textit{Inclusion-exclusion formula}
\begin{align}\label{eq:Difference}
D_E f = \sum_{F \subset E} (-1)^{\vert E \setminus  F\vert} f^F.
\end{align}
\item \textit{Telescoping formula}
\begin{align}\label{eq:Telescope}
f^E = \sum_{F \subset E} D_F f.
\end{align}
\item \textit{Leibniz formulas}
\begin{align}\label{eq:Leibniz1}
D_E(fg) = \sum_{F \subset E} (D_F f)(D_{E \setminus F} g^F),
\end{align}
and
\begin{align}\label{eq:Leibniz2}
D_E(fg) = \sum_{F, G \subset E, F \cup G = E} (D_F f)(D_{G} g).
\end{align}
\end{itemize}
\end{proposition}

\begin{proof}
These elementary identities can be proved by induction. We show \eqref{eq:Leibniz1} for illustration. Without loss of generality, we can assume that $E = \dbint{1,n}$ for some integer $n \in \N_+$. The case $n = 1$ is clear:
\begin{align}\label{eq:DiffProduct}
D_{1}(fg) =f^{\{1\}} g^{\{1\}} - f g = (D_{1} f)g^{\{1\}} + f (D_{1} g).
\end{align} 
Assuming that the formula is valid for $E =\dbint{1,n}$, we can then write
\begin{align*}
D_{\dbint{1, n+1}}(fg) &= D_{n+1} (D_{\dbint{1,n}}(fg))\\
&= D_{n+1} \Ll(\sum_{F \subset \dbint{1,n}} (D_F f)(D_{\dbint{1,n} \setminus F} g^F)\Rr)\\
&= \sum_{F \subset \dbint{1,n}} D_{n+1} \Ll((D_F f)(D_{\dbint{1,n} \setminus F} g^F)\Rr).
\end{align*}
We then use \eqref{eq:DiffProduct} to assert that
\begin{align*}
&D_{n+1} \Ll((D_F f)(D_{\dbint{1,n} \setminus F} g^F)\Rr)\\
& = (D_{F \cup \{n+1\}} f)(D_{\dbint{1,n} \setminus F} g^{F \cup \{n+1\}}) + (D_F f)(D_{\dbint{1,n+1} \setminus F} g^F) \\
& = (D_{F \cup \{n+1\}} f)(D_{\dbint{1,n+1} \setminus (F \cup \{n+1\})} g^{F \cup \{n+1\}}) + (D_F f)(D_{\dbint{1,n+1} \setminus F} g^F).
\end{align*}
Combining the two previous displays yields the claim.
\end{proof}

\subsection{Main strategy}\label{subsec:Strategy}
In this section, we present the structure of the proof of Theorem~\ref{t.smooth}. We will in fact mostly focus on the following finite-volume version of this statement. Recall the definitions of $\Delta^\rho$ and $\Delta^{\rho}_m$ in \eqref{eq:defPer} and \eqref{eq:defPerApp}, and we know from \eqref{p.C11.m} that ${\Delta^\rho = \lim_{m \to \infty} \Delta^\rho_m}$. In order to lighten the expressions appearing below, we use the notation in \eqref{eq:SymbolInte} to simplify the integration with respect to several particles.
\begin{proposition}[Smoothness in finite volume]
For every $\rho_0 > 0$ and $k,m \in \N_+$, we define
\label{prop:SmoothFinite}
\begin{equation}  
\label{e.def.ckm}
c_{k,m}(\rho_0) := \int_{(\Rd)^{\dbint{1,k}}}  \E\bigg[ \frac{1}{\rho_0 \vert \cu_m \vert} \int_{\cu_m} \nabla \psi_m
\cdot D_{\dbint{1,k}}\Ll((\a - \a^\# )\nabla \psi^\#_m\Rr) \, \d \mu \bigg] .
\end{equation}
There exists a positive constant $C_k(d,\Lambda) < \infty$ such that for every $m \in \N_+$ and $\rho_0 > 0$, 
\begin{equation}\label{eq:ckmBound}
\vert c_{k,m}(\rho_0)\vert \le C_k.
\end{equation}
Moreover, the quantity $\Delta^\rho_m(\rho_0)$ defined in \eqref{eq:defPerApp} satisfies that, for every $k \in \N_+$,
\begin{equation}\label{eq:ExpansionFinite} 
\Delta^\rho_m(\rho_0) =  \sum_{j = 1}^{k} c_{j,m}(\rho_0) \frac{\rho^j}{j!} + R_k(m, \rho_0, \rho),
\end{equation}
where $R_k(m, \rho_0, \rho)$ is such that, as $\rho > 0$ tends to zero and uniformly over $m$ and $\rho_0$,
\begin{align*}
R_k(m, \rho_0, \rho) = O(\rho^{k+1}).
\end{align*}
\end{proposition}
In Subsection~\ref{sub:proof.main}, we will obtain our main result Theorem~\ref{t.smooth} as a corollary of Proposition~\ref{prop:SmoothFinite}. For now, we present the structure of the proof of this proposition.

The first step of the proof of Proposition~\ref{prop:SmoothFinite} consists in identifying a convenient expansion for $\Delta_m^\rho$. 
As a starting point, one can check that if a function ${f : \mmd(\Rd) \to \R}$ is bounded and local, then we have
\begin{align}
\label{e.classical.formula}
\E[f(\mu+\mu_\rho)] = \sum_{k=0}^{\infty} \frac{\rho^k}{k!} \Ll(\int_{(\Rd)^{\dbint{1,k}}} \E[D_{\dbint{1,k}} f]  \Rr).
\end{align}
(See for instance \cite[Theorem 19.2]{bookPoisson}; a self-contained argument is given below.) Generalizing this observation, it is natural to expect that $\Delta^\rho_m$ can be rewritten from~\eqref{incr.rep.1} as
\begin{equation*}  
\sum_{k=1}^{\infty} \frac{\rho^k}{k!} \Ll(\int_{(\Rd)^{\dbint{1,k}}}  \E\Ll[ \frac{1}{\rho_0 \vert \cu_m \vert} \int_{\cu_m} \nabla \psi_m \cdot D_{\dbint{1,k}}\Ll((\a - \a^\# )\nabla \psi^\#_m\Rr) \, \d \mu \Rr]  \Rr),
\end{equation*}
where we dropped the summand indexed by $k = 0$, which vanishes. Notice that in the formula above, we could replace $\int_{(\Rd)^{\dbint{1,k}}}$ with $\int_{(\cu_{m+1})^{\dbint{1,k}}}$, because $\psi_m$ is $\mcl F_{Q_{3^m+1}}$-measurable, $\a$ is also local and the perturbation by adding particles outside $\cu_{m+1}$ will not contribute; this observation will be applied several times in the sequel. 
The following lemma states that the expansion formula is indeed valid for $\Delta^\rho_m$; its proof is provided in Subsection~\ref{subsec:Expansion}.

\begin{lemma}[Expansion of $\Delta^\rho_m$]
\label{lem:Expansion}
For each $m \in \N$, the quantity $\Delta^\rho_m$ is an analytic function of $\rho$ and satisfies
\begin{align}\label{eq:ExpansionModi}
\Delta^\rho_m = \sum_{k=1}^{\infty} \frac{\rho^k}{k!} \Ll(\int_{(\Rd)^{\dbint{1,k}}}  \E\Ll[ \frac{1}{\rho_0 \vert \cu_m \vert} \int_{\cu_m} \nabla \psi_m \cdot D_{\dbint{1,k}}\Ll((\a - \a^\# )\nabla \psi^\#_m\Rr) \, \d \mu \Rr]  \Rr).
\end{align}
\end{lemma}

The remainder of the proof of Proposition~\ref{prop:SmoothFinite} consists in the analysis of the summands in the expansion provided Lemma~\ref{lem:Expansion}. 
Applying the Leibniz formula \eqref{eq:ExpansionModi} to these summands:
\begin{equation}\label{eq:HigerOrderLeibniz}
D_{\dbint{1,k}}\Ll((\a - \a^\# )\nabla \psi^\#_m\Rr) 
= \sum_{\substack{E \cup F = \dbint{1,k}}} D_E (\a - \a^\# ) (D_{F} \nabla \psi_m),
\end{equation}
we are led to the expansion
\begin{align}\label{eq:Expansion}
\Delta^\rho_m  &= \sum_{k=1}^{\infty} \frac{\rho^k}{k!} \sum_{E \cup F = \dbint{1,k}} I(m,\rho_0,E,F),
\end{align}
where the quantity $I(m, \rho_0, E,F)$ is defined for  $E, F$ finite subsets of $\N_+$,
\begin{equation}
\label{eq:defI}
I(m,\rho_0,E,F) := \int_{(\Rd)^{\dbint{1,k}}} \E\Ll[ \frac{1}{\rho_0 \vert \cu_m \vert} \int_{\cu_m} \nabla \psi_m \cdot  D_E (\a - \a^\# ) D_{F} (\nabla \psi_m) \, \d \mu \Rr].
\end{equation}
It suffices to give a uniform estimate for the quantity $I(m,\rho_0,E,F)$ with respect to $m$ and $\rho_0$. Heuristically, the $k$ derivatives act either on the conductance or on the corrector, and they compensate with the integration $\int_{(\Rd)^{\dbint{1,k}}}$. With some more reduction, the estimation of these terms will be based on the following key result.

\begin{proposition}[Key estimate]
There exists a family of constants $\{C(i,j)\}_{i \geq j \geq 0}$ such that for every finite sets $G \subset  F \subset \N_+$,  $m \in \N_+$ and $\rho_0 > 0$, we have
\label{prop:KeyEstimate}
\begin{align}\label{eq:KeyEstimate}
\int_{(\Rd)^{F \setminus G }} \E\Ll[ \frac{1}{\rho_0 \vert \cu_m \vert} \int_{\cu_m} \Ll\vert \int_{(\Rd)^{G}} D_F \nabla \psi_m  \Rr \vert^2 \, \d \mu \Rr] \leq C(\vert F \vert, \vert G \vert) .
\end{align} 
\end{proposition}
The proof of this proposition is based on an induction argument. The base case, for $F = G = \emptyset$, is the standard Dirichlet energy estimate \eqref{energy.basic} for $\psi_m$.  Although this is not necessary, for greater clarity we first present the easier proof of the special case with $G = \emptyset$ and arbitrary $F$ in Subsection~\ref{subsec:KeyBasis}. We then give a proof for the general case in Subsection~\ref{subsec:KeyGeneral}. This requires a more careful use of Fubini's lemma and some inclusion-exclusion argument.
The proof of Proposition~\ref{prop:SmoothFinite} is then carried out in  Subsection~\ref{sub:proof.prop}, by combining the results above according to the outline just discussed.

\subsection{Expansion in finite volume}\label{subsec:Expansion}
We prove Lemma~\ref{lem:Expansion} in this part.
\begin{proof}[Proof of Lemma~\ref{lem:Expansion}]
We start by decomposing the expression for $\Delta^\rho_m$ with respect to $\mu_\rho \mres \cu_{m+1}$, as the particles outside $\cu_{m+1}$ will not contribute to the perturbation
\begin{align*}
\Delta^\rho_m & \stackrel{\eqref{incr.rep.1}}{=} \E\Ll[ \frac{1}{\rho_0\vert \cu_m \vert} \int_{\cu_m} \nabla  \psi_m \cdot(\a - \a^\rho) \nabla \psi^\rho_m \, \d \mu \Rr]\\
&= e^{- \rho \vert \cu_{m+1} \vert} \sum_{j=0}^{\infty}  \frac{(\rho \vert \cu_{m+1} \vert)^j}{j!}\Ll(\fint_{(\cu_{m+1})^{\dbint{1,j}}} \E\Ll[ \frac{1}{\rho_0\vert \cu_m \vert} \int_{\cu_m} \nabla  \psi_m \cdot(\a - \a^{\dbint{1,j}}) \nabla \psi^{\dbint{1,j}}_m \, \d \mu \Rr]\Rr).
\end{align*}
We establish first that the series in the above expression converges absolutely. Indeed, using the Cauchy--Schwarz inequality and applying the bound~\eqref{energy.basic} on the Dirichlet energy, we can write
\begin{equation}\label{eq:CS}
\begin{split}
&\Ll\vert \fint_{(\cu_{m+1})^{\dbint{1,j}}} \E\Ll[ \frac{1}{\rho_0\vert \cu_m \vert} \int_{\cu_m} \nabla  \psi_m \cdot(\a - \a^{\dbint{1,j}}) \nabla \psi^{\dbint{1,j}}_m \, \d \mu \Rr] \Rr\vert  \\
& \leq \Ll(\fint_{(\cu_{m+1})^{\dbint{1,j}}} \E\Ll[ \frac{1}{\rho_0\vert \cu_m \vert} \int_{\cu_m} \vert ( \a - \a^{\dbint{1,j}}) \nabla \psi^{\dbint{1,j}}_m \vert^2 \, \d \mu \Rr]\Rr)^{\frac{1}{2}} \\
& \qquad \times \Ll(\fint_{(\cu_{m+1})^{\dbint{1,j}}} \E\Ll[ \frac{1}{\rho_0\vert \cu_m \vert} \int_{\cu_m} \vert \nabla  \psi_m \vert^2 \, \d \mu \Rr]\Rr)^{\frac{1}{2}} \\
& \leq \Ll(\fint_{(\cu_{m+1})^{\dbint{1,j}}} \E\Ll[ \frac{1}{\rho_0\vert \cu_m \vert} \int_{\cu_m} \vert ( \a - \a^{\dbint{1,j}}) \nabla \psi^{\dbint{1,j}}_m \vert^2 \, \d \mu \Rr]\Rr)^{\frac{1}{2}}.
\end{split}
\end{equation}
We introduce the notation
\begin{align*}
A_j := \fint_{(\cu_{m+1})^{\dbint{1,j}}} \E\Ll[ \frac{1}{\rho_0\vert \cu_m \vert} \int_{\cu_m} \vert ( \a - \a^{\dbint{1,j}}) \nabla \psi^{\dbint{1,j}}_m \vert^2 \, \d \mu \Rr].
\end{align*}
We can further split the integrals contributing to $A_j$ according to the subset $E \subset \dbint{1,j}$ of particles outside of $\cu_m$, leading to 
\begin{equation}\label{eq:AB}
\begin{split}
    A_j & = \sum_{E \subset \dbint{1,j}}  \Ll(\frac{\vert \cu_{m+1} \setminus \cu_{m} \vert}{\vert \cu_{m+1} \vert}\Rr)^{\vert E \vert} \Ll(\frac{\vert \cu_m \vert}{\vert \cu_{m+1} \vert}\Rr)^{j - \vert E \vert} B_{j,E}, \\
    B_{j,E} &:=  \fint_{(\cu_{m+1} \setminus \cu_m)^E } \fint_{(\cu_m)^{\dbint{1,j} \setminus E}}  \E\Ll[ \frac{1}{\rho_0 |\cu_m|} \int_{\cu_m} |(\a - \a^{\dbint{1,j}}) \nabla \psi^{\dbint{1,j}}_m  |^2 \, d\mu \Rr].    
\end{split}
\end{equation}
 Now for (Lebesgue) almost every $(x_i)_{i \in E} \in (\cu_{m+1} \setminus \cu_m)^E$ fixed, they can be treated together with $\mu\mres (\cu_m)^c$ as the ``outer environment'', and we apply the improved energy inequality~\eqref{energy.basic.slice} for $(\mu(\cu_m) + j - \vert E\vert)$ particles to obtain that
\begin{align*}
\fint_{(\cu_m)^{\dbint{1,j} \setminus E}} \E\Ll[ \frac{1}{\rho_0 |\cu_m|} \int_{\cu_m} |\nabla \psi^{\dbint{1,j}}_m  |^2  \, d\mu  \, \biggl| \, \mu(\cu_m) , \mu\mres (\cu_m)^c  \Rr] \leq \frac{ (\mu(\cu_m)  + j - \vert E \vert)}{\rho_0|\cu_m|}.
\end{align*}
From this expression and the uniform ellipticity assumption \eqref{a.elliptic}, one obtains that 
\begin{align*}
B_{j,E} \leq C \, \frac{ \rho_0 |\cu_m| + j - \vert E \vert}{\rho_0|\cu_m|},
\end{align*}
and thus
\begin{align*}
A_j  \leq C \sum_{ \ell = 0}^j \binom{j}{\ell} 3^{-d (j-\ell)} (1-3^{-d})^{\ell} \Ll(1+\frac{j-\ell}{ \rho_0 \vert \cu_m \vert} \Rr)  \leq C \Ll(1 + \frac{j}{\rho_0 \vert \cu_m \vert} \Rr).
\end{align*}
We use this estimate with  \eqref{eq:CS} to get that 
\begin{multline*}
     \sum_{j=0}^{\infty}  \frac{(\rho \vert \cu_{m+1} \vert)^j}{j!}\Ll\vert\fint_{(\cu_{m+1})^{\dbint{1,j}}} \E\Ll[ \frac{1}{\rho_0\vert \cu_m \vert} \int_{\cu_m} \nabla  \psi_m \cdot(\a - \a^{\dbint{1,j}}) \nabla \psi^{\dbint{1,j}}_m \, \d \mu \Rr]\Rr \vert \\
     \leq C\sum_{j = 0}^\infty \frac{(\rho |\cu_{m+1}|)^j}{j!} \Ll(1 + \frac{j}{\rho_0 \vert \cu_m \vert} \Rr)^{\frac{1}{2}} < \infty,
\end{multline*}
which implies that the series is absolutely convergent. Since $e^{- \rho \vert \cu_{m+1} \vert}$ is analytic with respect to $\rho$, their product $\Delta^\rho_m$ is also an analytic function of $\rho$. Then we expand $e^{- \rho \vert \cu_{m+1} \vert}$ into its Talyor series
\begin{align*}
\Delta^\rho_m = \sum_{l=0}^\infty \frac{(-\rho \vert \cu_{m+1} \vert)^l}{l!}&\sum_{j=0}^{\infty}  \frac{(\rho \vert \cu_{m+1} \vert)^j}{j!}\\
&\times \Ll(\fint_{(\cu_{m+1})^{\dbint{1,j}}} \E\Ll[ \frac{1}{\rho_0\vert \cu_m \vert} \int_{\cu_m} \nabla  \psi_m \cdot(\a - \a^{\dbint{1,j}}) \nabla \psi^{\dbint{1,j}}_m \, \d \mu \Rr]\Rr),
\end{align*}
and the absolute convergence allows us to reorganize the summations according to
\begin{align*}
\Delta^\rho_m = \sum_{k=0}^{\infty}\sum_{\substack{l,j \in \N, \\ l+j = k}}^\infty &\frac{(-1)^l(\rho \vert \cu_{m+1} \vert)^k}{l!j!}\\
&\times \Ll(\fint_{(\cu_{m+1})^{\dbint{1,j}}} \E\Ll[ \frac{1}{\rho_0\vert \cu_m \vert} \int_{\cu_m} \nabla  \psi_m \cdot(\a - \a^{\dbint{1,j}}) \nabla \psi^{\dbint{1,j}}_m \, \d \mu \Rr]\Rr).
\end{align*}
We also observe that the part $\fint_{(\cu_{m+1})^{\dbint{1,j}}} (\cdots)$ means the adding of $j$ particles in $\cu_{m+1}$, but the indices do not play a specific role. Thus we have 
\begin{multline*}
\fint_{(\cu_{m+1})^{\dbint{1,j}}} \E\Ll[ \frac{1}{\rho_0\vert \cu_m \vert} \int_{\cu_m} \nabla  \psi_m \cdot(\a - \a^{\dbint{1,j}}) \nabla \psi^{\dbint{1,j}}_m \, \d \mu \Rr] \\ = {\binom{k}{j}}^{-1} \sum_{E \subset \dbint{1,k}, \vert E \vert = j} \fint_{(\cu_{m+1})^{\dbint{1,k}}} \E\Ll[ \frac{1}{\rho_0\vert \cu_m \vert} \int_{\cu_m} \nabla  \psi_m \cdot(\a - \a^{E}) \nabla \psi^{E}_m \, \d \mu \Rr].
\end{multline*}
This leads to
\begin{align*}
\Delta^\rho_m &= \sum_{k=0}^{\infty}\sum_{\substack{l,j \in \N, \\ l+j = k}}^\infty \frac{(-1)^l \rho ^k}{k!}   \sum_{E \subset \dbint{1,k}, \vert E \vert = j} \Ll(\int_{(\cu_{m+1})^{\dbint{1,k}}} \E\Ll[ \frac{1}{\rho_0\vert \cu_m \vert} \int_{\cu_m} \nabla  \psi_m \cdot(\a - \a^{E}) \nabla \psi^{E}_m \, \d \mu \Rr]\Rr)  \\
&= \sum_{k=0}^{\infty} \sum_{E \subset \dbint{1,k}} \frac{(-1)^{k-\vert E \vert} \rho ^k}{k!}   \Ll(\int_{(\cu_{m+1})^{\dbint{1,k}}} \E\Ll[ \frac{1}{\rho_0\vert \cu_m \vert} \int_{\cu_m} \nabla  \psi_m \cdot(\a - \a^{E}) \nabla \psi^{E}_m \, \d \mu \Rr]\Rr) \\
&= \sum_{k=1}^{\infty} \frac{\rho ^k}{k!}   \Ll(\int_{(\cu_{m+1})^{\dbint{1,k}}} \E\Ll[ \frac{1}{\rho_0\vert \cu_m \vert} \int_{\cu_m} \nabla  \psi_m \cdot D_{\dbint{1,k}}((\a - \a^{\#}) \nabla \psi^{\#}_m) \, \d \mu \Rr]\Rr).
\end{align*}
From the second to the third line, we use the inclusion-exclusion formula \eqref{eq:Difference}. The term $k=0$ can be dropped since it vanishes. Finally, we can extend $\int_{(\cu_{m+1})^{\dbint{1,k}}}$ to $\int_{(\Rd)^{\dbint{1,k}}}$ and this is the desired result \eqref{eq:ExpansionModi}.
\end{proof}

\subsection{Key estimate for base case}\label{subsec:KeyBasis}
In this part, for clarity of exposition, we prove \eqref{eq:KeyEstimate} in the simpler case $G = \emptyset$. That is, we show that for every finite $F \subset \N_+$,
\begin{align}\label{eq:KeyEstimateBasis}
\int_{(\Rd)^{F}} \E\Ll[ \frac{1}{\rho_0 \vert \cu_m \vert} \int_{\cu_m} \Ll\vert  D_F\nabla \psi_m\Rr \vert^2 \, \d \mu \Rr] \leq C(\vert F \vert, 0) .
\end{align}

We start by introducing some notation (that will mostly be useful in the more general case treated in the next subsection). 
For $x,z \in \R^d$, we write $\Upsilon(E,z)$ to denote the indicator function
\begin{align}\label{eq:SymbolRestriction}
\Upsilon(E,z)(x) := \prod_{i \in E}\Ind{x_i \in z + \cu}.
\end{align}
We record a handful of elementary observations concerning $\Upsilon$: for every finite sets $E,F \subset \N_+$ and $z \in \Rd$, we have
\begin{align}\label{eq:UpsilonProduct}
\Upsilon(E, z)\Upsilon(F,z) = \Upsilon(E \cup F, z),
\end{align}
\begin{align}\label{eq:UpsilonInt}
\int_{(\Rd)^F} \Upsilon(E,z) \leq \int_{(\Rd)^{F \setminus E}} \Upsilon(E \setminus F, z),
\end{align}
and
\begin{align}\label{eq:DAEstimate}
\vert D_E \a(\mu, z)\vert \leq 2^{\vert E \vert} \Lambda \Upsilon(E,z).
\end{align}

\begin{proof}[Proof of \eqref{eq:KeyEstimateBasis}]  
The case $F = \emptyset$ is the basic energy estimate in \eqref{energy.basic}, so we now assume that $F \neq \emptyset$. By Proposition~\ref{cor:var.formulation} for $\rho = 0$, we have for any finite $E_1, E_2 \subset \N_+$ that
\begin{multline}
\label{eq.KeyBasis1}
\fint_{(\cu_{m+1})^{E_1 \cup E_2}}\E\Ll[ \frac{1}{\rho_0 \vert \cu_m \vert} \int_{\cu_m} \nabla \psi^{E_2}_m  \cdot \a^{E_1} \nabla \psi^{E_1}_m \, \d \mu \Rr] \\
= \fint_{(\cu_{m+1})^{E_1 \cup E_2}} \E\Ll[ \frac{1}{\rho_0 \vert \cu_m \vert} \int_{\cu_m} \nabla \psi^{E_2}_m \cdot q \, \d \mu\Rr].
\end{multline}
We apply this with $E_1, E_2 \subset F$, thus we can extend as an average over particles in $F$ that
\begin{multline}\label{eq:KeyIdentity0}
\fint_{(\cu_{m+1})^{F}}\E\Ll[ \frac{1}{\rho_0 \vert \cu_m \vert} \int_{\cu_m} \nabla \psi^{E_2}_m  \cdot \a^{E_1} \nabla \psi^{E_1}_m \, \d \mu \Rr] \\
= \fint_{(\cu_{m+1})^{F}}\E\Ll[ \frac{1}{\rho_0 \vert \cu_m \vert} \int_{\cu_m} \nabla \psi^{E_2}_m \cdot q \, \d \mu \Rr].
\end{multline}
We do the linear combination of \eqref{eq:KeyIdentity0} over all the $E_1 \subset F$ and apply the inclusion-exclusion formula \eqref{eq:Difference} to obtain 
\begin{align}\label{eq:KeyIdentity}
\int_{(\cu_{m+1})^{F}}\E\Ll[ \frac{1}{\rho_0 \vert \cu_m \vert} \int_{\cu_m} \nabla \psi^{E_2}_m \cdot D_{F}\Ll(\a^\# \nabla \psi^\#_m\Rr)  \, \d \mu \Rr] = 0.
\end{align}
Here the sum on the {\rhs} is zero thanks to the inclusion-exclusion formula and $F \neq \emptyset$. We extend the integration $\int_{(\cu_{m+1})^{F}}$ to $\int_{(\Rd)^{F}}$ and then apply the linear combination of \eqref{eq:KeyIdentity} over all the $E_2 \subset F$ to obtain 
\begin{align}\label{eq:KeyIdentity2}
\int_{(\Rd)^{F}}\E\Ll[ \frac{1}{\rho_0 \vert \cu_m \vert} \int_{\cu_m} D_{F}(\nabla \psi_m) \cdot D_{F}\Ll(\a^\# \nabla \psi^\#_m\Rr)  \, \d \mu \Rr] = 0.
\end{align}
Now we use the Leibniz's formula in \eqref{eq:Leibniz1} and obtain that 
\begin{align*}
D_{F}\Ll(\a^\# \nabla \psi^\#_m\Rr) = \sum_{G \subset F} D_{F \setminus G} (\a^G) (D_G \nabla \psi_m).
\end{align*}
We put this back to \eqref{eq:KeyIdentity2}, and keep the term ${(D_F \nabla \psi_m) \cdot \a^F  (D_F \nabla \psi_m)}$ on the \lhs, while moving the other terms to the \rhs
\begin{multline*}
\int_{(\Rd)^{F}}\E\Ll[ \frac{1}{\rho_0 \vert \cu_m \vert} \int_{\cu_m} (D_F \nabla \psi_m) \cdot \a^F (D_F \nabla \psi_m)  \, \d \mu \Rr] \\
=  -\sum_{G \subsetneq F} \Ll(\int_{(\Rd)^{F}}\E\Ll[ \frac{1}{\rho_0 \vert \cu_m \vert} \int_{\cu_m} (D_F \nabla \psi_m) \cdot D_{F \setminus G} (\a^G) (D_G \nabla \psi_m)  \, \d \mu  \Rr]\Rr).
\end{multline*}
Using the Cauchy--Schwarz inequality and the triangle inequality, we obtain that 
\begin{multline*}
\Ll(\int_{(\Rd)^{F}}\E\Ll[ \frac{1}{\rho_0 \vert \cu_m \vert} \int_{\cu_m} \vert D_F \nabla \psi_m \vert^2  \, \d \mu  \Rr]\Rr)^{\frac{1}{2}} \\
\leq  \sum_{G \subsetneq F} \Ll(\int_{(\Rd)^{F}}\E\Ll[ \frac{1}{\rho_0 \vert \cu_m \vert} \int_{\cu_m} \vert D_{F \setminus G} (\a^G)\vert^2 \Ll\vert D_G \nabla \psi_m\Rr\vert^2  \, \d \mu \Rr]\Rr)^{\frac{1}{2}}.
\end{multline*}
Then we use Fubini's lemma to pass $\int_{(\Rd)^{F \setminus G}}$
\begin{multline*}
\int_{(\Rd)^{F}}\E\Ll[ \frac{1}{\rho_0 \vert \cu_m \vert} \int_{\cu_m} \vert D_{F \setminus G} (\a^G) \vert^2 \Ll\vert D_G \nabla \psi_m\Rr\vert^2  \, \d \mu \Rr]\\
 = \int_{(\Rd)^{G}}\E\Ll[ \frac{1}{\rho_0 \vert \cu_m \vert} \int_{\cu_m} \Ll( \int_{(\Rd)^{F \setminus G}}\vert D_{F \setminus G} (\a^G)\vert^2 \Rr) \Ll\vert D_G \nabla \psi_m\Rr\vert^2  \, \d \mu \Rr].
\end{multline*}
The last line uses the fact that $D_G \nabla \psi_m$ does not involve the particle in $F \setminus G$. A varied version of \eqref{eq:DAEstimate} and \eqref{eq:UpsilonInt} gives us that 
\begin{align*}
\int_{(\Rd)^{F \setminus G}}\vert D_{F \setminus G} (\a^G)\vert^2 \leq 4^{\vert F \setminus G\vert} \Lambda^2 \int_{(\Rd)^{F \setminus G}} \Upsilon(F \setminus G, \cdot) \leq 4^{\vert F \setminus G\vert} \Lambda^2. 
\end{align*}
Therefore, we obtain an estimate that 
\begin{multline}\label{eq:KeyBasisRecurrence}
\Ll(\int_{(\Rd)^{F}}\E\Ll[ \frac{1}{\rho_0 \vert \cu_m \vert} \int_{\cu_m} \vert D_F \nabla \psi_m \vert^2  \, \d \mu \Rr]\Rr)^{\frac{1}{2}} \\
\leq  \sum_{G \subsetneq F} \Ll(4^{\vert F \setminus G\vert}\Lambda^2 \int_{(\Rd)^{G}}\E\Ll[ \frac{1}{\rho_0 \vert \cu_m \vert} \int_{\cu_m} \Ll\vert D_G \nabla \psi_m\Rr\vert^2  \, \d \mu \Rr]\Rr)^{\frac{1}{2}}.
\end{multline}
This estimate allows us to justify the induction argument. Indeed, the case $\vert F \vert = 0$ is the Dirichlet energy estimate. Suppose \eqref{eq:KeyEstimateBasis} is valid for $\vert F \vert = n$, then for $\vert F \vert = n+1$, we apply \eqref{eq:KeyBasisRecurrence}. As the quantity on the {\rhs} only relies on $G \subsetneq F$, which implies $\vert G \vert \leq n$, we can invoke \eqref{eq:KeyEstimateBasis} for lower order. This completes the proof of~\eqref{eq:KeyEstimateBasis}.
\end{proof} 

\subsection{Key estimate for the general case}\label{subsec:KeyGeneral}
In this part, we now treat the general case of \eqref{eq:KeyEstimate}. 
\begin{proof}[Proof of Proposition~\ref{prop:KeyEstimate}]
We decompose the proof into three steps and we suppose $F \neq \emptyset$.

\textit{Step 1: Expansion.} We start once again from \eqref{eq:KeyIdentity0}, and apply a ``doubling variables trick''. For $G \subset F \subset \N_+$, we add another set $G' \subset \N_+ \setminus F$ such that $\vert G' \vert = \vert G \vert$, and consider \eqref{eq:KeyIdentity0} for some $E_1 \subset F$ and $E'_2 \subset (F \setminus G) \sqcup G'$. Then $(E_1 \cup E'_2) \subset (F \sqcup G')$ and \eqref{eq:KeyIdentity0} becomes 
\begin{multline*}
\fint_{(\cu_{m+1})^{F \sqcup G'}}\E\Ll[ \frac{1}{\rho_0 \vert \cu_m \vert} \int_{\cu_m} \nabla \psi^{E'_2}_m  \cdot \a^{E_1} \nabla \psi^{E_1}_m \, \d \mu \Rr] \\
= \fint_{(\cu_{m+1})^{F \sqcup G'}}\E\Ll[ \frac{1}{\rho_0 \vert \cu_m \vert} \int_{\cu_m} \nabla \psi^{E'_2}_m \cdot q \, \d \mu \Rr].
\end{multline*}
 Then we apply the inclusion-exclusion formula \eqref{eq:Difference} over all $E_1 \subset F$ and obtain that 
\begin{align*}
\int_{(\cu_{m+1})^{F \sqcup G'}}\E\Ll[ \frac{1}{\rho_0 \vert \cu_m \vert} \int_{\cu_m} \nabla \psi^{E'_2}_m \cdot D_{F}\Ll(\a^\# \nabla \psi^\#_m\Rr)  \, \d \mu \Rr] = 0.
\end{align*}
From this line, we can extend $\int_{(\cu_{m+1})^{F \sqcup G'}}$ to $\int_{(\Rd)^{F \sqcup G'}}$. We then apply the inclusion-exclusion formula \eqref{eq:Difference} over all $E'_2 \subset (F \setminus G) \sqcup G'$ and obtain 
\begin{align*}
\int_{(\Rd)^{F \sqcup G'}}\E\Ll[ \frac{1}{\rho_0 \vert \cu_m \vert} \int_{\cu_m} (D_{(F \setminus G) \sqcup G'} \nabla \psi_m) \cdot D_{F}\Ll(\a^\# \nabla \psi^\#_m\Rr)  \, \d \mu \Rr] = 0.
\end{align*}
Notice that the particles in $G'$ only act on the term $(D_{(F \setminus G) \sqcup G'} \nabla \psi_m)$, we can pass $\int_{(\Rd)^{G'}}$ to the interior and this equation becomes 
\begin{align*}
\int_{(\Rd)^{F}}\E\Ll[ \frac{1}{\rho_0 \vert \cu_m \vert} \int_{\cu_m} \Ll( \int_{(\Rd)^{G'}} D_{(F \setminus G) \sqcup G'} \nabla \psi_m\Rr) \cdot D_{F}\Ll(\a^\# \nabla \psi^\#_m\Rr)  \, \d \mu\Rr] = 0.
\end{align*}
Up to a relabelling of the particles, we can write
\begin{align*}
 \int_{(\Rd)^{G'}} D_{(F \setminus G) \sqcup G'} \nabla \psi_m =  \int_{(\Rd)^{G}} D_{F} \nabla \psi_m,
\end{align*}
and obtain a counter-part of \eqref{eq:KeyIdentity2} that 
\begin{align*}
\int_{(\Rd)^{F}}\E\Ll[ \frac{1}{\rho_0 \vert \cu_m \vert} \int_{\cu_m} \Ll( \int_{(\Rd)^{G}} D_{F} \nabla \psi_m \Rr) \cdot D_{F}\Ll(\a^\# \nabla \psi^\#_m\Rr)  \, \d \mu\Rr] = 0.
\end{align*}

Like \eqref{eq:KeyIdentity2}, we do an expansion for this identity for the term $D_{F}\Ll(\a^\# \nabla \psi^\#_m\Rr)$, but we need to treat it more carefully. We apply \eqref{eq:Leibniz2} on ${D_{F}\Ll(\a^\# \nabla \psi^\#_m\Rr)}$ and obtain that 
\begin{align*}
D_{F}\Ll(\a^\# \nabla \psi^\#_m\Rr) = \sum_{F_1 \cup F_2 = F}D_{F_2}(\a) (D_{F_1}\nabla \psi_m).
\end{align*}
We keep the term
\begin{align*}
 \{F_1 \cup F_2 = F\} \cap \{ F_2 \subset (F \setminus G)\} \cap \{ F_1 = F \},
\end{align*}
on the \lhs, while putting the other terms
\begin{align*}
\{ F_1 \cup F_2 = F\} \cap \Ll(\{F_2 \cap G \neq \emptyset\} \cup \{ F_1 \subsetneq F \}\Rr),
\end{align*}
on the \rhs . We also notice \eqref{eq:Telescope} that 
\begin{align*}
\sum_{ F_2 \subset (F \setminus G)} D_{F_2}(\a) = \a^{F \setminus G},
\end{align*}
so we obtain that 
\begin{multline}\label{eq:KeyIdentityGeneral}
\int_{(\Rd)^{F}}\E\Ll[ \frac{1}{\rho_0 \vert \cu_m \vert} \int_{\cu_m} \Ll( \int_{(\Rd)^{G}} D_{F} \nabla \psi_m \Rr) \cdot \a^{F \setminus G} (D_F \nabla \psi_m)  \, \d \mu \Rr] \\
=  \sum_{\substack{F_1 \cup F_2 = F\\
F_2 \cap G \neq \emptyset, \text{ or }  F_1 \subsetneq F}} -\Ll(\int_{(\Rd)^{F}}\E\Ll[ \frac{1}{\rho_0 \vert \cu_m \vert} \int_{\cu_m} \Ll( \int_{(\Rd)^{G}} D_{F} \nabla \psi_m \Rr) \cdot D_{F_2} (\a) (D_{F_1} \nabla \psi_m)  \, \d \mu \Rr]\Rr).
\end{multline}
Because $\a^{F \setminus G}, \int_{(\Rd)^{G}} D_F \psi_m$ and $ \d \mu $ do not depend on the particles indexed by $G$, we can apply Fubini's lemma to pass $\int_{(\Rd)^{G}}$ to interior, thus the \lhs\ of \eqref{eq:KeyIdentityGeneral}
becomes 
\begin{align*}
&\int_{(\Rd)^{F}}\E\Ll[ \frac{1}{\rho_0 \vert \cu_m \vert} \int_{\cu_m} \Ll(\int_{(\Rd)^{G}} D_F\nabla \psi_m \Rr) \cdot \a^{F \setminus G} (D_F \nabla \psi_m)  \, \d \mu \Rr] \\
&= \int_{(\Rd)^{F \setminus G}}\E\Ll[ \frac{1}{\rho_0 \vert \cu_m \vert} \int_{\cu_m} \Ll(\int_{(\Rd)^{G}} D_F\nabla \psi_m \Rr) \cdot \a^{F \setminus G} \Ll(\int_{(\Rd)^{G}} D_F\nabla \psi_m \Rr)  \, \d \mu \Rr]\\
&\geq \int_{(\Rd)^{F \setminus G}}\E\Ll[ \frac{1}{\rho_0 \vert \cu_m \vert} \int_{\cu_m} \Ll\vert \int_{(\Rd)^{G}} D_F\nabla \psi_m \Rr \vert^2  \, \d \mu \Rr].
\end{align*}
For the \rhs , we argue similarly by the Cauchy--Schwarz inequality and the triangle inequality to obtain that 
\begin{multline}\label{eq:KeyGeneralInter}
\Ll(\int_{(\Rd)^{F \setminus G}}\E\Ll[ \frac{1}{\rho_0 \vert \cu_m \vert} \int_{\cu_m} \Ll\vert \int_{(\Rd)^{G}} D_F\nabla \psi_m \Rr \vert^2  \, \d \mu \Rr]\Rr)^{\frac{1}{2}} \\
\leq  \sum_{\substack{F_1 \cup F_2 = F\\
F_2 \cap G \neq \emptyset, \text{ or }  F_1 \subsetneq F}} \Ll(\int_{(\Rd)^{F \setminus G}}\E\Ll[ \frac{1}{\rho_0 \vert \cu_m \vert} \int_{\cu_m} \Ll\vert \int_{(\Rd)^{G}}D_{F_2} (\a) (D_{F_1} \nabla \psi_m) \Rr \vert^2  \, \d \mu\Rr]\Rr)^{\frac{1}{2}}.
\end{multline}

\textit{Step 2: Simplification and recurrence inequality.} The final goal is to get a recurrence like \eqref{eq:KeyBasisRecurrence}, but \eqref{eq:KeyGeneralInter} still needs some further simplification. We focus on the term 
\begin{align*}
\int_{(\Rd)^{G}}D_{F_2} (\a) (D_{F_1} \nabla \psi_m).
\end{align*}
Since $F_1 \cup F_2 = F$, we use the disjoint decomposition that 
\begin{align*}
F = (F_2 \setminus F_1) \sqcup (F_1 \setminus F_2) \sqcup (F_2 \cap F_1),
\end{align*}
which also induces the decomposition of $G$
\begin{align*}
G = ((G \cap F_2) \setminus F_1) \sqcup ((G \cap F_1) \setminus F_2) \sqcup (G \cap F_2 \cap F_1).
\end{align*}
Thus, we can decompose 
\begin{align*}
\int_{(\Rd)^{G}} = \int_{(\Rd)^{(G \cap F_2) \setminus F_1}} \int_{(\Rd)^{(G \cap F_1) \setminus F_2}} \int_{(\Rd)^{G \cap F_2 \cap F_1}},
\end{align*}
and pass them respectively to the proper term
\begin{multline*}
\int_{(\Rd)^{G}}D_{F_2} (\a) (D_{F_1} \nabla \psi_m)  \\
=\int_{(\Rd)^{G \cap F_2 \cap F_1}}\Ll( \Ll(\int_{(\Rd)^{(G \cap F_2) \setminus F_1}} D_{F_2} (\a)\Rr) \Ll( \int_{(\Rd)^{(G \cap F_1) \setminus F_2}} D_{F_1} \nabla \psi_m\Rr)\Rr).
\end{multline*}
Let $z \in \supp(\mu)$ be the particle at which the gradient is computed, then we use the notation \eqref{eq:SymbolRestriction} and the estimate \eqref{eq:DAEstimate} to give its bound
\begin{align*}
&\Ll\vert \int_{(\Rd)^{G}}D_{F_2} (\a) D_{F_1} \nabla \psi_m  \Rr\vert^2(z) \\
& \leq \Ll\vert \int_{(\Rd)^{G \cap F_2 \cap F_1}}\Ll( \Ll(\int_{(\Rd)^{(G \cap F_2) \setminus F_1}} \vert D_{F_2} (\a) \vert \Rr) \Ll\vert \int_{(\Rd)^{(G \cap F_1) \setminus F_2}} D_{F_1} \nabla \psi_m\Rr\vert\Rr)\Rr\vert^2(z) \\
& \leq 4^{\vert F_2\vert}\Lambda^2 \Ll\vert \int_{(\Rd)^{G \cap F_2 \cap F_1}}\Ll( \Ll(\int_{(\Rd)^{(G \cap F_2) \setminus F_1}} \Upsilon(F_2, z)\Rr) \Ll\vert \int_{(\Rd)^{(G \cap F_1) \setminus F_2}} D_{F_1} \nabla \psi_m\Rr\vert\Rr)\Rr\vert^2(z).
\end{align*}
Next, we use the property that $\Upsilon(F_2, z)$ requires all the particles in $F_2$ to live in $z + \cu$ 
\begin{align*}
&\Ll\vert \int_{(\Rd)^{G \cap F_2 \cap F_1}}\Ll( \Ll(\int_{(\Rd)^{(G \cap F_2) \setminus F_1}} \Upsilon(F_2, z)\Rr) \Ll\vert \int_{(\Rd)^{(G \cap F_1) \setminus F_2}} D_{F_1} \nabla \psi_m\Rr\vert\Rr)\Rr\vert^2(z)\\
&=  \Ll\vert \int_{(z + \cu)^{G \cap F_2 \cap F_1}}\Ll( \Ll(\int_{(\Rd)^{(G \cap F_2) \setminus F_1}} \Upsilon(F_2, \cdot)\Rr) \Ll\vert \int_{(\Rd)^{(G \cap F_1) \setminus F_2}} D_{F_1} \nabla \psi_m\Rr\vert\Rr)\Rr\vert^2(z) \\
&\leq \int_{(z + \cu)^{G \cap F_2 \cap F_1}}\Ll( \Ll(\int_{(\Rd)^{(G \cap F_2) \setminus F_1}} \Upsilon(F_2, \cdot)\Rr)^2 \Ll\vert \int_{(\Rd)^{(G \cap F_1) \setminus F_2}} D_{F_1} \nabla \psi_m\Rr\vert^2\Rr)(z)\\
&= \int_{(\Rd)^{G \cap F_2 \cap F_1}}\Ll( \Ll(\int_{(\Rd)^{(G \cap F_2) \setminus F_1}} \Upsilon(F_2, \cdot)\Rr)^2 \Ll\vert \int_{(\Rd)^{(G \cap F_1) \setminus F_2}} D_{F_1} \nabla \psi_m\Rr\vert^2\Rr)(z).
\end{align*}
From the second line to the third line, we use the Cauchy--Schwarz inequality, and from the third line to the fourth line, we reapply the property of $\Upsilon(F_2, z)$. So in this step we gain a small factor for Cauchy--Schwarz inequality. Now, we use the property \eqref{eq:UpsilonInt}
\begin{align*}
\int_{(\Rd)^{(G \cap F_2) \setminus F_1}} \Upsilon(F_2, z) \leq  \Upsilon(F_2 \setminus ((G \cap F_2) \setminus F_1), z). 
\end{align*}
We put all these estimates back to the {\rhs} of \eqref{eq:KeyGeneralInter} to obtain that  
\begin{equation}\label{eq:KeyGeneralInter2}
\begin{split}
&\int_{(\Rd)^{F \setminus G}}\E\Ll[ \frac{1}{\rho_0 \vert \cu_m \vert} \int_{\cu_m} \Ll\vert \int_{(\Rd)^{G}}D_{F_2} (\a) (D_{F_1} \nabla \psi_m) \Rr \vert^2  \, \d \mu\Rr] \\
&\leq 4^{\vert F_2\vert} \Lambda^2 
\int_{(\Rd)^{F \setminus G}}\E\Ll[ \frac{1}{\rho_0 \vert \cu_m \vert} \int_{\cu_m} \Rr. \\
& \qquad \Ll. \int_{(\Rd)^{G \cap F_2 \cap F_1}}\Ll(\Upsilon(F_2 \setminus ((G \cap F_2) \setminus F_1), \cdot)
\Ll\vert \int_{(\Rd)^{(G \cap F_1) \setminus F_2}} D_{F_1} \nabla \psi_m\Rr\vert^2\Rr)   \, \d \mu\Rr].
\end{split}
\end{equation}
In this integral, we can continue some simplification with $\Upsilon(F_2 \setminus ((G \cap F_2) \setminus F_1), \cdot)$. We have the following disjoint union  
\begin{equation*}
F \setminus G = ((F_2 \cap F_1) \setminus G) \sqcup ((F_2 \setminus F_1) \setminus G) \sqcup ((F_1 \setminus F_2) \setminus G),
\end{equation*}
which implies that 
\begin{align}\label{eq:KeyGeneralDecom}
\int_{(\Rd)^{F \setminus G}} = \int_{(\Rd)^{(F_2 \cap F_1) \setminus G}} \int_{(\Rd)^{(F_2 \setminus F_1) \setminus G}} \int_{(\Rd)^{(F_1 \setminus F_2) \setminus G}}.
\end{align}
Because $\int_{(\Rd)^{(G \cap F_1) \setminus F_2}} D_{F_1} \nabla \psi_m$ only involves a subset of particles in $F_1$, which is disjoint from $(F_2 \setminus F_1) \setminus G$, we use Fubini's lemma to pass the integration of $\int_{(\Rd)^{(F_2 \setminus F_1) \setminus G}}$ to the inside
\begin{align*}
&\int_{(\Rd)^{(F_2 \setminus F_1) \setminus G}}\E\Ll[ \frac{1}{\rho_0 \vert \cu_m \vert} \int_{\cu_m} \int_{(\Rd)^{G \cap F_2 \cap F_1}}\Rr. \\
& \qquad \qquad \Ll. \Ll(\Upsilon(F_2 \setminus ((G \cap F_2) \setminus F_1), \cdot)
\Ll\vert \int_{(\Rd)^{(G \cap F_1) \setminus F_2}} D_{F_1} \nabla \psi_m\Rr\vert^2\Rr)   \, \d \mu\Rr] \\
& = \E\Ll[ \frac{1}{\rho_0 \vert \cu_m \vert} \int_{\cu_m}  \int_{(\Rd)^{G \cap F_2 \cap F_1}} 
\Rr.\\
& \qquad  \Ll. \Ll(\Ll(\int_{(\Rd)^{(F_2 \setminus F_1) \setminus G}} \Upsilon(F_2 \setminus ((G \cap F_2) \setminus F_1), \cdot)\Rr) \Ll\vert \int_{(\Rd)^{(G \cap F_1) \setminus F_2}} D_{F_1} \nabla \psi_m\Rr\vert^2\Rr)    \, \d \mu\Rr]\\
&= \E\Ll[ \frac{1}{\rho_0 \vert \cu_m \vert} \int_{\cu_m}  \int_{(\Rd)^{G \cap F_2 \cap F_1}} \Ll( \Upsilon(F_2 \cap F_1, \cdot)
\Ll\vert \int_{(\Rd)^{(G \cap F_1) \setminus F_2}} D_{F_1} \nabla \psi_m\Rr\vert^2\Rr)   \, \d \mu\Rr] \\ 
&= \int_{(\Rd)^{G \cap F_2 \cap F_1}} \E\Ll[ \frac{1}{\rho_0 \vert \cu_m \vert} \int_{\cu_m}   \Ll( \Upsilon(F_2 \cap F_1, \cdot)
\Ll\vert \int_{(\Rd)^{(G \cap F_1) \setminus F_2}} D_{F_1} \nabla \psi_m\Rr\vert^2\Rr)   \, \d \mu\Rr]. 
\end{align*}
From the second line to the third line, we used \eqref{eq:UpsilonInt} and the decomposition
\begin{equation*}
F_2 \setminus ((G \cap F_2) \setminus F_1) = (F_2 \cap F_1) \sqcup ((F_2 \setminus F_1) \setminus G).
\end{equation*}
See the Venn diagram in Figure~\ref{fig:Venn} to help check this equation. From the third line to the fourth line, we put the integral $\int_{(\Rd)^{G \cap F_2 \cap F_1}}$ outside the expectation using Fubini's lemma. We combine this integral together with the rest of integrals in \eqref{eq:KeyGeneralDecom} and we observe that 
\begin{align}\label{eq:KeyGeneralContract}
\int_{(\Rd)^{(F_2 \cap F_1) \setminus G}}\int_{(\Rd)^{(F_1 \setminus F_2) \setminus G}}\int_{(\Rd)^{G \cap F_2 \cap F_1}} = \int_{(\Rd)^{F_1 \setminus ((G \cap F_1) \setminus F_2)}}.
\end{align} 
because of the identity (see Figure~\ref{fig:Venn} to help check this equation)
\begin{align*}
((F_2 \cap F_1) \setminus G) \sqcup ((F_1 \setminus F_2) \setminus G) \sqcup (G \cap F_2 \cap F_1) = F_1 \setminus ((G \cap F_1) \setminus F_2).
\end{align*}

\begin{figure}[h!]
\centering
\includegraphics[scale=0.6]{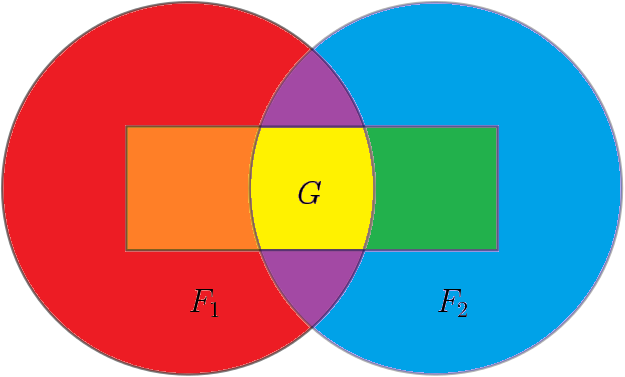}
\caption{A Venn diagram for illustration. The disk on the left represents $F_1$ and the disk on the right represents $F_2$. The rectangle is $G$. We use different colors for the partition of $F = F_1 \cup F_2$, and it has the following bijections.
\begin{align*}
F_1 \cap F_2 &= \{\text{yellow, purple}\},\\
G \cap F_1 \cap F_2 &= \{\text{yellow}\},\\
(F_1 \cap F_2) \setminus G &= \{\text{purple}\},\\
(F_1 \setminus F_2) \setminus G &= \{\text{red}\},\\
(F_2 \setminus F_1) \setminus G &= \{\text{blue}\},\\
F_1 \setminus ((G \cap F_1) \setminus F_2) &= \{\text{red, yellow, purple}\},\\
F_2 \setminus ((G \cap F_2) \setminus F_1) &= \{\text{blue, yellow, purple}\}.\\ 
\end{align*}
}
\label{fig:Venn}
\end{figure}
Therefore, one term in the {\rhs} of \eqref{eq:KeyGeneralInter} can be bounded 
\begin{multline*}
\int_{(\Rd)^{F \setminus G}}\E\Ll[ \frac{1}{\rho_0 \vert \cu_m \vert} \int_{\cu_m} \Ll\vert \int_{(\Rd)^{G}}D_{F_2} (\a) (D_{F_1} \nabla \psi_m) \Rr \vert^2  \, \d \mu\Rr] \\
\leq \int_{(\Rd)^{F_1 \setminus ((G \cap F_1) \setminus F_2)}} \E\Ll[ \frac{1}{\rho_0 \vert \cu_m \vert} \int_{\cu_m}   \Ll( \Upsilon(F_2 \cap F_1, \cdot)
\Ll\vert \int_{(\Rd)^{(G \cap F_1) \setminus F_2}} D_{F_1} \nabla \psi_m\Rr\vert^2\Rr)   \, \d \mu \Rr].
\end{multline*}

We can further drop out the indicator $\Upsilon(F_2 \cap F_1, \cdot)$, and put it back to \eqref{eq:KeyGeneralInter} to obtain that 
\begin{multline}\label{eq:KeyGeneralRecurrence}
\Ll(\int_{(\Rd)^{F \setminus G}}\E\Ll[ \frac{1}{\rho_0 \vert \cu_m \vert} \int_{\cu_m} \Ll\vert \int_{(\cu_m)^{G}} (D_F\nabla \psi_m)\Rr \vert^2  \, \d \mu \Rr]\Rr)^{\frac{1}{2}} \\
\leq  \sum_{\substack{F_1 \cup F_2 = F\\
F_2 \cap G \neq \emptyset, \text{ or }  F_1 \subsetneq F}} \Ll(4^{\vert F_2 \vert}\Lambda^2 \int_{(\Rd)^{F_1 \setminus ((G \cap F_1) \setminus F_2)}} \E\Ll[ \frac{1}{\rho_0 \vert \cu_m \vert} \int_{\cu_m} \Ll\vert \int_{(\Rd)^{(G \cap F_1) \setminus F_2}} D_{F_1} \nabla \psi_m\Rr\vert^2 \, \d \mu \Rr]\Rr)^{\frac{1}{2}}.
\end{multline}

\textit{Step 3: Induction argument.}  Equation \eqref{eq:KeyGeneralRecurrence} is the analogue of \eqref{eq:KeyBasisRecurrence} for the general case. In this step, we describe the induction argument, which consists in obtaining a bound for the constant $C(i,j)$ in \eqref{eq:KeyEstimate} in terms of a linear combination of the $C(i',j')$ with $i' \le i$, $j' \le j$, and $i' + j' <  i + j$. An illustration is in Figure~\ref{fig:recurrence}.

We denote by $\tilde{I}(m, \rho_0, F, G)$ the {\lhs} of \eqref{eq:KeyGeneralRecurrence}. This equation can be rewritten as
\begin{align}
\label{e.induction.inequality}
\tilde{I}(m, \rho_0, F, G) \leq \sum_{\substack{F_1 \cup F_2 = F\\
F_2 \cap G \neq \emptyset, \text{ or }  F_1 \subsetneq F}} 2^{\vert F_2 \vert}\Lambda\,   \tilde{I}(m, \rho_0, F_1, (G \cap F_1) \setminus F_2).
\end{align}
For sets $F_1, F_2$ as in the summands above, we clearly have 
\begin{equation*}  
|F_1| + |(G \cap F_1) \setminus F_2| \le |F| + |G|.
\end{equation*}
In fact, the inequality is always strict. Indeed, a possible case of equality would require that $F_1 = F$, since $F_1 \subset F$ and $(G \cap F_1) \setminus F_2 \subset G$. But if $F_1 = F$, then we must have $F_2 \cap G \neq \emptyset$, and thus 
\begin{equation*}  
\vert (G \cap F_1) \setminus F_2 \vert = \vert G  \setminus F_2 \vert \leq \vert G\vert - 1.
\end{equation*}
So all the summands on the right side of \eqref{e.induction.inequality} are such that 
\begin{equation*}  
|F_1| + |(G \cap F_1) \setminus F_2| < |F| + |G|.
\end{equation*}
The induction argument is then clear: the case when $F = G = \emptyset$ is the basic Dirichlet energy estimate. Next, assuming the boundedness of $\tilde{I}(m, \rho_0, F, G) $ for $|F| + |G| \le k$, we can obtain the result for $|F| + |G| = k+1$ by an application of \eqref{e.induction.inequality}.
This completes the proof of Proposition~\ref{prop:KeyEstimate}.
\end{proof}
\begin{figure}[ht]
\centering
\includegraphics[scale=0.6]{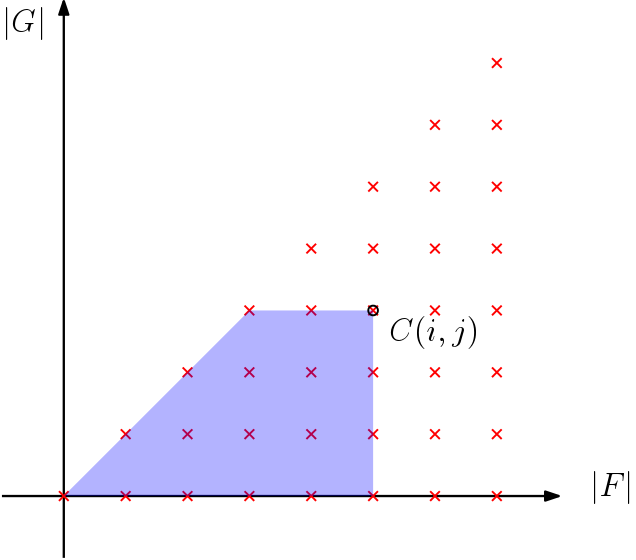}
\caption{An illustration of the recurrence argument. The constant $C(i,j)$ can be bounded by a linear combination of the $C(i',j')$ with $j' \leq j$, $i' \leq i$, and $i' + j' < i + j$.}\label{fig:recurrence}
\end{figure}

\subsection{Smoothness in finite volume}
\label{sub:proof.prop}

We can now combine Lemma~\ref{lem:Expansion} and Proposition~\ref{prop:KeyEstimate} to complete the proof of Proposition~\ref{prop:SmoothFinite}.

\begin{proof}[Proof of Proposition~\ref{prop:SmoothFinite}]
We decompose the proof into three steps.

\textit{Step 1: Decomposition and expansion.}
As stated in Subsection~\ref{subsec:Strategy}, we first expand $\Delta^\rho_m$ with respect to $\rho$ as in \eqref{eq:ExpansionModi} and use the Leibniz formula \eqref{eq:Leibniz2} to get that 
\begin{align}\label{eq:ExpansionCopy}
\Delta^\rho_m(\rho_0)  &= \sum_{k=1}^{\infty} \frac{\rho^k}{k!} c_{k,m}(\rho_0) = \sum_{k=1}^{\infty} \frac{\rho^k}{k!} \sum_{E \cup F = \dbint{1,k}} I(m,\rho_0,E,F),
\end{align}
with $c_{k,m}$ defined in \eqref{e.def.ckm} and $I(m,\rho_0,E,F)$ defined in \eqref{eq:defI}. Lemma~\ref{lem:Expansion} ensures that this series converges, and that $\rho \mapsto \Delta^\rho_m$ is analytic. In the next step, we aim to give a bound to $I(m, \rho_0, E, F)$ which is uniform with respect to $m$ and $\rho_0$.

\textit{Step 2: Reduction of $I(m, \rho_0, E, F)$.} Recall the expression of $I(m, \rho_0, E, F)$ in \eqref{eq:defI}, we use Fubini's lemma and pass the integration $\dbint{1,k} \setminus E$ inside. Notice that we have $\dbint{1,k} \setminus E = F \setminus E$ thanks to $ E \cup F = \dbint{1,k}$. Since the particles in the set $F \setminus E$ do not appear in $D_E (\a - \a^\# )$, we have 
\begin{align*}
I(m,\rho_0,E,F)
= \int_{(\Rd)^{E}} \E\Ll[ \frac{1}{\rho_0 \vert \cu_m \vert} \int_{\cu_m} \nabla \psi_m \cdot  D_E (\a - \a^\# )  \Ll(\int_{(\Rd)^{F \setminus E}} D_{F} \nabla \psi_m\Rr) \, \d \mu \Rr].
\end{align*}
We apply the Cauchy--Schwarz inequality and obtain that 
\begin{align*}
\vert I(m,\rho_0,E,F) \vert &\leq \Ll(\int_{(\Rd)^{E}} \E\Ll[ \frac{1}{\rho_0 \vert \cu_m \vert} \int_{\cu_m}   \vert D_E (\a - \a^\# )\vert  \vert \nabla \psi_m\vert^2  \, \d \mu \Rr]\Rr)^{\frac{1}{2}} \\
& \quad \times \Ll(\int_{(\Rd)^{E}} \E\Ll[ \frac{1}{\rho_0 \vert \cu_m \vert} \int_{\cu_m}  \vert D_E (\a - \a^\# ) \vert  \Ll\vert\int_{(\Rd)^{F \setminus E}} D_F\nabla \psi_m\Rr\vert^2 \, \d \mu \Rr]\Rr)^{\frac{1}{2}}.
\end{align*}
The first term is easy to treat since we can use Fubini's lemma that 
\begin{align*}
&\int_{(\Rd)^{E}} \E\Ll[ \frac{1}{\rho_0 \vert \cu_m \vert} \int_{\cu_m}  \vert D_E (\a - \a^\# ) \vert  \vert\nabla \psi_m\vert^2  \, \d \mu \Rr] \\
& =  \E\Ll[ \frac{1}{\rho_0 \vert \cu_m \vert} \int_{\cu_m}  \Ll(\int_{(\Rd)^{E}} \vert D_E (\a - \a^\# )\vert\Rr)   \vert\nabla \psi_m\vert^2  \, \d \mu \Rr] \\
& \leq 2^{\vert E \vert}\Lambda.
\end{align*}
In the last step, we apply \eqref{eq:DAEstimate} and \eqref{eq:UpsilonInt} that 
\begin{align*}
\int_{(\Rd)^{E}} \vert D_E (\a - \a^\# )\vert \leq 2^{\vert E \vert}\Lambda\int_{(\Rd)^{E}}\Upsilon(E, \cdot) \leq 2^{\vert E \vert} \Lambda.
\end{align*}
For the second term, we use  the decomposition 
\begin{align*}
D_E = D_{E \setminus F} \circ D_{E \cap F}, \qquad \int_{(\Rd)^E} = \int_{(\Rd)^{E \setminus F}} \int_{(\Rd)^{E \cap F}},
\end{align*}
and pass the integration $\int_{(\Rd)^{E \setminus F}}$ inside
\begin{align*}
&\int_{(\Rd)^{E}} \E\Ll[ \frac{1}{\rho_0 \vert \cu_m \vert} \int_{\cu_m}  \vert D_E (\a - \a^\# ) \vert  \Ll\vert\int_{(\Rd)^{F \setminus E}} D_F\nabla \psi_m\Rr\vert^2 \, \d \mu \Rr] \\
& = \int_{(\Rd)^{E \cap F}} \E\Ll[ \frac{1}{\rho_0 \vert \cu_m \vert} \int_{\cu_m} \Ll(\int_{(\Rd)^{E \setminus F}}   \vert D_E (\a - \a^\# ) \vert\Rr)  \Ll\vert \int_{(\Rd)^{F \setminus E}} D_F\nabla \psi_m\Rr \vert^2 \, \d \mu \Rr]. 
\end{align*}
We apply once again the estimate \eqref{eq:UpsilonInt} and \eqref{eq:DAEstimate} that 
\begin{align*}
\int_{(\Rd)^{E \setminus F}}   \vert D_E (\a - \a^\# ) \vert \leq 2^{\vert E \vert} \int_{(\Rd)^{E \setminus F}} \Upsilon(E,\cdot) \leq  2^{\vert E \vert}\Lambda \Upsilon(E \cap F, \cdot) \leq 2^{\vert E \vert}\Lambda.
\end{align*}
Therefore, we bound the second term by 
\begin{multline*}
\int_{(\Rd)^{E}} \E\Ll[ \frac{1}{\rho_0 \vert \cu_m \vert} \int_{\cu_m}  \vert D_E (\a - \a^\# ) \vert  \Ll\vert\int_{(\Rd)^{F \setminus E}} D_F\nabla \psi_m\Rr\vert^2 \, \d \mu \Rr] \\
\leq 2^{\vert E \vert} \Lambda \int_{(\Rd)^{E \cap F}} \E\Ll[ \frac{1}{\rho_0 \vert \cu_m \vert} \int_{\cu_m} \Ll\vert \int_{(\Rd)^{F \setminus E}} D_F\nabla \psi_m\Rr \vert^2 \, \d \mu \Rr].
\end{multline*}
Here we apply the key estimate Proposition~\ref{prop:KeyEstimate} to conclude the proof a the uniform bound of $I(m, \rho_0, E, F)$ with respect to $m$ and $\rho_0$. This also implies the uniform bound \eqref{eq:ckmBound} for $c_{k,m}(\rho_0)$.

\textit{Step 3: Control of the tail $R_k$.} In this step, we need to control the tail in the expansion of $\Delta^\rho_m(\rho_0)$. 
Even if one were to keep track of the dependence on $k$ in the upper bound $\vert c_{k,m}(\rho_0) \vert \leq C_k$ obtained above, one cannot ensure the summability of the series $\sum_{j> k}\frac{C_j}{j!}\rho^j$. On the other hand, we know the function $\rho \mapsto \Delta^\rho_m(\rho_0)$ is indeed analytic for any fixed $\rho_0 \in \R_+$ (using a naive bound of $c_{k,m}$ depending on $m$, see~Lemma~\ref{lem:Expansion} and its proof in Subsection~\ref{subsec:Expansion}), so for $c_{j,m}$ defined in Proposition~\ref{prop:SmoothFinite} and $\rho_0 > 0$, we have 
\begin{align*}
c_{j,m}(\rho_0) = \Ll(\frac{\d}{\d \rho}\Rr)^j_{\vert \rho = 0} \Delta^\rho_m(\rho_0).  
\end{align*}
We also write $\partial^j \Delta^{\rho}_m$ as a shorthand notation for the $j$-th derivative at $\rho$,
\begin{align*}
\partial^j \Delta^{\rho}_m(\rho_0) := \Ll(\frac{\d}{\d \rho}\Rr)^j \Delta^\rho_m(\rho_0).
\end{align*}
Then we apply Taylor's expansion for the function $\rho \mapsto \Delta^{\rho}_m(\rho_0)$ until order $k$
\begin{align}\label{eq:Taylor1}
\Delta^{\rho}_m(\rho_0) = \sum_{j = 0}^k \frac{\partial^j \Delta^{0}_m(\rho_0)}{j!} \rho^j + \int_{0}^\rho \frac{\partial^{k+1}\Delta^{s}_m(\rho_0)}{k!} s^k \, \d s.
\end{align}
Recalling the definition of $\Delta^{\rho}_m$ in \eqref{eq:defPerApp}, we have 
\begin{align*}
\partial^{k+1} \Delta^{s}_m(\rho_0) &= \Ll(\frac{\d}{\d \rho}\Rr)^j_{\vert \rho = s} \Ll( q \cdot \ab_*^{-1}(\cu_m, \rho_0 + \rho) q - q \cdot \ab_*^{-1}(\cu_m, \rho_0 ) q\Rr) \\
&= \Ll(\frac{\d}{\d \rho}\Rr)^j_{\vert \rho = 0} \Ll( q \cdot \ab_*^{-1}(\cu_m, \rho_0 + s+ \rho) q - q \cdot \ab_*^{-1}(\cu_m, \rho_0 + s) q \Rr) \\
&= \partial^{k+1} \Delta^{0}_m(\rho_0 + s).
\end{align*}
Since $\partial^{k+1}\Delta^{0}_m(\rho_0 + s) = c_{k+1,m}(\rho_0+s)$, upon inserting this back into  \eqref{eq:Taylor1}, it follows that
\begin{align}\label{eq:Taylor2}
\Delta^{\rho}_m(\rho_0) = \sum_{j = 0}^k \frac{c_{j,m}(\rho_0)}{j!} \rho^j + \int_{0}^\rho \frac{c_{k+1,m}(\rho_0 + s)}{k!} s^k \, \d s.
\end{align}
This gives us an expression for the remainder of order $k$ in~\eqref{eq:ExpansionFinite}, which is
\begin{align}\label{eq:defRemainder}
R_k(m, \rho_0, \rho) := \int_{0}^\rho \frac{c_{k+1,m}(\rho_0 + s)}{k!} s^k \, \d s.
\end{align}
Using the uniform estimate~\eqref{eq:ckmBound} of $c_{k+1,m}(\rho_0+s)$ with respect to $\rho_0 + s$ and $m$, the remainder is of order $O(\rho^{k+1})$ independent of $\rho_0$ and $m$. This finishes our proof of Proposition~\ref{prop:SmoothFinite}.
\end{proof}

\subsection{Proof of the main theorem}
\label{sub:proof.main}

In this final subsection, we conclude the proof of the main Theorem~\ref{t.smooth}, using Proposition~\ref{prop:SmoothFinite}.

\begin{proof}[Proof of Theorem~\ref{t.smooth}]
As a first step, we show the existence of the limit in~\eqref{e.def.ck}. Let $k \geq 2$ and assume by induction that the existence of $c_j(\rho_0)$ is established for $1 \leq j \leq k-1$ and $\rho_0 > 0$ (recall that the existence of $c_1(\rho_0)$ follows from Theorem~\ref{t.C11}). For $\rho_0 > 0$, the sequence $\{ c_{k,m}(\rho_0) \}_{m\in \N}$ defined in~\eqref{e.def.ckm} is bounded by some positive constant $C_k(d,\Lambda)$ using~Proposition~\ref{prop:SmoothFinite}. Thus, there exists a subsequence $\{ c_{k,m_\ell}(\rho_0) \}_{\ell\in \N}$ (possibly depending on $\rho_0$) such that 
$$
c^*_k(\rho_0) := \lim_{\ell \to +\infty }c_{k,m_\ell}(\rho_0)
$$
exists. By~\eqref{eq:ExpansionFinite}, one has for $\rho > 0$
\begin{equation*}
\Ll\vert \Delta^\rho_{m_\ell}(\rho_0) - \sum_{j = 1}^{k-1} \frac{c_{j,m_\ell}(\rho_0)\rho^j }{j!} - \frac{c_{k,m_\ell}(\rho_0)}{k!} \rho^k \Rr\vert \leq |R_k(m_\ell,\rho_0,\rho)|,
\end{equation*}
and passing to the limit $\ell \rightarrow \infty$ yields (upon using~Proposition~\ref{prop:SmoothFinite}, the induction hypothesis and~\eqref{limit.increments}) 
\begin{equation}
\label{eq:LocUnifBoundDeltaRho}
\sup_{\rho_0 \in (0,\infty)} \Ll\vert \Delta^\rho(\rho_0) - \sum_{j = 1}^{k-1} \frac{c_j(\rho_0)\rho^j }{j!} - \frac{c_k^\ast(\rho_0)}{k!} \rho^k \Rr\vert \leq O(\rho^{k+1}), \qquad \text{for }\rho > 0.
\end{equation}
In particular,~\eqref{eq:LocUnifBoundDeltaRho} implies that $c^\ast_k(\rho_0)$ is the unique limit of the full sequence $\{ c_{k,m}(\rho_0) \}_{m\in \N}$, and we denote it by $c_k(\rho_0)$. This proves~\eqref{e.def.ck}. We note in passing that
\begin{equation}
\label{eq:BoundednessInfVolCoeff}
\text{$c_1,\cdots,c_k : (0,\infty) \rightarrow \R$ are bounded functions,}
\end{equation}
which follows by~\eqref{e.def.ck} and $|c_{k,m}| \leq C_k(d,\Lambda)$, see~Proposition~\ref{prop:SmoothFinite}.
Thus, we can write
\begin{equation}
\label{eq:ExpansionPosRho}
\Delta^\rho(\rho_0) = \sum_{j = 1}^k \frac{c_j(\rho_0)}{j!}\rho^j + O(\rho^{k+1}), \qquad \text{for } \rho > 0,
\end{equation}
with the error term independent of $\rho_0$. To simplify notation, we again set $f(\cdot) := q \cdot \ab^{-1}(\cdot )q$. We claim that the expansion~\eqref{eq:ExpansionPosRho} implies that 
\begin{equation}
\label{eq:LipCont}
\text{$c_1,\cdots,c_k : (0,\infty) \rightarrow \R$ are Lipschitz-continuous,}
\end{equation}
and moreover
\begin{equation}
\label{eq:differentiability_f}
\text{$f$ has $k$ derivatives, and }f^{(j)}(\rho_0) = c_j(\rho_0), \ 1 \leq j \leq k, \ \rho_0 > 0.
\end{equation}
We first define the forward difference of order $\ell$ of $f$ at $\rho_0 > 0$, $\ell \in \N_+$, with step size $\rho \geq 0$ as 
\begin{equation}
\label{eq:ForwardDef}
\Delta^{\ell,\rho}[f](\rho_0) = \Delta^{\ell,\rho}(\rho_0) := \sum_{i = 0}^\ell \binom{\ell}{i} (-1)^{\ell-i} f(\rho_0 + i\rho),
\end{equation}
(note that with our fixed choice of $f$, $\Delta^{1,\rho}(\rho_0) = \Delta^\rho(\rho_0)$). We claim that these quantities fulfill for $1 \leq \ell \leq k$, $\rho_0 > 0$ and $\rho \geq 0$,
\begin{equation}
\label{eq:ForwardEq}
\Delta^{\ell,\rho}(\rho_0)  = c_\ell(\rho_0)\rho^\ell + O(\rho^{\ell+1}),
\end{equation}
with the error term independent of $\rho_0$, and for $1 \leq \ell < k$, $\rho_0 > 0$ and $\rho \geq 0$,
\begin{equation}
\label{eq:RecursionDiff}
\Delta^{\ell,\rho}(\rho_0+\rho) - \Delta^{\ell,\rho}(\rho_0)  = \Delta^{\ell+1,\rho}(\rho_0).
\end{equation}
We prove~\eqref{eq:ForwardEq}. To this end, we infer from~\eqref{eq:ExpansionPosRho} that
\begin{equation}
\label{eq:ExpansionPosRhoRewritten}
f(\rho_0 + i\rho) = \sum_{j = 0}^k \frac{c_j(\rho_0)}{j!}(i\rho)^j + O(\rho^{k+1}),\qquad \text{for } \rho > 0, 1 \leq i \leq k,
\end{equation}
where we defined for convenience $c_0(\rho_0) := f(\rho_0)$. Equation~\eqref{eq:ExpansionPosRhoRewritten} is now inserted into~\eqref{eq:ForwardDef}, which yields for $\rho_0 > 0$, $\rho \geq 0$ and $1 \leq \ell \leq k$ that
\begin{align*}
\Delta^{\ell,\rho}(\rho_0) & = \sum_{i = 0}^\ell \binom{\ell}{i} (-1)^{\ell-i} \left\{ \sum_{j = 0}^\ell \frac{c_j(\rho_0)}{j!} (i\rho)^j + \sum_{j = \ell + 1}^k \frac{c_j(\rho_0)}{j!} (i\rho)^j + O(\rho^{k+1})
\right\} \\
& = \sum_{j = 0}^\ell \frac{c_j(\rho_0)}{j!} \rho^j \left(\sum_{i = 0}^\ell \binom{\ell}{i} (-1)^{\ell-i} i^j \right) + O(\rho^{\ell+1}) \\
& = c_\ell(\rho_0) \rho^\ell + O(\rho^{\ell+1}).
\end{align*}
Here, we combined the terms involving $\rho^j$ with $j \in \{\ell+1,\cdots,k\}$ and $O(\rho^{k+1})$ into a contribution $O(\rho^{\ell+1})$ and using~\eqref{eq:BoundednessInfVolCoeff} in going from the first to the second line. From the second to the third line, we use the fact that for any polynomial $P$ with real coefficients $A_0,\cdots,A_k$, i.e., ${P(X) = A_\ell X^\ell + \cdots + A_0}$ of degree smaller or equal to $\ell$, one has ${\sum_{i = 0}^\ell \binom{\ell}{i} (-1)^{\ell-i} P(i) = \ell! A_\ell}$. Equation~\eqref{eq:RecursionDiff} also follows directly from elementary properties of binomial coefficients. Indeed:  
\begin{align*}
\Delta^{\ell,\rho}(\rho_0+\rho) - \Delta^{\ell,\rho}(\rho_0) & = \sum_{i = 0}^\ell (-1)^{\ell-i} \binom{\ell}{i} f(\rho_0+(i+1)\rho) - \sum_{i = 0}^\ell (-1)^{\ell-i} \binom{\ell}{i} f(\rho_0+i\rho) \\
& = \sum_{i = 1}^{\ell+1} (-1)^{\ell - i +1} \binom{\ell}{i-1} f(\rho_0+i\rho) + \sum_{i = 0}^\ell (-1)^{\ell - i +1} \binom{\ell}{i} f(\rho_0+i\rho) \\
& = \sum_{i = 0}^{\ell+1} (-1)^{\ell+1-i} \left\{\binom{\ell}{i-1} + \binom{\ell}{i}   \right\} f(\rho_0+i\rho) \\
& = \sum_{i = 0}^{\ell+1} (-1)^{\ell+1-i} \binom{\ell+1}{i} f(\rho_0+i\rho) = \Delta^{\ell+1,\rho}(\rho_0).
\end{align*}

Identity~\eqref{eq:ForwardEq} is now used at $\rho_0+\rho$ and $\rho_0$ on the left-hand side of~\eqref{eq:RecursionDiff}, and at $\rho_0$ on the right-hand side of the same equation (recall that the $O(\rho^\ell)$ resp. $O(\rho^{\ell+1})$ terms do not depend on $\rho_0$):
\begin{equation}
\begin{split}
& \frac{1}{\rho^\ell} (\Delta^{\ell,\rho}(\rho_0 + \rho ) - \Delta^{\ell,\rho}(\rho_0)) = \frac{1}{\rho^\ell} \Delta^{\ell+1,\rho}(\rho_0) \\
\Rightarrow \qquad & c_\ell(\rho_0 + \rho) - c_\ell(\rho_0) = c_{\ell+1}(\rho_0)\rho + O(\rho), \qquad \text{for }\rho > 0.
\end{split}
\end{equation}
By the boundedness of $c_{\ell+1}$~\eqref{eq:BoundednessInfVolCoeff}, this establishes the Lipschitz-continuity~\eqref{eq:LipCont} of $c_\ell$. \\

Now we prove the differentiability~\eqref{eq:differentiability_f}: By induction, suppose that we already established that $f^{(\ell-1)}(\rho_0) = c_{\ell-1}(\rho_0)$ for all $\rho_0 \in (0,\infty)$, and $1 \leq \ell < k$. Now, for $\rho > 0$, one has 
\begin{equation}
\begin{split}
\Delta^{\ell-1,\rho}(\rho_0) & = 
\sum_{i = 0}^{\ell-1} \binom{\ell-1}{i} (-1)^{\ell-1-i} \left\{ \sum_{j = 0}^{\ell-1} \frac{c_j(\rho_0)}{j!} (i\rho)^j + \frac{c_\ell(\rho_0)}{\ell!} (i\rho)^\ell\right\} + O(\rho^{\ell+1}) \\
&= c_{\ell-1}(\rho_0) \rho^{\ell-1} + \underbrace{\sum_{i = 0}^{\ell-1} \binom{\ell-1}{i} (-1)^{\ell-1-i} \frac{i^{\ell}}{\ell!}}_{=:C(\ell)} c_\ell(\rho_0) \rho^\ell + O(\rho^{\ell+1}),
\end{split}
\end{equation}
having used the same arguments as in the proof of~\eqref{eq:ForwardEq}, with $C(\ell) \in \R$ some numerical constant. The latter gives us that for $\rho > 0$,
\begin{equation}
\label{eq:Rightderiv}
\begin{split}
\frac{1}{\rho^\ell} \Delta^{\ell,\rho}(\rho_0) & \stackrel{\eqref{eq:RecursionDiff}}{=} \frac{1}{\rho} \left\{ \frac{1}{\rho^{\ell-1}} \Delta^{\ell-1}(\rho_0+\rho) - \frac{1}{\rho^{\ell-1}} \Delta^{\ell-1}(\rho_0)  \right\} \\
& = \underbrace{\frac{1}{\rho} (c_{\ell-1}(\rho_0+\rho) - c_{\ell-1}(\rho_0))}_{ = \frac{1}{\rho}(f^{(\ell-1)}(\rho_0 + \rho) - f^{(\ell-1)}(\rho_0))} + C(\ell)(c_\ell(\rho_0 + \rho) - c_\ell(\rho_0)) + O(\rho).
\end{split}
\end{equation}
On the other hand, the left-hand side of the equation above is also equal to $c_\ell(\rho_0) + O(\rho)$. Letting $\rho \downarrow 0$ then shows that the right-derivative of $f^{(\ell-1)}$ at $\rho_0$ exists and equals $c_\ell(\rho_0)$, upon using~\eqref{eq:LipCont} for $c_\ell$. Replacing $\rho_0$ by $\rho_0-\rho$ in~\eqref{eq:Rightderiv} then gives 
\begin{equation}
\frac{1}{\rho}(f^{(\ell-1)}(\rho_0) - f^{(\ell-1)}(\rho_0-\rho)) = c_\ell(\rho_0-\rho) + O(\rho),
\end{equation}
from which one can then infer the left-derivative of $f^{(\ell-1)}$ as well (using once more~\eqref{eq:LipCont}). This finishes the proof of~\eqref{eq:differentiability_f}. Since $k \in \N_+$ was arbitrary, the proof of Theorem~\ref{t.smooth} is complete.
\end{proof}

\section{Local uniform convergence}
\label{s.local.unif}

The aim of this section is to strengthen the pointwise convergence of the sequences $(\ab(\cu_m,\rho_0))_{m \geq 1}$ and $(\ab_*(\cu_m,\rho_0))_{m \geq 1}$ towards $\ab(\rho_0)$ for each fixed $\rho_0 > 0$  (see~\eqref{eq:defab} and below) to a locally uniform convergence, that is to show the following statement. 

\begin{proposition}
\label{p.loc.unif}
The mappings $\ab(\cu_m,\cdot)$ and $\ab_*(\cu_m,\cdot)$ both converge to $\ab(\cdot)$ locally uniformly over $[0,\infty)$ as $m$ tends to infinity. Moreover, for every $k \in \N_+$, the sequence of approximate derivatives $c_{k,m}$ converges locally uniformly to $c_k$, as $m$ tends to infinity (recall~\eqref{e.def.ckm} and~\eqref{e.def.ck} for the respective definitions).
\end{proposition}

The local uniform convergence of $\ab(\cu_m,\cdot)$ and $\ab_*(\cu_m,\cdot)$ could in fact be obtained as a consequence of the quantitative estimate~\eqref{eq:rate} and the observation that the exponent $\al > 0$ and the constant $C < \infty$ appearing there can be chosen locally uniformly over $\rho_0 > 0$. However, we think it useful to point out that Proposition~\ref{p.loc.unif} is actually a rather straightforward consequence of the qualitative statement that, for each fixed $\rho_0 > 0$,
\begin{equation}
\label{e.identity.limits}
\ab(\rho_0) = \lim_{m \to \infty} \ab(\cu_m,\rho_0) = \lim_{m \to \infty} \ab_*(\cu_m,\rho_0).
\end{equation}
As will be seen, once \eqref{e.identity.limits} is granted, the fact that these sequences converge locally uniformly as $\rho_0$ varies is an application of Dini's theorem.

Since we will need to show the continuity of $\ab(\cu_m, \cdot)$, we first need to develop some version of Lemma~\ref{lem:J} geared towards $\nu(U,p,\rho_0)$ instead of $\nu_\ast(U,q,\rho_0)$. To state it, we define for a bounded open $U \subseteq \R^d$ the function space $\cD(U)$ to consist of sequences of functions ${f = (f_n)_{n \geq 0}}$, where $f_n : U^n \rightarrow \R$ satisfy
\begin{enumerate}
\item $f_0$ is a constant and for every $n \in \N_+$, $f_n \in C^{\infty}(U^n)$;
\item There exists a compact set $K \subset U$ such that for any $x_i \notin K$
\begin{align*}
f_n(x_1, \cdots, x_{i-1}, x_i, x_{i+1}, \cdots, x_n) = f_{n-1}(x_1, \cdots, x_{i-1}, x_{i+1}, \cdots, x_n).
\end{align*}
\end{enumerate}
The canonical projection in \eqref{eq:Projection} is an injection from $\cC^{\infty}_c(U)$ to $\cD(U)$; in other words, we can think of $\cC^\infty_c(U)$ as a subset of $\cD(U)$.

We then define a version of the minimization problem in the first line of \eqref{eq:defNu} with $\cD(U)$ replacing $\cH^1_0(U)$. Define for $f \in \cD(U)$ the quantity
\begin{multline}\label{eq:defKFunc}
\mcl K(f, U, p, \rho_0) \\
:= \frac{e^{-\rho_0 \vert U \vert} }{2 \rho_0 \vert U \vert}\sum_{n=0}^{\infty} \frac{(\rho_0 \vert U \vert)^n}{n!} \Ll(\fint_{U^n} \sum_{i=1}^n (p + \nabla_{x_i} f_n) \cdot \a\Ll(\sum_{k=1}^n \delta_{x_k}, x_i\Rr)(p + \nabla_{x_i} f_n) \, \d x_1 \cdots \d x_n\Rr).
\end{multline}
With this definition, one has the following result.

\begin{lemma}\label{lem:Enlarge}
For every bounded open set $U$,  
$\nu(U,p, \rho_0) = \inf_{f \in \cD(U)} \mcl K(f, U, p, \rho_0)$.
\end{lemma}
\begin{proof}

For every $f = (f_n)_{n \geq 0} \in \cD(U)$, we consider the symmetrization $\widetilde{f} = (\widetilde{f}_n)_{n \geq 0}$ by defining $\widetilde{f}_n = \frac{1}{n!}\sum_{\sigma \in S_n} f(x_{\sigma(1)}, \cdots , x_{\sigma(n)})$. This function fulfills $\widetilde{f} \in \cD(U)$: Indeed $\widetilde{f}_n \in C^\infty(U^n)$ follows directly from the definition, and letting $K \subseteq U$ be the compact set associated with $f$, one has e.g.~for the case $x_1 \notin K$  
\begin{align*}
\tilde{f}_n(x_1, x_2, \cdots, x_n) &= \frac{1}{n!}\sum_{\sigma \in S_n} f_n(x_{\sigma(1)}, x_{\sigma(2)}, \cdots, x_{\sigma(n)}) \\
&= \frac{1}{(n-1)!}\sum_{\sigma \in S_{n-1}(\{2,\cdots,n\})} f_{n-1}(x_{\sigma(2)}, x_{\sigma(3)}, \cdots, x_{\sigma(n)}) \\
&= \tilde{f}_{n-1}(x_2, \cdots, x_n).
\end{align*}
Here, from the first line to the second line, we can remove $x_1$ in the function, and make use of the natural $n$-to-$1$ bijection from the group of permutations $S_n$ to the group of permutations $S_{n-1}(\{2,\cdots,n\})$. This establishes the second condition for functions in $\cD(U)$, so $\tilde{f} \in \cD(U)$. 

By an application of Jensen's inequality, it follows that
\begin{align*}
\mcl K(\tilde{f}, U, p, \rho_0) & \leq \frac{e^{-\rho_0 \vert U \vert} }{2 \rho_0 \vert U \vert}\sum_{n=0}^{\infty} \frac{(\rho_0 \vert U \vert)^n}{n!}  \frac{1}{n!} \sum_{\sigma \in S_n} \Ll(\fint_{U^n} \sum_{i=1}^n  \Ll((p + \nabla_{x_i}   f_n(x_{\sigma(1)}, \cdots, x_{\sigma(n)})) \Rr. \Rr. \\
& \qquad  \Ll.\Ll. \cdot \a\Ll(\sum_{k=1}^n \delta_{x_k}, x_i\Rr)(p + \nabla_{x_i}   f_n(x_{\sigma(1)}, \cdots, x_{\sigma(n)}))\Rr) \, \d x_1 \cdots \d x_n\Rr),
\end{align*}
which implies that $\mcl K(\tilde{f}, U, p, \rho_0)  \leq \mcl K(f, U, p, \rho_0)$. This establishes that the value $\inf_{f \in \cD(U)} \mcl K(f, U, p, \rho_0)$ can be attained on the subspace with invariance by permutation, which can be identified as $\cC^{\infty}_c(U)$.
\end{proof}

\begin{proof}[Proof of Proposition~\ref{p.loc.unif}]

We need to verify that 
\begin{equation}
\label{eq:Continuity_finite_m}
    \text{For fixed $m \in \N$, $\ab(\cu_m, \cdot)$ and $\ab_*(\cu_m, \cdot)$ are continuous.}
\end{equation}

Once~\eqref{eq:Continuity_finite_m} is established, the uniform convergence follows from Dini's theorem, which states that if a decreasing or increasing sequence of continuous functions $(f_n)_{n \geq 1}$ converges pointwisely to a continuous function $f$, then the convergence is a locally uniform. Recall that $(\ab(\cu_m, \cdot))_{m \geq 1}$ is decreasing and $(\ab_*(\cu_m, \cdot))_{m \geq 1}$ is increasing and the common limit~\eqref{e.identity.limits} is ensured by \cite[Theorem 1.1]{bulk}. Moreover, note that $\ab(\cdot)$ is continuous by Theorem~\ref{t.smooth} (in fact, to establish the continuity of $\ab(\cdot)$ it suffices to establish its upper and lower semicontinuity, which follows from the monotone convergence of $\ab(\cu_m,\cdot)$ and $\ab_*(\cu_m,\cdot)$, respectively). Therefore, it suffices to justify the continuity condition~\eqref{eq:Continuity_finite_m}.

\textit{Step 1: Continuity of $\ab_*(\cu_m, \cdot)$.} The continuity of $\ab_*^{-1}(\cu_m, \cdot)$ follows immediately from~\eqref{expansion.11.m}, and this implies the continuity of $\ab_*(\cu_m, \cdot)$. 

\textit{Step 2: Continuity of $\ab(\cu_m, \cdot)$.} 
We use the exact expression of the subadditive quantity
\begin{align*}
\nu(\cu_m, p,\rho_0+\rho) & = p \cdot \ab(\cu_m, \rho_0 + \rho) p  \\
& = \E \Ll[ \frac{1}{(\rho_0 + \rho) \vert \cu_m \vert} \int_{\cu_m} (p + \nabla \phi^{\rho}_{m}) \cdot \a^\rho (p + \nabla \phi^{\rho}_{m}) \, \d (\mu + \mu_\rho)\Rr],
\end{align*} 
for $m \in \N$, $p \in \R^d$ and $\phi^{\rho}_m$ denotes the minimizer in the definition of $\nu(\cu_m,p,\rho_0+\rho)$.
We now derive an upper bound on the above expression. Using Lemma~\ref{lem:Enlarge}, we know that $\phi_{m}(\mu)$ is a sub-minimizer for the problem $\nu(\cu_m, p,\rho_0 + \rho)$ with density $\rho + \rho_0$. Also with the help of Mecke's identity \eqref{mecke}, we obtain that 
\begin{align*}
& p \cdot \ab(\cu_m, \rho_0 + \rho) p \\
& \leq \E \Ll[ \frac{1}{(\rho_0 + \rho) \vert \cu_m \vert} \int_{\cu_m} (p + \nabla \phi_{m}(\mu, \cdot)) \cdot \a(\mu+\mu_\rho, \cdot) (p + \nabla \phi_{m}(\mu, \cdot)) \, \d (\mu + \mu_\rho)\Rr] \\
& \leq \E \Ll[ \frac{1}{(\rho_0 + \rho) \vert \cu_m \vert} \int_{\cu_m} (p + \nabla \phi_{m}(\mu, \cdot)) \cdot \a(\mu+\mu_\rho, \cdot) (p + \nabla \phi_{m}(\mu, \cdot)) \, \d \mu \Rr] + \frac{\rho \Lambda \vert p \vert^2}{\rho_0 + \rho} \\
& = \E \Ll[ \frac{1}{(\rho_0 + \rho) \vert \cu_m \vert} \int_{\cu_m} (p + \nabla \phi_{m, \xi}(\mu, \cdot)) \cdot (\a(\mu+\mu_\rho, \cdot) - \a(\mu, \cdot)) (p + \nabla \phi_{m}(\mu, \cdot)) \, \d \mu \Rr] \\
& \qquad + \Ll(\frac{\rho_0}{\rho_0 + \rho}\Rr) p \cdot \ab(\cu_m, \rho_0) p  + \Ll(\frac{\rho}{\rho_0 + \rho}\Rr)  \Lambda \vert p \vert^2.
\end{align*}
For the first term, we perform an expansion with respect to $\mu_\rho$, and note that $\a(\mu + \mu_\rho,\cdot) - \a(\mu) = 0$ on the event $\{\mu_\rho = 0\}$. Therefore,
\begin{align*}
& \E \Ll[ \frac{1}{(\rho_0 + \rho) \vert \cu_m \vert} \int_{\cu_m} (p + \nabla \phi_{m}(\mu, \cdot)) \cdot (\a(\mu+\mu_\rho, \cdot) - \a(\mu, \cdot)) (p + \nabla \phi_{m}(\mu, \cdot)) \, \d \mu \Rr] \\
&= e^{-\rho \vert \cu_m \vert} \sum_{k=1}^{\infty} \Ll(\frac{(\rho \vert \cu_m \vert)^k}{k!}  \frac{1}{(\rho_0 + \rho) \vert \cu_m \vert}  \Rr.\\
& \quad \Ll. \times \fint_{(\cu_m)^k} \E_{\rho_0} \Ll[\int_{\cu_m} (p + \nabla \phi_{m}(\mu, \cdot)) \cdot \Ll(\a(\mu+ \sum_{i=1}^k \delta_{x_i}, \cdot) - \a(\mu, \cdot)\Rr) (p + \nabla \phi_{m}(\mu, \cdot)) \, \d \mu \Rr] \, \d x_1 \cdots \d x_k \Rr)\\
& \leq \rho \vert \cu_m \vert \Ll( e^{-\rho \vert \cu_m \vert} \sum_{k=1}^{\infty}\frac{(\rho \vert \cu_m \vert)^{(k-1)}}{(k-1)!} \Lambda^2 \vert p \vert^2\Rr) \\
& = \rho \vert \cu_m \vert \Lambda^2 \vert p \vert^2.
\end{align*}
This gives us 
\begin{multline}\label{eq:BoundRightUp}
p \cdot \ab(\cu_m, \rho_0 + \rho) p -  p \cdot \ab(\cu_m, \rho_0) p \\
\leq  \rho \vert \cu_m \vert \Lambda^2 \vert p \vert^2 + \Ll(\frac{\rho}{\rho_0 + \rho}\Rr)  \Lambda \vert p \vert^2 - \Ll(\frac{\rho}{\rho_0 + \rho}\Rr) p \cdot \ab(\cu_m, \rho_0) p.
\end{multline}
Taking $\rho \searrow 0$ we obtain that 
\begin{align}\label{eq:RightUp}
\lim_{\rho \searrow 0} \ab(\cu_m, \rho_0 + \rho) \leq  \ab(\cu_m, \rho_0).
\end{align}

 We now establish that $\lim_{\rho \searrow 0} \ab(\cu_m, \rho_0 + \rho) = \ab(\cu_m, \rho_0)$. To this end, we drop out the part of integration against $\mu_\rho$ and obtain 
\begin{multline}\label{eq:Dropout}
p \cdot \ab(\cu_m, \rho_0 + \rho) p \\
 \geq \frac{\rho_0}{\rho_0 + \rho} \E \Ll[ \frac{1}{\rho_0 \vert \cu_m \vert} \int_{\cu_m} (p + \nabla \phi^{\rho}_{m}) \cdot \a(\mu + \mu_\rho, \cdot) (p + \nabla \phi^{\rho}_{m}) \, \d \mu\Rr].
\end{multline}
We compare this with the following minimization problem, in which we fix $\mathcal{M}_\delta(\R^d)$, $p \in \R^d$ and $U \subseteq \R^d$ a bounded domain, 
\begin{align}\label{eq:QuenchedNu}
\nu(U,p; \mu_\rho) := \inf_{v \in \cH^1_0(U)} \int \Ll( \frac{1}{\rho \vert U \vert} \int_{U} \frac{1}{2} (p + \nabla v) \cdot \a(\mu + \mu_\rho, \cdot) (p + \nabla v) \, \d \mu \Rr) \d \P_{\rho_0}(\mu).
\end{align}
This can always be seen as the problem like \eqref{eq:defNu}, but with a perturbation with a fixed point process $\mu_\rho$. We denote by $\mu \mapsto \phi_{m}(\mu; \mu_\rho) \in \cH^1_0(U)$ its minimizer, and for every fixed $\mu_\rho \in \mathcal{M}_\delta(\R^d)$, $\phi^{\rho}_{m}(\cdot + \mu_\rho)$ is a sub-minimizer for \eqref{eq:QuenchedNu}. Therefore, \eqref{eq:Dropout} gives that 
\begin{multline*}
p \cdot \ab(\cu_m, \rho_0 + \rho) p \\
\geq \frac{\rho_0}{\rho_0 + \rho} \E_{\rho_0} \Ll[ \frac{1}{\rho_0 \vert \cu_m \vert} \int_{\cu_m} (p + \nabla\phi_{m}(\mu; \mu_\rho)) \cdot \a(\mu + \mu_\rho, \cdot) (p + \nabla \phi_{m}(\mu; \mu_\rho)) \, \d \mu\Rr].
\end{multline*}
We perform an expansion with respect to $\mu_\rho$ and notice that, when $\mu_\rho = 0$ the problem \eqref{eq:QuenchedNu} is exactly the same as \eqref{eq:defNu} and ${\phi_{m}(\mu; 0) = \phi_{m}(\mu)}$, so we obtain that 
\begin{align}\label{eq:BoundRightLow}
p \cdot \ab(\cu_m, \rho_0 + \rho) p \geq \Ll(\frac{\rho_0  e^{-\rho \vert \cu_m \vert}}{\rho_0 + \rho}\Rr) p \cdot \ab(\cu_m, \rho_0) p.
\end{align}
This also concludes that 
\begin{align}\label{eq:RightLow}
\lim_{\rho \searrow 0} \ab(\cu_m, \rho_0 + \rho) \geq  \ab(\cu_m, \rho_0).
\end{align}
Combining \eqref{eq:RightLow} and \eqref{eq:RightUp} yields the right continuity of $\ab(\cu_m, \cdot)$.

We also need to verify the left continuity. We define $\rho_1 := \rho_0 + \rho$ which is fixed, then \eqref{eq:BoundRightUp} becomes 
\begin{align*}
p \cdot \ab(\cu_m, \rho_1) p -  p \cdot \ab(\cu_m, \rho_0) p 
\leq  \rho \vert \cu_m \vert \Lambda^2 \vert p \vert^2 + \Ll(\frac{\rho}{\rho_1}\Rr)  \Lambda \vert p \vert^2 - \Ll(\frac{\rho}{\rho_1}\Rr) \vert p \vert^2.
\end{align*}
We let $\rho_1 \nearrow \rho_0$ and obtain that 
\begin{align*}
\ab(\cu_m, \rho_1) \leq \lim_{\rho_0 \nearrow \rho_1}\ab(\cu_m, \rho_0).
\end{align*}
Similarly, we put fixed $\rho_1 = \rho_0 + \rho$ into \eqref{eq:BoundRightLow} and get 
\begin{align*}
p \cdot \ab(\cu_m, \rho_1) p - p \cdot \ab(\cu_m, \rho_0) p \geq \Ll(\frac{\rho_0  e^{-\rho \vert \cu_m \vert}}{\rho_1} - 1\Rr) \Lambda \vert p \vert^2, 
\end{align*}
which means that 
\begin{align*}
\ab(\cu_m, \rho_1) \geq \lim_{\rho_0 \nearrow \rho_1}\ab(\cu_m, \rho_0).
\end{align*}
These prove the left continuity of $\ab(\cu_m, \cdot)$, establishing that $\ab(\cu_m, \cdot)$ is continuous. 

\textit{Step 3: Locally uniform convergence of $c_{k,m}$.} We now turn to the proof of the locally uniform convergence of $\{c_{k,m}\}_{m \in \N}$. Let $K > 0$ and $\rho > 0$. For the case $k=1$, by~\eqref{eq:ExpansionFinite} and~\eqref{eq:ExpansionPosRho}, we find that
\begin{align*}
    \sup_{\rho_0 \in [0,K]} |c_{1,m}(\rho_0) - c_1(\rho_0)| & \leq \frac{1}{\rho} \sup_{\rho_0 \in [0,K]} |R_1(m,\rho_0,\rho)| + O(\rho) \\
    & + \frac{1}{\rho} \sup_{\rho_0 \in [0,K]} | q \cdot (\ab^{-1}(\rho_0 + \rho) - \ab_*^{-1}(\cu_m,\rho_0+\rho) )q| \\
    & + \frac{1}{\rho} \sup_{\rho_0 \in [0,K]} | q \cdot (\ab^{-1}(\rho_0) - \ab_*^{-1}(\cu_m,\rho_0) )q|.
\end{align*}
Using the statement of~Proposition~\ref{prop:SmoothFinite}, the first line on the right-hand side of the previous display is uniformly bounded by a constant $O(\rho)$ independent of $m$ and $\rho_0$, and the locally uniform convergence of $(\ab_*(\cu_m,\cdot))_{m \geq 1}$ towards $\ab$ makes the second line and third line vanish when $m$ tends to infinity. Thus, we obtain 
\begin{align*}
    \limsup_{m\to\infty} \sup_{\rho_0 \in [0,K]} |c_{1,m}(\rho_0) - c_1(\rho_0)| \leq O(\rho).
\end{align*}
Since the left hand side of the display above does not depend on $\rho$, we can let $\rho$ be arbitrarily small, which proves the locally uniform convergence of  $\{c_{1,m}\}_{m \in \N}$. For the case $k \geq 2$, the claim about $\{c_{k,m}\}_{m \in \N}$ follows in the same manner by induction.
\end{proof}

\subsection*{Acknowledgements}
CG was supported by a PhD scholarship from Ecole Polytechnique, and part of this project was developed during his visit at Fudan University. JCM was partially supported by NSF grant DMS-1954357.

\bibliographystyle{abbrv}
\bibliography{KawasakiRef}

\end{document}